\documentclass[reqno,11pt]{amsart}
\usepackage{xcolor}
\usepackage[top=2.0cm,bottom=2.0cm,left=3cm,right=3cm]{geometry}
\usepackage{amsthm,amsmath,amssymb,dsfont}
\usepackage{mathrsfs,amsfonts,functan,extarrows,mathtools}

\usepackage{cite}
\usepackage{indentfirst, latexsym, amssymb, enumerate,amsmath,graphicx}
\usepackage{relsize}
\usepackage{marginnote}
\usepackage{stmaryrd}
\usepackage{esint}
\usepackage{graphicx}
\usepackage{bm}
\usepackage{caption}
\usepackage{subfigure}
\usepackage{cite}
\usepackage{color}
\usepackage{graphicx}
\usepackage{subfigure}
\usepackage{paralist}
\usepackage{indentfirst}
\usepackage{cite}
\allowdisplaybreaks %[4]
\theoremstyle{plain}
\newtheorem{thm}{Theorem}[section]

\newtheorem{lem}[thm]{Lemma}
\newtheorem{prop}[thm]{Proposition}
\newtheorem{rem}{Remark}[section]

\newtheorem{defn}{Definition}[section]

\numberwithin{equation}{section}

\usepackage{appendix}
\usepackage{xcolor}

\newcommand{\eps}{{\varepsilon}}

\begin{document}

\title[The compressible  Navier-Stokes-transport system]
{Low Mach number limit and optimal time decay rates of the compressible  Navier-Stokes-transport system in   critical  Besov spaces}

\author[F. Li]{Fucai Li}
\address{School of Mathematics, Nanjing University, Nanjing
 210093, P. R. China}
\email{fli@nju.edu.cn}

\author[J. Ni]{Jinkai Ni$^*$}  \thanks{$^*$\! Corresponding author}
\address{School  of Mathematics, Nanjing University, Nanjing 
 210093, P. R. China}
\email{jinkaini123@gmail.com}

\author[Y. Wang]{Yuzhu Wang}
\address{School  of Mathematics, Nanjing University, Nanjing 
 210093, P. R. China}
\email{yuzhuwang@smail.nju.edu.cn}

\begin{abstract}
In this paper, we investigate the 
%low Mach number limit and optimal time decay rates for the compressible 
Navier-Stokes-Transport (NST) system in the framework of Besov spaces. This system  contains 
of a compressible Navier-Stokes system for the density and momentum  of a fluid, and  a transport equation for the potential temperature of the fluid.
In stark contrast to the well-known Navier-Stokes-Fourier (NSF) system where the temperature satisfies a parabolic type equation providing dissipative effect for the temperature and the density,  the temperature in our NST system  enjoys a transport equation which  precludes a dissipative mechanism for the density, leading to significant different  effects to the whole system. 
% from the nonlinear pressure term in the momentum equation. 
We first establish the global well-posedness of strong solutions to the compressible NST system in critical Besov spaces over $\mathbb{R}^d$ with $d \geq 2$. 
Furthermore, by introducing the Mach number  $\varepsilon > 0$, we rigorously prove the low Mach number limit as $\varepsilon \to 0$, showing that the solutions converge to that of the incompressible  inhomogeneous Navier-Stokes system. 
This singular limit holds globally in time, even for {\it ill-prepared} initial data. 
To address the challenge posed by the lack of dissipation on the density and temperature, we develop a refined energy analysis and establish optimal time decay rates for strong solutions in $\mathbb{R}^d$ with $d \geq 3$. Notably, the density remains uniformly bounded in time, displaying asymptotic behavior fundamentally distinct from that in the NSF system, where the density possesses a dissipative structure via  the momentum and temperature equations and exhibits temporal decay.

%In this paper, we investigate the low Mach number limit and optimal time decay rates for the compressible Navier–Stokes–Transport (NST) system, which couples the compressible Navier–Stokes equations with a transport equation for potential temperature. 
%In contrast to the Navier–Stokes–Fourier (NSF) system, where pressure is typically a function of density alone, in our NST system, pressure depends nonlinearly on both density and potential temperature.
%This dependence precludes a dissipative mechanism for the density, leading to significant coupling effects from nonlinear terms in the momentum equation. 
%We first establish the global well-posedness of strong solutions to the compressible NST system in critical Besov spaces over $\mathbb{R}^d$ with $d \geq 2$. 
%Furthermore, by introducing the Mach number parameter $\varepsilon > 0$, we rigorously prove the incompressible limit as $\varepsilon \to 0$, showing that solutions converge to those of the incompressible NST system. 
%This singular limit holds globally in time, even for {\it ill-prepared} initial data. 
%To address the challenge posed by the lack of density dissipation, we develop a refined energy analysis and establish optimal time decay rates for strong solutions in $\mathbb{R}^d$ with $d \geq 3$. Notably, the density remains uniformly bounded in time, displaying asymptotic behavior fundamentally distinct from that in the NSF system, where density possesses a dissipative structure and exhibits temporal decay.

\end{abstract}

\keywords{Compressible Navier-Stokes-Transport system, Global well-posedness, Low Mach number limit, Optimal time-decay rates, Critical regularity.}

\subjclass[2020]{76N10, 35Q30, 35B25, 35B40}

\maketitle

%\tableofcontents

%-----------------section one-------------------------------------------------------------
%\renewcommand{\theequation}{\thesection.\arabic{equation}} 
%\setcounter{equation}{0}
\setcounter{equation}{0}
 \indent \allowdisplaybreaks

\section{Introduction and main results}
\subsection{Introduction}

In this paper, we investigate a  fluid model
%low Mach number limit and optimal time decay rates for the compressible 
% Navier–Stokes–Transport (NST) system 
in the framework of Besov spaces. This model  contains of  a compressible Navier-Stokes system for the density and momentum  of a fluid, and  a transport equation for the potential temperature of the fluid. It  describes the evolution of the fluid  under the assumption that diabatic processes and molecular diffusion effects on potential temperature are neglected, and  is widely used in meteorology and astrophysics 
 (see \cite{FKNZ-M3AS-2016,almgren-06,klein} and references therein). Mathematically, this model in $\mathbb{R}^d ( d\geq 2)$ can be formulated as follows (\!\!\cite{LS-JMFM-2022}):
\begin{equation}\label{A1}   
\left\{  
\begin{aligned}  
& \partial_t \rho + {\rm div}\,(\rho u)=0,\\  
& \partial_t (\rho u)+{\rm div}\,(\rho u\otimes u)-\mu\Delta u-(\mu+\lambda)\nabla {\rm div}\,u+\nabla P(\rho,\theta)=0, \\  
& \partial_t(\rho\theta)+{\rm div}\,(\rho\theta u)=0,  
\end{aligned}  
\right.  
\end{equation}  
with the initial data
\begin{align}\label{A1-1}
(\rho(t,x),u(t,x),\theta(t,x))|_{t=0}=(\rho_0(x),u_0(x),\theta_0(x)),\quad x\in\mathbb{R}^d.     
\end{align}
Here, the unknown functions $\rho = \rho(t,x) > 0$, $u = (u_1, u_2, \dots, u_d) \in \mathbb{R}^d$, and $\theta = \theta(t,x) > 0$ represent the density, velocity field, and potential temperature of the fluid, respectively, where $t > 0$ denotes time and $x = (x_1, x_2, \dots, x_d) \in \mathbb{R}^d$ denotes the spatial position. The scalar pressure $P = P(\rho, \theta)$ is given by  
\begin{align*}
P(\rho, \theta) = A(\rho \theta)^\gamma, \quad A > 0, \quad \gamma > 1.
\end{align*}
The parameters $\lambda$ and $\mu > 0$ are the viscosity coefficients satisfying the physical condition $  2\mu+\lambda > 0$.  In the earlier literature, the system \eqref{A1} is usually referred as the \emph{Navier-Stokes equations with entropy transport}, for example, see  \cite{MM-JMFM-2015,M-JDE-2016}. Due to the link between potential temperature and entropy, to emphasize the effect of  potential temperature, the system \eqref{A1} is 
renamed as \emph{Navier-Stokes equations with potential temperature transport} in \cite{LS-JMFM-2022}. In current paper, we call the system  \eqref{A1} as 
 \emph{Navier-Stokes-Transport} (\emph{NST}) \emph{system} for  presentation brevity.

Due to its physical importance, complexity,   and mathematical challenges, there have been   a lot of studies on the NST system \eqref{A1}, see, for example,  \cite{MM-JMFM-2015,lions1998,M-JDE-2016,LS-JMFM-2022,LS-JMFM-2023,
LS-JDE-2025,BDGL-SAP-2022,FKNZ-M3AS-2016,ZLZ-CMS-2023,FKNZ-M3AS-2016,almgren-06,klein} by physicists and mathematicians. Particularly, Mich\'{a}lek \cite{MM-JMFM-2015} established a stability result for solutions under mild assumptions on the sequence of densities for $\gamma>3/2$ in $\mathbb R^3$, thus differing from Lions' work \cite{lions1998}, which addressed the case $\gamma\geq 9/5$. Maltese  et al. \cite{M-JDE-2016} investigated global weak solutions to the NST system under the conditions $\gamma> 3/2$ in $\mathbb R^3$ or $\gamma>1$ in $\mathbb R^2$. 
Later, the concept of dissipative measure-valued (DMV) solutions for the system \eqref{A1} was introduced in \cite{LS-JMFM-2022,LS-JMFM-2023}, where the existence and DMV-strong uniqueness of such solutions were rigorously established.   Recently, the existence and uniqueness of local-in-time strong solutions, together with a blow-up criterion for these solutions, were proven in \cite{LS-JDE-2025}. Fan and Huang \cite{fanhuang}  obtained global strong solutions and asymptotic behavior for arbitrarily large initial data in  $\mathbb{R}^2$ for special viscous coefficients initial introduced by Vaigant and  Kazhikhov  \cite{vk}.  In \cite{ZLZ-CMS-2023}, Zhai, Li and Zhou obtained the   global strong solutions to \eqref{A1} in the framework of Besov spaces   when the   initial data is sufficient small and no   decay rate estimate of these solutions was given.

There are also a few results on the singular limits  to the compressible NST system \eqref{A1} in the framework of weak solutions.   Bresch et al. \cite{BDGL-SAP-2022}   gave a formal asymptotic analysis  on 
the low  Mach number limit  of  the system \eqref{A1}   in a periodic domain with ill-prepared initial data for $\gamma>9/5$.  
Feireisl et al. \cite{FKNZ-M3AS-2016} considered the coupled low Mach number/Froude number limits to \eqref{A1} with well-prepared initial data on a flat torus and  ill-prepared initial data in an infinite slab for $\gamma>3/2$.  
Until now, the singular limit to the physically relevant case $1<\gamma\leq {3}/{2}$ remains an open problem. 

Based on the aforementioned results on the system \eqref{A1}, the objectives  of this paper are two folds: (1) We establish the global well-posedness and optimal decay rates of the strong solutions to the system \eqref{A1} in the framework of \emph{critical Besov spaces}. 
(2) We investigate the low Mach number limit of \eqref{A1} in the Besov space with ill-prepared initial data for $\gamma>1$.

By criticality, we mean that the scaling transformation preserving the norms of the function space in which the solution lies leaves \eqref{A1} invariant. This central idea, originating from the work of Fujita and Kato \cite{FK-1964-ARMA}, asserts that the “optimal" function spaces for the well-posedness of system \eqref{A1} must respect the following scaling invariance:
\begin{align*}
\rho_{\nu}(t,x) = \rho(\nu^2 t, \nu x), \quad u_{\nu}(t,x) = \nu u(\nu^2 t, \nu x), \quad \theta_{\nu}(t,x) = \theta(\nu^2 t, \nu x),
\end{align*}
along with the corresponding scaling of the initial data:
\begin{align*}
\rho_0(x) \leadsto  \rho_0(\nu x), \quad u_0(x) \leadsto  \nu u_0(\nu x), \quad \theta_0(x) \leadsto  \nu \theta_0(\nu x),
\end{align*}
for any $\nu > 0$. It is then evident that $\dot B_{2,1}^{\frac{d}{2}} \times \dot B_{2,1}^{\frac{d}{2}-1} \times \dot B_{2,1}^{\frac{d}{2}}$ constitutes a critical space for the system \eqref{A1}.

%Our is to establish the low Mach number limit for strong solutions to system \eqref{A2}. Even though we concentrate on strong solutions, we take into account the general case $\gamma >1$. 
%One of the main objectives is to establish a rigorous justification for the convergence of the compressible Navier-Stokes-Transport equations \eqref{A2} to the incompressible Navier-Stokes-Transport equations \eqref{A5}. Since no decay results are currently available for the compressible Navier-Stokes-Transport equations \eqref{A1}, another primary goal of this work is to establish  
%the optimal decay rate for solutions to the system \eqref{A3}-\eqref{A3-1}.
%This paper discusses strong solutions in the so-called critical Besov spaces. 

In order to state our results, we fist give some analyses on the system \eqref{A1}. To address the difficulties from  the  strong nonlinearity of the pressure in the momentum equation\eqref{A1}$_2$, following \cite{ZLZ-CMS-2023}, we combine the equations for $\rho$ and $\theta$ and reformulate \eqref{A1} in terms of $\rho$, $u$  and $P$ as %. This yields the following: 
\begin{equation}\label{A2}   
\left\{  
\begin{aligned}  
& \partial_t \rho + {\rm div}\,(\rho u)=0,\\  
& \partial_t (\rho u)+{\rm div}\,(\rho u\otimes u)-\mu\Delta u-(\mu+\lambda)\nabla {\rm div}\,u+\nabla P  =0, \\  
& \partial_t P+  u\cdot\nabla P+\gamma P{\rm div}\,u=0.
\end{aligned}  
\right.  
\end{equation}  
It is evident that the system \eqref{A2} constitutes a hyperbolic-parabolic system, where \eqref{A2}$_1$ and \eqref{A2}$_3$ are transport equations, and dissipation arises  only from $u$. We consider the initial value problem for \eqref{A2} in $\mathbb{R}^d$ with $d\geq 2$, subject to the initial data  
\begin{align}\label{A2-1}
(\rho, u, P)(0,x) = (\rho_0, u_0, P_0)(x) \to (\rho_{\infty}, 0, P_{\infty}) \quad \text{as} \quad |x| \to \infty,
\end{align}
where $\rho_{\infty} > 0$ and $P_{\infty} > 0$ are positive constants.
For the sake of simplicity, we set $\rho_{\infty} = P_{\infty}= 1$.
By setting $a = \rho - 1$ and $z = P - 1$, we  rewrite system \eqref{A2} as
\begin{equation}\label{A3}   
\left\{  
\begin{aligned}  
& \partial_t a +{\rm div}\,u=- {\rm div}\,(a u) ,\\  
& \partial_t u-\mu\Delta u-(\mu+\lambda)\nabla {\rm div}\,u+ \nabla z \\
&\quad=-u\cdot\nabla u-f(a)\nabla z+f(a)(\mu\Delta u+(\mu+\lambda)\nabla {\rm div}\,u), \\  
& \partial_t z +\gamma{\rm div}\,u=- u\cdot\nabla z-\gamma z{\rm div}\,u, 
\end{aligned}  
\right.  
\end{equation}  
with the initial data 
\begin{align}\label{A3-1}
(a(t,x),u(t,x),z(t,x))|_{t=0}=(a_0(x),u_0(x),z_0(x)),\quad x\in\mathbb{R}^d,     
\end{align}
where $f(a)$ is defined by
\begin{align*}
f(a)=-\frac{a}{1+a}.    
\end{align*}

Let $\omega=\gamma a-z$. From the equations \eqref{A3}$_1$ and \eqref{A3}$_3$, it follows that  
\begin{align}\label{A3-2}
\partial_t \omega + u\cdot\nabla\omega+  \omega{\rm div}\,u-(\gamma-1)z{\rm div}\,u= 0,
\end{align}
which yields a new transport equation for the derived quantity $\omega$. Although $\omega$ does not possess a dissipative structure, it provides a crucial link between the unknowns $a$ and $z$, thereby overcoming the difficulty introduced by the additional linear term $\mathrm{div}\,u$ in \eqref{A3}$_1$ and \eqref{A3}$_3$.

In this paper, we shall also investigate the asymptotic behavior of \eqref{A2} in the low Mach number regime. Let $\varepsilon\in(0,1]$ be the Mach number, a dimensionless parameter. 
Inspired by \cite[Section 5.1]{Danchin-2018} and \cite{DH-MATHANN-2016}, we consider the following change of variables:
\begin{align}\label{G4.1}
(a,u,z)(t,x) := \varepsilon (a^\varepsilon,u^\varepsilon,z^\varepsilon)(\varepsilon^2 t,\varepsilon x),
\end{align}
along with the corresponding transformation of the initial data:
\begin{align}\label{G4.2}
(a_0,u_0,z_0)(x) := \varepsilon (a^\varepsilon_0,u^\varepsilon_0,z^\varepsilon_0)(\varepsilon x).
\end{align}
% where  $a = \rho - 1$ and $z = P - 1$.
 {In addition, $\mu=\varepsilon\bar{\mu}$, $\lambda=\varepsilon\bar{\lambda}$, where $\bar{\mu}$ and $\bar{\lambda}$ are constants independent of $\varepsilon$.}
It is straightforward to verify that $(a^\varepsilon,u^\varepsilon,z^\varepsilon)$ satisfies the following scaled compressible NST system:
%Following the framework established in \cite{PM-JMPA-1998}, we introduce the following scaling transformations:
%\begin{align*}
%\rho(t,x) = \rho^{\varepsilon}(\varepsilon t,x),\quad u(t,x) = \varepsilon u^{\varepsilon}(\varepsilon t,x),\quad P(t,x) = P^{\varepsilon}(\varepsilon t,x),\quad \mu = \varepsilon\bar\mu,\quad \lambda = \varepsilon\bar\lambda,
%\end{align*}
%where $\bar\mu$ and $\bar\lambda$ are constants independent of $\varepsilon$. Thus, the system \eqref{A2} becomes
%\begin{equation}\label{A3-3}   
%\left\{  
%\begin{aligned}  
%&  \partial_t  \rho^\varepsilon  +{\rm div}\,(\rho^{\varepsilon} u^{\varepsilon})=0,\\  
%& \partial_t\big(\rho^{\varepsilon}   u^{\varepsilon}\big) +{\rm div}\,(\rho^{\varepsilon}   u^\varepsilon\otimes u^{\varepsilon})+\frac{\nabla   P^{\varepsilon}}{\varepsilon^2 } \\%
%%+\frac{\nabla P(1+\varepsilon z^{\varepsilon})}{\varepsilon^2} \\
%&\quad= \bar\mu \Delta u^{\varepsilon}+(\bar\mu+\bar\lambda)\nabla {\rm div}\,u^{\varepsilon},\\
%&  \partial_t  P^\varepsilon  +  u^{\varepsilon} \cdot\nabla P^\varepsilon
%+ \gamma P^\varepsilon{\rm div}\,u^\varepsilon=0.
%\end{aligned}  
%\right.  
%\end{equation}  
%By expressing $\rho^\varepsilon(t,x) = 1 + \varepsilon a^{\varepsilon}(t\varepsilon,x)$ and $P^\varepsilon(t,x) = 1 + \varepsilon z^{\varepsilon}(t\varepsilon,x)$, the system \eqref{A3-3} can be reformulated as
%
\begin{equation}\label{A4}   
\left\{  
\begin{aligned}  
&  \partial_t  a^\varepsilon  + \frac{{\rm div}\, u^{\varepsilon}}{\varepsilon} + {\rm div}\,(a^{\varepsilon} u^{\varepsilon})=0,\\  
& \partial_t\big((1+\varepsilon a^{\varepsilon} ) u^{\varepsilon}\big) +{\rm div}\,((1+\varepsilon a^{\varepsilon} ) u^\varepsilon\otimes u^{\varepsilon})+\frac{\nabla   z^{\varepsilon}}{\varepsilon } \\%
%+\frac{\nabla P(1+\varepsilon z^{\varepsilon})}{\varepsilon^2} \\
&\quad= \bar\mu \Delta u^{\varepsilon}+(\bar\mu+\bar\lambda)\nabla {\rm div}\,u^{\varepsilon},\\
&  \partial_t  z^\varepsilon  + \frac{\gamma{\rm div}\, u^{\varepsilon}}{\varepsilon} +  u^\varepsilon\cdot\nabla z^\varepsilon+ \gamma z^{\varepsilon}{\rm div}\,u^\varepsilon=0,\\
&(a^\varepsilon,u^\varepsilon,z^\varepsilon)|_{t=0} := (a^\varepsilon_0,u^\varepsilon_0,z^\varepsilon_0).
\end{aligned}  
\right.  
\end{equation}  
Assume that $a_0^\varepsilon=\frac{\rho_0^\varepsilon-1}{\varepsilon}$, $u_0^\varepsilon$ and $z_0^\varepsilon=\frac{P_0^\varepsilon-1}{\varepsilon}$ are uniformly bounded with respect to $\varepsilon$ in some suitable sense. Passing to the limits $(a^\varepsilon,u^\varepsilon,z^\varepsilon)$  as $\varepsilon \rightarrow 0$
and  assuming that $(\varrho,v, 0)$ is the limit, then 
it is readily to check that  $(\varrho,v)$ satisfies the following incompressible inhomogeneous Navier-Stokes  system:
\begin{equation}\label{A5}   
\left\{  
\begin{aligned}  
&\partial_t\varrho+{\rm  div}(\varrho v)=0,\\
&  \partial_t v+v\cdot\nabla v-\bar\mu \Delta v+\nabla \Pi=0,\\
&{\rm div}\,v=0,
\end{aligned}  
\right.  
\end{equation} 
with the initial data
\begin{align}\label{A5-1}
(\varrho(0,x),v(0,x))=(\varrho_0(x),v_0(x))\rightarrow0,\quad \text{as}\quad |x|\rightarrow\infty.  
\end{align}
This process corresponds to the well-known  low Mach number limit, which is both a mathematically challenging and physically significant problem. Formally, due to the appearance of the singular terms $\frac{\nabla z^\varepsilon}{\eps}$ and $\frac{{\rm div} u^\varepsilon}{\eps}$ in \eqref{A4}, it is difficult to derive uniform estimates with respect to the Mach number $\varepsilon$ and to rigorously establish convergence to the incompressible 
inhomogeneous Navier-Stokes equations \eqref{A5}.

%There have been extensive studies on the mathematical justification of the low Mach limit in the framework of Sobolev spaces, which was initiated by Ebin\cite{Ebin}, Klainerman-Majda\cite{KM-1982,KM-2010} for local strong solutions of compressible Navier-Stokes or Euler equations in the whole space with well-prepared initial data. These works were extended by several authors in different settings concerninig isentropic case in the last few decades, for example, previous studies\cite{asa,Ukai1986,ig,isa,isb,Se-00} and references therein. In fact, fluid dynamics equations are often related to temperature or entropy. It is then very natural and interesting to study the low Mach number limit of non-isentropic equations. In the non-isentropic case, the low Mach limit is much more complicated due to the complex structure of the coupled system. For the non-isentropic Euler equations and ill-prepared initial data, the breakthrough was made by Métivier and Schochet\cite{MetivierSchochet2001}, where the methods of pseudodifferential operators and wave-package transform were applied to the low Mach limit. For further development, one can motion the works\cite{AlazardADE2005,AlazardARMA2006,FN-ARMA-2007,FN-CMP-2013,jiangouJMPA,juouJMPA} and references therein.  

Let's mention that there have been extensive studies on the mathematical justification of the low Mach number limit  in the framework of Sobolev spaces, initiated by Ebin \cite{Ebin} and Klainerman-Majda \cite{KM-1982,KM-2010} for local strong solutions of the compressible Navier-Stokes and Euler equations in the whole space  or a torus with well-prepared initial data. 
These foundational results were subsequently extended by numerous authors in various settings, particularly for the isentropic case, over the past few decades; see, for example, \cite{asa,Ukai1986,isa} and the references cited therein. 
However, fluid dynamics equations are inherently coupled with thermodynamic variables such as temperature or entropy. Therefore, it is both natural and significant to investigate the low Mach number limit for non-isentropic fluid systems. In this context, the limit process becomes substantially more complex due to the intricate coupling among the  velocity, pressure, and thermal fields.
A major breakthrough on low Mach number limit for the non-isentropic Euler equations with ill-prepared initial data was achieved by Métivier and Schochet \cite{MetivierSchochet2001}, who achieved this goal by employing advanced tools such as pseudo-differential operators and wave packet transforms. Further developments in this direction can be found,
for example,  in \cite{ AlazardARMA2006,FN-ARMA-2007,FN-CMP-2013,jiangouJMPA} and the cited references therein.
%related works.
%In fact, in the study of slightly commpressible fluids, finding a good functional framework is one of the key points, which induces to look for a better space in the process of incompressible limit. As regards the study of the low Mach number limit for the Navier-Stokes equations in the framework of critical space, one can refer to \cite{Dr-2002-AJM} for the case of periodic boundary conditions and \cite{Dr-2002-ASENS} for the case of the whole space. Very recently, Fujii\cite{FM-2024-MA} showed the low Mach number limit of the Navier-Stokes equations in the scaling critical Besov space, where the author pointed out the method to obtain the convergence dose not require different arguments for different spatial dimensions. 

 In the study of slightly compressible fluids, identifying an appropriate functional framework is a crucial aspect, as it naturally leads to the search for more suitable function spaces when considering the low Mach number limit. Regarding the low Mach number limit for the isentropic Navier-Stokes equations  in  critical Besov spaces, one may refer to \cite{Dr-2002-AJM} for the case of periodic boundary conditions and to \cite{Danchin-2002} for the whole space. See also \cite{DH-MATHANN-2016} on non-isentropic fluid equations. Very recently, Fujii \cite{FM-2024-MA} established the low Mach number limit for the isentropic Navier-Stokes equations in a scaling-critical Besov space, where it was emphasized that the method used to prove the convergence does not require distinct arguments depending on the spatial dimension.

%Let us now turn to more specifically on the study of the Navier-Stokes-Transport equations, which is more related to the interest of the current paper. Mich\'{a}lek\cite{MM-JMFM-2015} showed a kind of stability result for solutions under mild assumptions on the sequence of densities for any $\gamma>3/2, d=3$, which is different from the work by Lions\cite{lions1998} for the case $\gamma\geq 9/5$. Maltese\cite{M-JDE-2016} studied the global weak solutions to the Navier-Stokes-Transport system under the assumption $\gamma\geq 9/5, d=3$ or $\gamma>1, d=2$. Recently, Zhai, Li and Zhou\cite{ZLZ-CMS-2023} studied the global strong solutions to the Cauchy problem of the Navier-Stokes-Transport equations in the Besov spaces. Lukáčová-Medvid'ová and Sch\"{o}mer\cite{LS-JMFM-2022,LS-JMFM-2023} introduced the concept of dissipative measure-valued(DMV) solutions for the Navier-Stokes-Transport equations and proved the existence and DMV-strong uniqueness of such solutions. Later, they\cite{LS-JDE-2025} proved the existence and uniqueness of local-in-time strong solutions and a blow-up criterion for the strong solution to the Navier-Stokes-Transport equations.

\subsection{Main results}

Now we are in a position to state our results.  The notations in the following results shall be  presented  in Section 2 below. 
We first focus on the well-posedness. 

\begin{thm}[Local well-posedness]\label{T1.1}
 Let $d\geq 2$. Assume that the initial data $(a_0,u_0,z_0)$ satisfy
\begin{align}\label{TA.1}
a_0\in \dot B_{2,1}^{\frac{d}{2}},\quad \inf_{x\in\mathbb R^d} (1+a_0)(x)>0, \quad u_0\in     \dot B_{2,1}^{\frac{d}{2}-1},\quad z_0\in\dot B_{2,1}^{\frac{d}{2} }.
\end{align}
Then, there exists a time $T>0$, such that the Cauchy problem \eqref{A3}--\eqref{A3-1} admits a unique strong solution $(a,u,z)$ satisfying for $t\in[0,T)$ that 
\begin{equation}\label{TA.2}\left\{
\begin{aligned}
&    a\in \mathcal{C}([0,T); \dot B_{2,1}^{\frac{d}{2}}),\quad \inf_{x\in\mathbb R^d} (1+a )(x,t)>0,      \\
&  u \in    \mathcal{C}([0,T); \dot B_{2,1}^{\frac{d}{2}-1})\cap L^1(0,T;\dot B_{2,1}^{\frac{d}{2}+1}),\\
& z\in \mathcal{C}([0,T); \dot B_{2,1}^{\frac{d}{2}}).
\end{aligned}\right.
\end{equation}
\end{thm}

\begin{rem}
Noticing that the system \eqref{A3} has a structure similar to that of the compressible isentropic Navier-Stokes equations, we can  treat the variables $a$ and $z$ collectively and exploit the smallness of $T$. By adapting the approach employed in \cite{Danchin-00-IM} and
\cite[Section 5]{Haspot-2011-JDE} on the  isentropic Navier-Stokes equations, the local well-posedness can be established directly. For brevity, we omit the proof here. 
\end{rem}

\begin{thm}[Global well-posedness]\label{T1.2}
 Let $d\geq 2$. There exists a positive constant $\delta_0>0$ such that if the initial data $(a_0,u_0,z_0)$ satisfy 
$a_0,z_0\in \dot B_{2,1}^{\frac{d}{2}-1}\cap  \dot B_{2,1}^{\frac{d}{2}}$, $u_0\in \dot B_{2,1}^{\frac{d}{2}-1}$ and 
\begin{align}\label{TB.1}
\|(a_0,z_0)\|_{\dot B_{2,1}^{\frac{d}{2}-1}\cap  \dot B_{2,1}^{\frac{d}{2}}}+   \|u_0\|_{\dot B_{2,1}^{\frac{d}{2}-1}} \leq \delta_0,
\end{align} 
then the   Cauchy problem \eqref{A3}--\eqref{A3-1} admits a unique global strong solution 
$(a,u,z)$, which satisfies
\begin{equation}\label{TB.2}\left\{
\begin{aligned}
&    a\in \mathcal{C}_b(\mathbb R^+; \dot B_{2,1}^{\frac{d}{2}-1}\cap B_{2,1}^{\frac{d}{2}}) ,      \\
&  u \in    \mathcal{C}_b(\mathbb R^+;  \dot B_{2,1}^{\frac{d}{2}-1})\cap L^1(\mathbb R^+;\dot B_{2,1}^{\frac{d}{2}+1}),\\
& z\in \mathcal{C}_b(\mathbb R^+; \dot B_{2,1}^{\frac{d}{2}-1}\cap B_{2,1}^{\frac{d}{2}}) \cap L^1 (\mathbb R^+;\dot B_{2,1}^{\frac{d}{2}+1}+\dot B_{2,1}^{\frac{d}{2}}),\\
&  \gamma a-z\in \mathcal{C}_b(\mathbb R^+; \dot B_{2,1}^{\frac{d}{2}-1}\cap B_{2,1}^{\frac{d}{2}}) ,
\end{aligned}\right.
\end{equation}
Moreover, there exists a positive constant $C_1$ independent of time $t$ 
such that, for any $t>0$,
\begin{align}\label{TB.3}
\|\gamma a-z\|_{\widetilde L_t^\infty( \dot B_{2,1}^{\frac{d}{2}-1} )}\leq&\,  C_1 \Big(\|(a_0,z_0,u_0)\|_{\dot B_{2,1}^{\frac{d}{2}-1} }+ \|z_0\|_{ \dot B_{2,1}^{\frac{d}{2}}}\Big),\nonumber\\  
 \|\gamma a-z\|_{\widetilde L_t^\infty(  \dot B_{2,1}^{\frac{d}{2}})}\leq&\,  C_1\Big(\|(u_0,z_0)\|_{\dot B_{2,1}^{\frac{d}{2}-1}}+  \|(a_0,z_0)\|_{ \dot B_{2,1}^{\frac{d}{2}}}\Big) ,      
\end{align}
and
\begin{align}\label{TB.4}
& \|(a,z)\|_{\widetilde L_t^\infty (\dot B_{2,1}^{\frac{d}{2}-1}\cap B_{2,1}^{\frac{d}{2}})}+\|u\|_{\widetilde L_t^\infty(\dot B_{2,1}^{\frac{d}{2}-1})} +\|u\|_{L_t^1(\dot B_{2,1}^{\frac{d}{2}+1})}+\|z\|_{L_t^1(\dot B_{2,1}^{\frac{d}{2}+1}+\dot B_{2,1}^{\frac{d}{2}})} \nonumber\\
&\quad\leq C_1 \Big(\|(a_0,z_0)\|_{\dot B_{2,1}^{\frac{d}{2}-1}\cap  \dot B_{2,1}^{\frac{d}{2}}}+\|u_0\|_{\dot B_{2,1}^{\frac{d}{2}-1}}\Big).
\end{align}
\end{thm}

\begin{rem} 
Zhai et al. \cite{ZLZ-CMS-2023} have already established these global strong solutions to the system \eqref{A1}, here we use a different method.  In addition, our   assumption  on the initial data $(a_0,u_0, z_0)$ is different from  that used in \cite{ZLZ-CMS-2023}. More importantly, we shall give the optimal time decay rates of the strong solutions which are absent in \cite{ZLZ-CMS-2023}.
\end {rem}

\begin{rem}
In stark contrast to \cite{DX-ARMA-2017,CMZ-CPAM-2010} on the isentropic Navier-Stokes equations, where all variables  exhibit dissipative structures, our system \eqref{A3} features the unknowns $a$ and $z$ \emph{(}hence $\omega$ due to the relation  $\omega=\gamma a-z$\emph{)} which lack such structures. Consequently, when handling the nonlinear term $\omega{\rm div}\,u$ in \eqref{A3-2}, the only viable approach is to exploit the dissipative structure inherent in $u$. To this end, we consider Bony's decomposition:
\begin{align*}
(\omega u)^\ell = (\dot T_\omega u)^\ell + \big(R(\omega,u)\big)^\ell + \big(\dot T_{u} \omega^\ell\big)^\ell + \big(\dot T_u\omega^h\big)^\ell.
\end{align*}
The term $\big\|\big(\dot T_{u} \omega^\ell\big)^\ell\big\|_{L_t^1(\dot B_{2,1}^\frac{d}{2})}$ can be bounded by $\|u\|_{\widetilde L_t^\infty(\dot B_{p,1}^{\frac{d}{p}-1})}\|\omega^\ell\|_{L_t^1(\dot B_{2,1}^{\frac{d}{2}+1})}$ \emph{(}see  {(115)} in \cite{Danchin-2018}\emph{)}. However, due to the absence of a dissipative structure for $\omega$, controlling this term relies entirely on the dissipation of $u$, which proves insufficient. As a result, this term becomes uncontrollable, implying that the solution $(a,u,z)$ may fail to be well-defined in the $L^p$-type Besov framework.
\end{rem}

\begin{rem}
The new mode $\omega=\gamma a-z$ in the equation \eqref{A3-2} allows for a more accurate estimate as shown in \eqref{TB.3}, compared to \eqref{TB.4}. Although $\omega$ lacks a dissipative structure, combining it   with the equation for $z$   derives a uniform estimate for the density $a$.
\end{rem}

Next, we focus on the low Mach number limit of the NST system \eqref{A3}.

\begin{thm}[Low Mach number limit]\label{T1.4} 
Let $d\geq 2$ and let $\varepsilon\in (0,1)$ be a constant. There exists a positive constant $\delta_1>0$ such that if the initial data $(a_0^\varepsilon,u_0^\varepsilon,z_0^\varepsilon)$ satisfies $a_0^\varepsilon,z_0^\varepsilon\in \dot B_{2,1}^{\frac{d}{2}-1}\cap \dot B_{2,1}^{\frac{d}{2}}$, $u_0^\varepsilon\in \dot B_{2,1}^{\frac{d}{2}-1}$ and
\begin{align}\label{TD.1}
\|(a_0^\varepsilon,z^\varepsilon_0)\|_{\dot B_{2,1}^{\frac{d}{2}-1} }+   \|u_0^\varepsilon\|_{\dot B_{2,1}^{\frac{d}{2}-1}}+ \varepsilon\|(a_0^\varepsilon,z^\varepsilon_0)\|_{\dot B_{2,1}^{\frac{d}{2} } }\leq \delta_1,
\end{align} 
then the NST system \eqref{A4} with 
the initial data $(a_0^\varepsilon, u_0^\varepsilon, z_0^\varepsilon)$ admits a unique strong solution $(a^\varepsilon, u^\varepsilon, z^\varepsilon)$ satisfying
\begin{align}\label{TD.2}
\|\gamma a^\varepsilon-z^\varepsilon\|_{\widetilde L_t^\infty( \dot B_{2,1}^{\frac{d}{2}-1})}\leq&\,  C_2 \|(a^\varepsilon_0,u_0^\varepsilon,z^\varepsilon_0)\|_{\dot B_{2,1}^{\frac{d}{2}-1}}+C_2\varepsilon\|z_0^\varepsilon\|_{\dot B_{2,1}^{\frac{d}{2}}}  ,\nonumber \\  
  \varepsilon\|\gamma a^\varepsilon-z^\varepsilon\|_{\widetilde L_t^\infty( \dot B_{2,1}^{\frac{d}{2} })}\leq&\,  C_2\|(u^\varepsilon_0,z_0^\varepsilon)\|_{\dot B_{2,1}^{\frac{d}{2}-1}}+ C_2 \varepsilon \|(a^\varepsilon_0,z^\varepsilon_0)\|_{\dot B_{2,1}^{\frac{d}{2} }}  , 
\end{align}
and
\begin{align}\label{TD.3}
 &\|(a^\varepsilon,z^\varepsilon )\|_{\widetilde L_t^\infty (\dot B_{2,1}^{\frac{d}{2}-1} )}+\varepsilon\|(a^\varepsilon,z^\varepsilon )\|_{\widetilde L_t^\infty (\dot B_{2,1}^{\frac{d}{2}} )} +\|u^\varepsilon\|_{\widetilde L_t^\infty(\dot B_{2,1}^{\frac{d}{2}-1})} \nonumber\\
& \quad+\|u^\varepsilon\|_{L_t^1(\dot B_{2,1}^{\frac{d}{2}+1})}+\|z^\varepsilon\|_{L_t^1(\dot B_{2,1}^{\frac{d}{2}+1} )}^{\ell,\varepsilon}+\frac{1}{\varepsilon}\|z^\varepsilon\|_{L_t^1(\dot B_{2,1}^{\frac{d}{2}} )}^{h,\varepsilon}  \nonumber\\
&\quad\quad\leq C_2 \Big(\|(a_0^\varepsilon,z^\varepsilon_0)\|_{\dot B_{2,1}^{\frac{d}{2}-1} }+\|u^\varepsilon_0\|_{\dot B_{2,1}^{\frac{d}{2}-1}}+\varepsilon\|(a_0^\varepsilon,z^\varepsilon_0)\|_{\dot B_{2,1}^{\frac{d}{2}} }\Big),   
\end{align}
for any $t>0$, where $C_2>0$ is independent of $t$ and $\varepsilon$.
In addition,   it holds that
\begin{equation}\label{TD.4}\left\{
\begin{aligned}
&   \|(\mathbb{Q}u^\varepsilon,z^\varepsilon)\|_{\widetilde L_t^2(\dot B_{p,1}^{\frac{5}{2p}-\frac{1}{4}})}\leq C_2\varepsilon^{\frac{1}{4}-\frac{1}{2p}} ,\quad\text{for}\quad p\in(2,\infty),  \quad\quad d=2;\\
&\|(\mathbb{Q}u^\varepsilon,z^\varepsilon)\|_{\widetilde L_t^2(\dot B_{p,1}^{\frac{d+1}{p}-\frac{1}{2}})}\leq C_2\varepsilon^{\frac{1}{2}-\frac{1}{p}},\,\,\,\,\,\,\text{for}\quad p\in(2,6)  ,  \,\,\, \,\,\,\,\, \quad d\geq3.
\end{aligned}\right.
\end{equation}
Assume further that there exists $v_0$ such that as $\varepsilon \rightarrow 0$, $\mathbb{P} u_0^\varepsilon \rightharpoonup v_0$. Then, $\mathbb{P} u^\varepsilon$ converges to the limit $v$, and the new mode $\omega^{\varepsilon}:=\gamma a^\varepsilon - z^\varepsilon$ converges to the limit $\varrho$ in the sense of distributions, where $(\varrho, v)$ is the global solution of the incompressible inhomogeneous Navier-Stokes system \eqref{A5} with the initial data $(\varrho_0, v_0)$.

\end{thm}

\begin{rem}
The orthogonal projectors $\mathbb{P}$ and $\mathbb{Q}$ in Theorem \ref{T1.4} are defined as  
\begin{align*}
\mathbb P:={\rm Id}+\nabla(-\Delta)^{-1}\mathrm{div},\quad \mathbb{Q} := -\nabla(-\Delta)^{-1}\mathrm{div}.    
\end{align*}  
\end{rem}

\begin{rem}
The low Much number  limit result in Theorem \ref{T1.4} corresponds to the so-called {\it ill-prepared} data case, characterized by ${\rm div}\, u_0^\varepsilon = \mathcal{O}(1)$ and
  $\nabla z^\varepsilon  = \mathcal{O}(1)$,  % $\nabla P^\varepsilon  = \mathcal{O}(\varepsilon)$, 
 as opposed to the  {\it well-prepared}  case, which requires ${\rm div}\, u_0^\varepsilon = \mathcal{O}(\varepsilon)$ and 
 $\nabla z^\varepsilon = \mathcal{O}(\varepsilon )$. % $\nabla P^\varepsilon = \mathcal{O}(\varepsilon^2)$.
\end{rem}

Finally, we  establish the optimal time decay rate for the  compressible NST  system \eqref{A3} for $d\geq 3$. This optimality is evident by combining the results obtained in Theorems \ref{T1.5}--\ref{T1.7} below.

\begin{thm}[Upper-bound: the bounded condition] \label{T1.5}
Let $d\geq 3$. Under the assumptions of Theorem \ref{T1.2}, if the initial data $u_0$ and $z_0$ additionally satisfy  
\begin{align}\label{TE.1}
u_0^\ell,z_0^\ell\in   \dot B_{2,\infty}^{\sigma_0} \quad \text{with}\quad \sigma_0\in \Big[-\frac{d}{2},\frac{d}{2}-2\Big).
\end{align}
Then for all $t \geq 1$, there exists a universal constant $C_3 > 0$ such that  
\begin{align}\label{TE.2}
\|(u,z)(t)\|_{\dot B_{2,1}^\sigma}\leq&\, C_3 \delta_{*} (1+t)^{-\frac{1}{2}(\sigma-\sigma_0)}, \\  \label{TE.3}
\|u(t)\|^h_{\dot B_{2,1}^{\frac{d}{2}+1}}+\|z(t)\|^h_{\dot B_{2,1}^{\frac{d}{2}}}\leq&\, C_3 \delta_{*} (1+t)^{-\frac{1}{2}(\frac{d}{2}-1-\sigma_0)}, 
\end{align}
for all $\sigma\in\big(\sigma_0,\frac{d}{2}-1 \big]$. Here $\delta_{*}$ is defined by
\begin{align}\label{TE.4}
\delta_{*}:=\|(u_0^\ell,z_0^\ell)\|_{\dot B_{2,\infty}^{\sigma_0}} +\|u_0^h\|_{\dot B_{2,1}^{\frac{d}{2}-1}}+\|z_0^h\|_{\dot B_{2,1}^{\frac{d}{2} }} +\|a_0\|_{\dot B_{2,1}^{\frac{d}{2}-1}\cap \dot B_{2,1}^{\frac{d}{2} }}.  
\end{align}
\end{thm}

\begin{rem}
Since the density $a$ lacks a dissipative structure, the nonlinear terms $f(a)\nabla z$, $f(a)\Delta u$, and $f(a)\nabla {\rm div}\, u$ in the $(a,z)$ system pose significant technical challenges, resulting in the restriction condition $\sigma_0\in\big[-\frac{d}{2},\frac{d}{2}-2\big)$ for the case when $d\geq 3$. 
This contrasts with the cases in \cite{XX-JDE-2021,LS-SIMA-2023} on the isentropic 
Navier-Stokes equations, where the density equation possesses a dissipative structure which allows $\sigma_0\in  \big[-\frac{d}{2}, \frac{d}{2} - 1\big)$, thereby accommodating the case $d = 2$.
\end{rem}

\begin{rem}
For the compressible isentropic Navier-Stokes equations  analyzed in \cite{Danchin-00-IM} 
it exhibits a hyperbolic-parabolic structure in terms of the variables $(a,u)$,  endows  
the density $a$ with dissipative properties. However, in our system \eqref{A3},   the dissipative mechanism
of $u$ does not effectively act on the density $a$, and 
the classical perturbation theory developed in \cite{MN-JMKU-1980} is not applicable.
% to our system \eqref{A3}. 
 This limitation fundamentally explains the absence of a long time decay estimate for the variable $a$. A similar situation arises in the pressureless Navier-Stokes system studied in \cite{Dr-2024-PMP}.
\end{rem}

\begin{rem}
Although \eqref{A3}$_2$--\eqref{A3}$_3$ constitute a hyperbolic-parabolic system in terms of $u$ and $z$, the coupling arising from nonlinear terms involving the  density  $a$, such as $f(a)\nabla z$, $f(a)\Delta u$  and $f(a)\nabla{\rm div}\,u$,  introduces significant challenges. Since the  density $a$ itself lacks a dissipation mechanism, these terms not only lead to a loss of derivatives but also complicate the decay estimates.
\end{rem}

\begin{rem}
Particularly, when $d=3$ and $\sigma_0 = -\frac{3}{2}$, it follows from \eqref{TE.2} that by utilizing the embeddings $L^1(\mathbb{R}^3) \hookrightarrow \dot{B}_{2,1}^{-\frac{3}{2}}(\mathbb{R}^3)$ and $\dot{B}_{2,1}^{0}(\mathbb{R}^3) \hookrightarrow L^2(\mathbb{R}^3)$, we obtain the following optimal time-decay estimate for $u$ and $z$  in the sense of Matsumura and Nishida \cite{MN-PJASAMS-1979}: 
\begin{align*}
\|u(t)\|_{L^2}+\|z(t)\|_{L^2} \lesssim (1+t)^{-\frac{3}{4}}.
\end{align*}
\end{rem}

If we further assume that the low frequency part of the initial data $(u_0, z_0)$, i.e. $(u_0^\ell, z_0^\ell)$, is sufficiently small in $\dot B_{2,\infty}^{\sigma_0}$ and that $a_0^\ell \in \dot B_{2,\infty}^{\sigma_0}$ for $\sigma_0 \in \big [-\frac{d}{2}, \frac{d}{2}-2\big)$, then the decay estimates of  $\|(u,z)(t)\|_{\dot B_{2,1}^{\sigma}}$ holds  for  the broader range $\sigma \in \big(\sigma_0, \frac{d}{2}+1\big]$, stated as follows.

\begin{thm}[Upper-bound: the smallness condition] \label{T1.6}
Let $d\geq 3$. Under the assumptions of Theorem \ref{T1.2}, if  the initial data $(a_0,u_0,z_0)$ further satisfy 
\begin{align}\label{TF.1}
  a_0^\ell,u_0^\ell,z_0^\ell \in \dot B_{2,\infty}^{\sigma_0},\quad\|( u_0^\ell,z_0^\ell)\|_{\dot B_{2,\infty}^{\sigma_0}}\leq\eps_1 \quad \text{with}\quad \sigma_0\in \Big[-\frac{d}{2},\frac{d}{2}-2\Big),
\end{align}
for some sufficiently small constant $\eps_1 > 0$.
Then for all $t \geq 1$, there exists a universal constant $C_4> 0$ such that  
\begin{align}\label{TF.2}
\|a(t)\|_{\dot B_{2,1}^\sigma}+\|(\gamma a-z)(t)\|_{\dot B_{2,1}^\sigma}\leq&\,C_4\delta_{*},  
\end{align}
for all $\sigma\in \big(\sigma_0,\frac{d}{2}\big]$,
and
\begin{align}\label{TF.3}
\|(u,z)(t)\|_{\dot B_{2,1}^\sigma}^\ell\leq&\, C_4 \delta_{*} (1+t)^{-\frac{1}{2}(\sigma-\sigma_0)}, 
\\ \label{TF.4}
\|u(t)\|_{\dot B_{2,1}^{\frac{d}{2}+1}}^h+\|z(t)\|_{\dot B_{2,1}^{\frac{d}{2}}}^h\leq&\, C_4 \delta_{*} (1+t)^{-\frac{1}{2}(\frac{d}{2}+1- \sigma_0 )}, 
\end{align}
for all $\sigma\in\big(\sigma_0,\frac{d}{2}+1\big]$. Here $\delta_{*}$ is defined by \eqref{TE.4}.
\end{thm}

\begin{rem}
From \eqref{TF.3} and \eqref{TF.4}, we directly obtain that 
\begin{align}\label{TTU}
 \|u(t)\|_{\dot B_{2,1}^\sigma} \lesssim&\,   (1+t)^{-\frac{1}{2}(\sigma-\sigma_0)},  \quad \sigma\in\Big(\sigma_0,\frac{d}{2}+1\Big],  \\ \label{TTZ}
  \|z(t)\|_{\dot B_{2,1}^\sigma} \lesssim&\,   (1+t)^{-\frac{1}{2}(\sigma-\sigma_0)},  \quad \sigma\in\Big(\sigma_0,\frac{d}{2} \Big].
\end{align}
In contrast to the assumption on the pressureless isentropic Navier-Stokes system studied in \cite{LNZ-2025-preprint}, which requires $a_0^\ell\in \dot B_{2,\infty}^{\sigma_0+1}$, our work maintains the  regularity assumption $a_0^\ell\in \dot B_{2,\infty}^{\sigma_0}$, owing to the presence of the additional nonlinear term $f(a)\nabla z$ in \eqref{A3}$_2$. By leveraging the newly introduced mode $\gamma a-z$, we derive an improved estimate for the density $a$ as stated in \eqref{TF.2}.
\end{rem}

\begin{rem}
The decay estimates in \eqref{TF.4} do not exhibit faster decay as we expected, since the density $a$ in our system lacks a dissipative structure, which prevents the nonlinear coupling term from contributing an additional decay rate. This behavior is markedly different from that in \cite{LS-SIMA-2023} on the coupled isentropic Navier-Stokes/Euler equations, where $(a^h,u^h)$ decay at a faster rate.
% %\emph{(}see the definitions
%of $\cdot^\ell, \cdot^h$, and the norms $\|\cdot\|_{\dot B_{2,1}^{\frac{d}{2}+1}}^\ell$ and $\|\cdot\|_{\dot B_{2,1}^{\frac{d}{2}+1}}^h$ in Section 2 below\emph{)}.
\end{rem}

\begin{rem}
For $p \geq 2$ and $t \geq 1$, let $\Lambda := (-\Delta)^{-\frac{1}{2}}$. By combining \eqref{TTU} and \eqref{TTZ} with the embedding $\dot B_{2,1}^{\frac{d}{2}-\frac{d}{p}}(\mathbb R^d) \hookrightarrow L^p(\mathbb R^d)$, we obtain the following $L^p$-type time decay  estimates for   $u$ and $z$:
\begin{align*}
\|\Lambda^\sigma u \|_{L^p}\lesssim&\, (1+t)^{-\frac{1}{2}(\sigma+\frac{d}{2}-\frac{d}{p}-\sigma_0)},\quad \text{with}\quad   \sigma+\frac{d}{2}-\frac{d}{p}\in \Big(\sigma_0,\frac{d}{2}+1\Big],\\
\|\Lambda^\sigma z \|_{L^p}\lesssim&\, (1+t)^{-\frac{1}{2}(\sigma+\frac{d}{2}-\frac{d}{p}-\sigma_0)},\quad \text{with}\quad   \sigma+\frac{d}{2}-\frac{d}{p}\in \Big(\sigma_0,\frac{d}{2}\Big].
\end{align*}
\end{rem}

To demonstrate the optimality of the time decay rate of the strong solution $(u,z)$ to the  compressible NST system \eqref{A3}, it is essential to establish a lower bound for them. To this end, we introduce a subset of the Besov spaces $\dot B_{2,\infty}^{\sigma_1}$ with $\sigma_1 \in \mathbb{R}$ (see \cite[Section 3]{Bl-2016-SIMA} and \cite[Theorem 1.2]{BSXZ-Adv-2024}):
\begin{equation}\label{G1.12}
\dot{\mathfrak{B}}_{2,\infty}^{\sigma_1}:=
\bigg\{ f\in\dot B_{2,\infty}^{\sigma_1}\,\bigg|
 \begin{array}{l}
 \exists \,  c_0,   M_0>0, \,\exists \, \{k_j\}_{j\in\mathbb N}\subset \mathbb Z,\,  \textrm{such that}\, \, k_j\rightarrow-\infty,   \\
  |k_j-k_{j-1}|\leq M_0, \,\,\textrm{and}\,\, 2^{\sigma_1k_j} \|\dot\Delta_{k_j}f\|_{L^2}\geq c_0
 \end{array}
   \bigg\}.
\end{equation} 

\begin{thm}[Lower-bound] \label{T1.7}
Let $d\geq 3$. Under the assumptions of Theorem \ref{T1.2}, if the initial data $(a_0,u_0,z_0)$ further satisfy
\begin{align*}
a_0^\ell\in \dot B_{2,\infty}^{\sigma_0},\quad u_0^\ell,z_0^\ell\in 
\dot {\mathfrak B}_{2,\infty}^{\sigma_0},\quad\|( u_0^\ell,z_0^\ell)\|_{\dot B_{2,\infty}^{\sigma_0}}\leq\eps_2 \quad \text{with}\quad \sigma_0\in\Big[-\frac{d}{2},\frac{d}{2}-2\Big),
\end{align*}
for some sufficiently small constant $\eps_2 > 0$.
Then for all $t \geq 1$, there exist two universal constants $c_5>0$ and  $C_5 > 0$ such that  
\begin{align}\label{TG.1}
c_5(1+t)^{-\frac{1}{2}(\sigma-\sigma_0)}\leq \|(u,z)(t)\|_{\dot B_{2,1}^\sigma}\leq C_5 (1+t)^{-\frac{1}{2}(\sigma-\sigma_0)},    
\end{align}
for all $\sigma\in \big(\sigma_0,\frac{d}{2}\big]$, where  $\dot{\mathfrak{B}}_{2,\infty}^{\sigma_0}$ is defined in \eqref{G1.12}.
\end{thm}

\begin{rem}
Since the analysis focuses on the hyperbolic-parabolic structure of the coupled solution $(u, z)$ governed by the equations \eqref{A3}$_2$--\eqref{A3}$_3$, the condition $a_0^\ell \in \dot{\mathfrak{B}}_{2,\infty}^{\sigma_0}$ is no longer necessary. This differs slightly from \cite{BSXZ-Adv-2024} on the isentropic Navier-Stokes equations, which requires $a_0^\ell \in \dot{\mathfrak{B}}_{2,\infty}^{\sigma_0}$. 
\end{rem}

\begin{rem}
In the current framework, a lower bound estimate for $\|u\|_{\dot B_{2,1}^{\sigma}}$ with $\sigma \in \big(\frac{d}{2}, \frac{d}{2}+1\big]$ cannot be established due to the intrinsic coupling between $u$ and $z$. This stands in contrast to \cite{LNZ-2025-preprint}, where only the velocity equation for $u$ was analyzed, enabling a high regularity estimate of $u$.
\end{rem}

\begin{rem}
Unlike the isentropic Navier-Stokes equations  studied  in \cite{CTWZ-JDE-2020,LZ-M2AS-2011}, which impose additional smallness assumptions on the initial data at low frequencies, the present framework does not require such restrictions.  Inspired by \cite{BSXZ-Adv-2024}, we find that, thanks to the  coupling mechanism of $(u, z)$, the hyperbolic component in the frequency spectrum can be effectively suppressed during spectral analysis.
\end{rem}

\subsection{Strategies in our proofs}
Different from   the classical compressible Navier-Stokes-Fourier system studied \cite{MN-PJASAMS-1979,MN-JMKU-1980,Danchin-01-CPDE}, 
where the density $a$ exhibits a dissipative structure and possesses a long-time decay rate, in our compressible NST system \eqref{A3}, the density lacks such a dissipative structure and does not decay over time.
This absence poses a significant challenge in establishing the global well-posedness and time decay rates 
of solutions to the system \eqref{A3}.

In the proof of Theorem \ref{T1.2}, in contrast to the approaches developed in \cite{CMZ-CPAM-2010,Danchin-00-IM} and \cite[Section 4]{Danchin-2018}, where the orthogonal projectors $\mathbb{P}$ and $\mathbb{Q}$ are applied to decompose the velocity field into divergence-free and potential components, respectively, 
we estimate the incompressible part $\mathbb{P}u$ and the potential part $\mathbb{Q}u$ separately in \emph{both low and high frequency regimes}. In particular, for the high-frequency component, the effective velocity method introduced in \cite{DH-MATHANN-2016,Haspot-2011-JDE} is employed. 
In this paper, we treat the subsystem \eqref{A3}$_2$--\eqref{A3}$_3$ as a {\it new} hyperbolic-parabolic system in the variables $(u,z)$. By applying hyperbolic-parabolic theory in \cite{SK-1985-HMJ} , we construct a frequency-localized Lyapunov inequality (see \eqref{G3.9}--\eqref{G3.11}), which captures the dissipative behavior of $(u,z)$ across different frequency ranges: at low frequencies $k \leq 0$, $(u,z)$ behaves like a heat-diffusive system, while at high frequencies $k \geq -1$, it exhibits a damping effect. 
Furthermore, certain nonlinear terms such as $f(a)\nabla z$, $f(a)\Delta u$, and $f(a)\nabla{\rm div}\, u$, lead to a loss of derivatives due to the absence of a dissipative structure in the density $a$. This necessitates more refined estimates for both the high- and low-frequency components of the solution.

On the other hand, for the remaining subsystem \eqref{A3}$_1$, the main difficulty lies in establishing an estimate for $\|a\|_{\widetilde L_t^\infty(\dot B_{2,1}^{\frac{d}{2}-1})}$, as the linear term ${\rm div}\,u$ poses a significant obstacle --- particularly because $\|u\|_{L_t^1(\dot B_{2,1}^{\frac{d}{2}})}$ cannot be controlled. To overcome this issue, we analyze the equation \eqref{A3-2} governing the new variable $\omega =\gamma a - z$, and establish an estimate for $\omega$ in the space $\dot B_{2,1}^{\frac{d}{2}-1} \cap \dot B_{2,1}^{\frac{d}{2}}$ (see Lemma \ref{L3.3}). By means of this new {\it mode} $\omega$, we are able to derive a desired estimate of  $a$. Consequently, we obtain uniform {a priori} estimates with critical regularity. Combined with the local well-posedness result established in Theorem \ref{T1.1}, this guarantees the global well-posedness of solutions to the system \eqref{A3}.

To prove Theorem \ref{T1.4}, we first employ a rescaling argument analogous to that used in the isentropic viscous flows \cite{PM-JMPA-1998,Danchin-2002,DH-MATHANN-2016}, specifically by leveraging the uniform estimates in critical norms established in Theorem \ref{T1.2} with $\varepsilon=1$. Since the density $a^{\varepsilon}$ in \eqref{A4} lacks a dissipative structure, we initially focus on the equation \eqref{A4}$_3$. Thanks to the dissipative nature of the coupling term ${\rm div}\,(zu)$, we  adapt the method from \cite[Section 5.1]{Danchin-2018} to show that \eqref{A4}$_3$ converges to \eqref{A5}$_3$ in the sense of distributions.
Since $a^{\varepsilon} \notin L^2(\mathbb R^+; \dot B_{2,1}^{\frac{d}{2}})$, it follows that ${\rm div}\,(a^{\varepsilon} u^{\varepsilon }) \notin L^1(\mathbb R^+; \dot B_{2,1}^{\frac{d}{2}-1})$, which prevents the previous argument from being applied to establish the convergence of \eqref{A4}$_1$ toward \eqref{A5}$_1$ in the sense of distributions. To overcome this difficulty, we introduce the new variable $\omega^\varepsilon := \gamma a^\varepsilon - z^\varepsilon$, which satisfies  
\begin{align}\label{NJKW}  
\partial_t \omega^{\varepsilon}+{\rm div}(\omega^{\varepsilon }u^{\varepsilon }) -  (\gamma-1)  z^{\varepsilon} {\rm div}\,u^{\varepsilon} = 0.    
\end{align}  
We then employ the approach used in \cite[Section 6]{LSZ-2025-arXiv} to rigorously justify the convergence of \eqref{NJKW} toward \eqref{A5}$_1$ in the sense of distributions. 
For the remaining subsystem \eqref{A4}$_2$, since $a^{\varepsilon}$ in \eqref{A4} lacks a dissipative structure, analogous to that in \cite[Section 6]{LSZ-2025-arXiv}, we apply the operator $\mathbb P$ to \eqref{A4}$_2$ to obtain
\begin{align}
\label{NJKWW}
\partial_t \mathbb P m^{\varepsilon } - \bar\mu\Delta\mathbb P m^{\varepsilon } + \mathbb P{\rm div}\, (m^{\varepsilon }\otimes u^{\varepsilon }) = 0,
\end{align}  
where the momentum is defined as $m^{\varepsilon } := (1+\varepsilon a^{\varepsilon })u^{\varepsilon }$. Using some refined analytical techniques, we establish that $\partial_t\mathbb P m^{\varepsilon } - \bar\mu\Delta\mathbb P m^{\varepsilon } \rightharpoonup v - \bar\mu \Delta v$ as $\varepsilon\rightarrow 0$. To further prove that $\mathbb P{\rm div}\,(m^{\varepsilon } \otimes u^{\varepsilon }) \rightharpoonup \mathbb P(v\cdot\nabla v)$ as $\varepsilon\rightarrow 0$, we require the strong convergence ${\rm div}\,u^\varepsilon \to 0$ as $\varepsilon\rightarrow 0$. To achieve this, we derive dispersive estimates for $z^{\varepsilon }$ and $\mathbb Q u^{\varepsilon }$, following the method originally introduced in \cite{Danchin-2002}. As a result, we obtain the convergence rates stated in \eqref{TD.4} for $( \mathbb Qu^\varepsilon,z^\varepsilon)$, which  allow  us to rigorously justify the low Mach number limit.

In the proof of Theorem \ref{T1.5}, the main difficulty arises from the nonlinear terms involving the density $a$, such as $f(a)\nabla z$, $f(a)\Delta u$,   and $f(a)\nabla{\rm div}\,u$. Due to technical challenges, we need restrict $\sigma_0$ to the interval $\big[-\frac{d}{2}, \frac{d}{2}-2 \big)$, which implies that our analysis is limited to $\mathbb{R}^3(d \geq 3)$. This situation contrasts with previous works such as \cite{LS-SIMA-2023, DX-ARMA-2017} on the isentropic Navier-Stokes equations, where the density equation exhibits a dissipative structure which permits $\sigma_0 \in \big[-\frac{d}{2}, \frac{d}{2} - 1\big)$, thus including the two-dimensional case.  
In Lemma \ref{L5.1}, we first establish the propagation of the $\dot B_{2,\infty}^{\sigma_0}$ norm of $(u,z)$ in the low-frequency regime. Then, by employing the time-weighted energy method from \cite{LS-SIMA-2023} and deriving time-weighted estimates for $(u,z)$ at both high and low frequencies in Lemmas \ref{L5.2}--\ref{L5.3}, we obtain an estimate for $\mathcal{E}_M(t)$ (see its definition in \eqref{G5.9}): 
\begin{align*}
  \mathcal{E}_{M}(t) \lesssim \delta_{*} t^{M - \frac{1}{2}\left(\frac{d}{2} - 1 - \sigma_0\right)},
\end{align*}
which yields \eqref{TE.2}. Additionally, \eqref{TE.3} follows from maximal regularity estimates provided in Lemma \ref{LA.7}. Thus, we complete the proof of the time decay estimates of $(u,z)$ in the case of lower regularity.   

To obtain the time decay estimates for $(u,z)$ in the case of higher regularity stated in Theorem \ref{T1.6}, 
a stronger smallness assumption on $ (u_0^\ell,z_0^\ell)$ is required. 
Following some ideas developed in \cite{GW-CPDE-2012,XX-JDE-2021}, we establish optimal time decay rates for $(u,z)$ in both low and high frequencies separately, as presented in Lemmas \ref{L5.5}--\ref{L5.6}. A key difficulty arises from the derivative loss induced by the density $a$, which lacks a dissipative structure. Consequently, we require estimates for the density $a$ in   Besov space $\dot B_{2,\infty}^{\sigma_0} \cap \dot B_{2,\infty}^{\sigma_0+1}$. However, obtaining the $\dot B_{2,\infty}^{\sigma_0}$ estimate for the density $a$  is challenging due to the presence of the linear term ${\rm div}\,u$ in the  equation \eqref{A3}$_1$. To overcome this
difficulty, we introduce a new method by analyzing the {\it new mode} $\omega =\gamma a - z$ in Lemma \ref{L5.4}, from which the desired estimate for the density $a$ can be derived. 
This allows us to close the energy estimate for $\mathcal{Z}(t)$ (see its definition in \eqref{G5.32}) and thereby completing the proof of Theorem \ref{T1.6}. 
It is worth noting that the decay estimates in \eqref{TF.4} do not exhibit accelerated decay, this behavior contrasts with the results in \cite{LS-SIMA-2023}, where $(a^h,u^h)$ display faster decay rates.

Finally, to derive the lower bound estimates of $(u,z)$, we decompose them into a linear component $(u_L,z_L)$ and a nonlinear component $(u_N,z_N)$.
For the linearized system \eqref{G5.53}, we apply the so-called Hodge decomposition to split the system into \eqref{G5.54} and \eqref{G5.55}. Then, by adapting the method from \cite[Section 3]{BSXZ-Adv-2024}, we establish in Lemma \ref{L5.7} the lower bound decay rate for the pointwise behavior of the solution $(u_L,z_L)$ in the low-frequency regime. It is worth noting that we have eliminated the hyperbolic structure of $(u,z)$ at low frequencies, which \emph{removes} the need to impose additional smallness assumptions on the low-frequency part of the initial data --- a key distinction from earlier works such as \cite{LZ-M2AS-2011} on the isentropic Navier-Stokes equations. Subsequently, through linear analysis, we derive the estimates of $(u_L,z_L)$ in the homogeneous Besov space $\dot B_{2,1}^{\sigma}$ with $\sigma \in (\sigma_0, \frac{d}{2}]$, as stated in Lemma \ref{L5.8}. The corresponding estimates for the nonlinear component $(u_N,z_N)$ is obtained in Lemma \ref{L5.9} via nonlinear analysis. By applying Duhamel's principle and combining the estimates of both the linear and nonlinear parts, we ultimately obtain the desired estimates for $(u,z)$ in Theorem \ref{T1.7}.

\subsection{Structure of this paper}
The rest of this paper is organized as follows. 
In Section 2, we introduce the necessary preliminaries on notations, the Littlewood–Paley decomposition, and certain mixed Besov spaces.
In Section 3, we establish a priori estimates for the solution $(a, u, z)$ to 
prove the global well-posedness of the Cauchy problem \eqref{A3}--\eqref{A3-1}. 
In Section 4, we rigorously justify the low Mach number limits stated in Theorem \ref{T1.4}. 
In Section 5, we derive the optimal decay rates of $(u, z)$ in $L^2$-type Besov norms.
For the upper bound estimates of $(u, z)$, we first derive decay estimates at low regularity; furthermore, under a smallness assumption on the initial data $( u_0, z_0)$, we obtain decay estimates at higher regularity. 
For the lower bound analysis, we first employ spectral analysis to establish pointwise estimates of the solution at low frequencies, and then combine linear and nonlinear analysis to derive the time decay rates of $(u, z)$. 
Finally, Appendix A collects several fundamental properties of Besov spaces and essential technical lemmas.

\section{Preliminaries}
In this section, we introduce several important notations, the Littlewood-Paley decomposition, Besov spaces  and Chemin-Lerner spaces, which will be frequently used throughout this paper.

\subsection{Notations}
The letters $C $ and $c$ represent two generic positive constants that are independent of both time $t$ and the Mach number $\varepsilon$. The notation $A \lesssim B$ means $A \leq CB$, $A \gtrsim B$ means $A \geq cB$, and $A \sim B$ signifies that $cB \leq A \leq CB$. For any Banach space $X$ and functions $f, g \in X$, we define $\|(f,g)\|_{X} := \|f\|_{X} + \|g\|_{X}$. For any $T > 0$ and $1 \leq p \leq \infty$, $L^p(0,T;X)$ denotes the space of measurable functions $f: [0,T] \to X$ such that the mapping $t \mapsto \|f(t)\|_{X}$ belongs to $L^p(0,T)$, equipped with the norm $\|\cdot\|_{L^p(0,T;X)} =: \|\cdot\|_{L_T^p(X)}$. Let $\mathcal{F}(f)=\widehat{f}$ and $\mathcal{F}^{-1}(f)=\check{f}$ denote the Fourier transform of $f$ and its inverse, respectively.

\subsection{Littlewood-Paley decomposition}
We now recall the Littlewood–Paley decomposition. For further details, interested reader can refer to \cite[Chapters 2–3]{BCD-Book-2011}. Let $\chi(\xi)$ be a smooth, radial, non-increasing function supported in $B(0,\frac{4}{3})$ such that $\chi(\xi) = 1$ for all $\xi \in B(0,\frac{3}{4})$. Define $\phi(\xi) := \chi\big(\frac{\xi}{2}\big) - \chi(\xi)$. Then the function $\phi$ satisfies
\begin{align*}
\sum_{k \in \mathbb{Z}} \phi(2^{-k}\xi) = 1 \quad \text{for all } \xi \neq 0,
\quad \text{and} \quad
\mathrm{supp}\,\phi \subset \Big\{ \xi \in \mathbb{R}^d \,\Big|\, \frac{3}{4} \leq |\xi| \leq \frac{8}{3} \Big\}.
\end{align*}
For any $k \in \mathbb{Z}$, the homogeneous dyadic block $\dot{\Delta}_k$ is defined by  
\begin{align*}
\dot{\Delta}_k f := \mathcal{F}^{-1}\big( \phi(2^{-k}\cdot)\mathcal{F}(f) \big) = 2^{kd} h(2^k\cdot) \ast f, \quad \text{where} \quad h := \mathcal{F}^{-1}\phi.
\end{align*}
Let $\mathcal{P}$ denote the class of all polynomials  on $\mathbb R^d$,   and let
$\mathcal{S}^\prime_h=\mathcal{S}^\prime/\mathcal{P}$ represent the
tempered distributions on $\mathbb R^d$ modulo polynomials.  Then for any $f\in \mathcal{S}^\prime_h$, one gets 
\begin{align*}
f=\sum_{k\in\mathbb Z}\dot\Delta_k f  ,\quad\text{where}\quad \dot\Delta_k\dot\Delta_l=0\quad {\rm if}\quad |k-l|\geq2. 
\end{align*}

With the aid of those dyadic blocks, we next present the following definition of homogeneous Besov spaces and mixed space-time Besov spaces.

\subsection{Besov spaces  and Chemin-Lerner spaces}

\begin{defn}
For $ s \in \mathbb{R} $ and $ 1 \leq p, r \leq \infty $, the homogeneous Besov spaces $ \dot{B}_{p,r}^s$ are defined as 
\begin{align*}
\dot B_{p,r}^s:=\big\{ f\in \mathcal{S}^\prime_h \,\big{|}\, \|f\|_{\dot{B}_{p,r}^s}:=\big\|\{2^{ks}\|\dot\Delta_k u\|_{L^p}\}_{k\in\mathbb Z}\big\|_{l^r}<\infty\big\}. 
\end{align*}
\end{defn}
It is natural to recall the class of mixed space-time Besov spaces introduced by Chemin and Lerner \cite{C-N-1995}.

\begin{defn}
For $T>0$, $s\in \mathbb R$, $1\leq \varkappa,p,r\leq \infty $,  the homogeneous Chemin-Lerner spaces  $\widetilde L^\kappa(0,T;\dot B_{p,r}^s)$ are defined by
\begin{align*}
\widetilde L^\varkappa (0,T;\dot B_{p,r}^s)=\big\{f\in L^\varkappa (0,T;\mathcal{S}^\prime_h)\,\big{|}\, \|f\|_{\widetilde L^\varkappa (0,T;\dot B_{p,r}^s)}:= \big\|\{2^{ks}\|\dot\Delta_k f\|_{L^\varkappa _T(L^p)}\}_{k\in\mathbb Z}\big\|_{l^r}<\infty             \big\} .   
\end{align*}
\end{defn}
By applying Minkowski’s inequality, one has
\begin{align*}
\|f\|_{\widetilde L^\varkappa _T(\dot B_{p,r}^s)}\leq \|f\|_{  L^\varkappa _T(\dot B_{p,r}^s)}\quad {\rm if} \quad  \varkappa  \leq r; \quad \|f\|_{\widetilde L^\varkappa _T(\dot B_{p,r}^s)}\geq \|f\|_{  L^\varkappa _T(\dot B_{p,r}^s)}\quad {\rm if} \quad   \varkappa  \geq r.     
\end{align*}
where $\|\cdot\|_{L_T^\varkappa (\dot B_{p,r}^s)}$ is the usual Lebesgue-Besov norm. Moreover, we denote
\begin{align*}
\mathcal{C}_b(\mathbb R^+;\dot B_{p,r}^s):=\big\{f\in \mathcal{C}(\mathbb R^+;\dot B_{p,r}^s)\,\big|\, \|f\|_{\widetilde L^\infty(\mathbb R^+;\dot B_{p,r}^{s})}<\infty  \big\}.    
\end{align*}
Restricting Besov norms to the low- and high-frequency components of distributions is central to our approach. We consistently use the following notation for any $s\in \mathbb R$ and $1\leq p,r\leq \infty$:
\begin{equation*} \left\{
\begin{array}{ll}
 \!\!   \|f\|^\ell_{\dot B_{p,r}^s}:=\big\|\{2^{ks}\|\dot\Delta_k f\|_{L^p}\}_{k\leq 0} \big\|_{l^r}, \ & \|f\|^\ell_{\widetilde L^\varkappa _T(\dot B_{p,r}^s)}:=\big\|\{2^{ks}\|\dot\Delta_k f\|_{L^\varkappa _T(L^p)}\}_{k\leq 0} \big\|_{l^r},       \\
   \!\!   \|f\|^h_{\dot B_{p,r}^s}:=\big\|\{2^{ks}\|\dot\Delta_k f\|_{L^p}\}_{k\geq -1} \big\|_{l^r},\  & \|f\|^h_{\widetilde L^\varkappa _T(\dot B_{p,r}^s)}:= \big\|\{2^{ks}\|\dot\Delta_k f\|_{L^\varkappa _T(L^p)}\}_{k\geq -1} \big\|_{l^r}.            
\end{array}\right.
\end{equation*}
Define
\begin{align*}
 f ^\ell:= \sum_{k\leq -1}\dot\Delta_k f,\quad  f ^h:= \sum_{k\geq 0}\dot\Delta_k f.  
\end{align*}
For any $s^\prime > 0$, it holds that
\begin{equation*} \left\{
\begin{array}{ll}
    \!\!   \|f^\ell\|_{\dot B_{p,r}^s}\lesssim \|f\|_{\dot B_{p,r}^s}^\ell\lesssim \|f\|_{\dot B_{p,r}^{s-s^\prime}}^\ell, \quad & \|f^h\|_{\dot B_{p,r}^s}\lesssim \|f\|_{\dot B_{p,r}^s}^h\lesssim \|f\|_{\dot B_{p,r}^{s+s^\prime}}^h,     \\
   \!\!   \|f^\ell\|_{\widetilde L^\varkappa _T(\dot B_{p,r}^s)}\lesssim \|f\|_{\widetilde L^\varkappa _T(\dot B_{p,r}^s)}^\ell\lesssim \|f\|_{\widetilde L^\varkappa _T(\dot B_{p,r}^{s-s^\prime})}^\ell, \  & \|f^h\|_{\widetilde L^\varkappa _T(\dot B_{p,r}^s)}\lesssim \|f\|_{\widetilde L^\varkappa _T(\dot B_{p,r}^s)}^h\lesssim \|f\|_{\widetilde L^\varkappa _T(\dot B_{p,r}^{s+s^\prime})}^h.          
\end{array}\right.
\end{equation*}
In addition, we have
\begin{equation*} \left\{
\begin{aligned}
&    \|f\|_{\dot B_{p,r}^s\cap \dot B_{p,r}^{s^\prime}}=\|f\|_{\dot B_{p,r}^s}^\ell+\|f\|_{\dot B_{p,r}^{s^\prime}}^h ,\quad s<s^\prime, \\
&   \|f\|_{\dot B_{p,r}^s+ \dot B_{p,r}^{s^\prime}}=\|f\|_{\dot B_{p,r}^s}^\ell+\|f\|_{\dot B_{p,r}^{s^\prime}}^h ,\quad s>s^\prime.   
\end{aligned}\right.
\end{equation*}

Finally, to investigate the low Mach number limit of the scaled NST system \eqref{A4} as it converges to the incompressible inhomogeneous Navier-Stokes  equations \eqref{A5} , we introduce several auxiliary notations. For any function $f \in \mathcal{S}'(\mathbb{R}^d)$, we define 
\begin{align*}
f^{\ell,\varepsilon}:=\sum_{2^k \varepsilon\leq 2^{k_0}}\dot\Delta_k  f,\quad f^{h,\varepsilon}:= \sum_{2^k \varepsilon> 2^{k_0}}\dot\Delta_k  f, 
\end{align*}
where $k_0$ is a sufficiently large nonnegative integer depending only on the spatial dimension $d$. Additionally, we introduce the following norms:
\begin{align*}
\|f\|_{\dot B_{2,1}^s}^{\ell,\varepsilon}:=\sum_{2^k\varepsilon\leq 2^{k_0}} 2^{ks}\|\dot\Delta_k f\|_{L^2},\quad    \|f\|_{\dot B_{2,1}^s}^{h,\varepsilon}:=\sum_{2^k\varepsilon\geq 2^{k_0}} 2^{ks}\|\dot\Delta_k f\|_{L^2}.
\end{align*}

\section{Global existence of the compressible NST system \eqref{A3}}
In this section, we establish the global existence of strong solutions to the Cauchy problem \eqref{A3}--\eqref{A3-1}. For clarity, we first define the dissipation functional as
\begin{align}\label{G3.1}
\mathcal{D}(t):=\|u\|_{L_t^1(\dot B_{2,1}^{\frac{d}{2}+1})}+\|z\|_{L_t^1(\dot B_{2,1}^{\frac{d}{2}+1}+\dot B_{2,1}^{\frac{d}{2}})} , 
\end{align}
and the norm of the initial data as
\begin{align}\label{G3.2}
\mathcal{X}_0:= \|(a_0,z_0)\|_{\dot B_{2,1}^{\frac{d}{2}-1}\cap  \dot B_{2,1}^{\frac{d}{2}}}+   \|u_0\|_{\dot B_{2,1}^{\frac{d}{2}-1}}.    
\end{align}

We now establish the following uniform-in-time {\it a priori} estimates.

\begin{prop}\label{P3.1}
Assume that $(a,u,z)$ is a strong solution to the Cauchy problem \eqref{A3}--\eqref{A3-1} defined on $[0,T)\times \mathbb R^d$ with a given time $T>0$. For any $t\in (0,T)$, the following inequality holds with a generic constant $0 < \delta < 1$ to be chosen later:
\begin{align}\label{G3.3}
 \|(a,z)\|_{\widetilde L_t^\infty(\dot B_{2,1}^{\frac{d}{2}-1}\cap \dot B_{2,1}^{\frac{d}{2}})}+\|u\|_{\widetilde L_t^\infty(\dot B_{2,1}^{\frac{d}{2}-1})}  \leq \delta.
\end{align}
Then, the solution $(a,u,z)$ satisfies
\begin{align}
\|\gamma a-z\|_{\widetilde L_t^\infty(\dot B_{2,1}^{\frac{d}{2}-1} )}\leq&\, C_0 \Big(\|(a_0,u_0,z_0)\|_{\dot B_{2,1}^{\frac{d}{2}-1} }+ \|z_0\|_{ \dot B_{2,1}^{\frac{d}{2}}}\Big),\label{g333}\\
 \|\gamma a-z\|_{\widetilde L_t^\infty(  \dot B_{2,1}^{\frac{d}{2}})}\leq&\,C_0\Big(\|(u_0,z_0)\|_{\dot B_{2,1}^{\frac{d}{2}-1}}+  \|(a_0,z_0)\|_{ \dot B_{2,1}^{\frac{d}{2}}}\Big), \label{G3.4}
\end{align}
and
\begin{align}\label{G3.5}
 \|(a,z)\|_{\widetilde L_t^\infty(\dot B_{2,1}^{\frac{d}{2}-1}\cap \dot B_{2,1}^{\frac{d}{2}})}+\|u\|_{\widetilde L_t^\infty(\dot B_{2,1}^{\frac{d}{2}-1})}+\mathcal{D}(t)\leq C_0 \mathcal{X}_0,   
\end{align}
where $C_0$ is a constant independent of the time $t$, and $\mathcal{D}(t)$ and $\mathcal{X}_0$ are defined in \eqref{G3.1} and \eqref{G3.2}, respectively.
\end{prop}

The proof of Proposition \ref{P3.1} relies on the results established in Lemmas \ref{L3.2}--\ref{L3.4}.

\subsection{Estimates of $u$ and $z$}
Noting that the equations \eqref{A3}$_2$--\eqref{A3}$_3$ form a hyperbolic-parabolic system, we follow the approach of \cite{SK-1985-HMJ} to show that the solution pair $(u,z)$ exhibits heat-equation-like behavior at low frequencies and damping-dominated dynamics at high frequencies. We first establish uniform regularity estimates of $u$ and $z$.

\begin{lem}\label{L3.2}
Let $(a,u,z)$ be a strong solution to the Cauchy problem \eqref{A3}--\eqref{A3-1} on $[0,T)\times \mathbb R^d$. Then, under the condition \eqref{G3.3}, it holds that
\begin{align}\label{G3.6}
 &\|u\|_{\widetilde L_t^\infty(\dot B_{2,1}^{\frac{d}{2}-1})}+\|z\|_{\widetilde L_t^\infty(\dot B_{2,1}^{\frac{d}{2}-1}\cap \dot B_{2,1}^{\frac{d}{2}})} +  \mathcal{D}(t) \lesssim    \|u_0\|_{\dot B_{2,1}^{\frac{d}{2}-1}}+\|z_0\|_{\dot B_{2,1}^{\frac{d}{2}-1}\cap  \dot B_{2,1}^{\frac{d}{2}}} ,
\end{align}
where $\mathcal{D}(t)$  is defined in \eqref{G3.1}.  
\end{lem}
\begin{proof}
Applying the operator $\dot\Delta_k$ to the subsystem \eqref{A3}$_2$--\eqref{A3}$_3$ and employing the standard $L^2$ energy method, we have  
\begin{align}\label{G3.7}
&\frac{1}{2}{\frac{{\rm d}}{{\rm d}t}} \big(  \|\dot\Delta_k u\|_{L^2}^2+\frac{1}{\gamma}\|\dot\Delta_k z\|_{L^2}^2     \big) +(\mu+\lambda)\|{\rm div}\,\dot\Delta_k u\|_{L^2}^2+\mu \|\nabla \dot\Delta_k u\|_{L^2}^2\nonumber\\  
&\quad\lesssim  \big(\|\dot\Delta(u\cdot\nabla u)\|_{L^2}+ \|\dot\Delta_k  (f(a) (\mu\Delta u+(\mu+\lambda)\nabla{\rm div}\,u )       )  \|_{L^2}+\|\dot\Delta_k(f(a)\nabla z)\|_{L^2} \nonumber\\
&\quad\quad+\|\dot\Delta_k (z{\rm div}\,u)\|_{L^2}+\|[u\cdot\nabla,\dot\Delta_k]z\|_{L^2}+\|{\rm div}\,u\|_{L^\infty}\|\dot\Delta_k z\|_{L^2}\big) \|\dot\Delta_k(u,z)\|_{L^2},
\end{align}
and
\begin{align}\label{G3.8}
&\frac{{\rm d}}{\rm {d}t}\bigg( \int_{\mathbb R^d}\dot\Delta_k u\cdot\nabla\dot\Delta_k z{\rm d}x+ \frac{2\mu+\lambda}{2\gamma} \|\nabla\dot\Delta_k z\|_{L^2}^2  \bigg)+  \|\nabla\dot\Delta_k z\|_{L^2}^2-\gamma\|{\rm div}\,\dot\Delta_k u\|_{L^2}^2      \nonumber\\
&\quad \lesssim  \bigg( \|\dot\Delta_k  (f(a) (\mu\Delta u+(\mu+\lambda)\nabla{\rm div}\,u )       )  \|_{L^2}+\|\dot\Delta_k(f(a)\nabla z)\|_{L^2} \nonumber\\
&\qquad +\|\nabla\dot\Delta_k(z{\rm div}\,u)\|_{L^2}
+\|[u\cdot\nabla,\dot\Delta_k]u\|_{L^2}+\sum_{j=1}^d\|[u\cdot\nabla,\partial_j\dot\Delta_k]z\|_{L^2}\nonumber\\
&\qquad +\|{\rm div}\,u\|_{L^\infty} \| \dot\Delta_k(u,\nabla z)\|_{L^2} \bigg) \|\dot\Delta_k(u,\nabla z)\|_{L^2},
\end{align}
for any $k\in \mathbb Z$.
We now  define the following Lyapunov functional together with its associated dissipation rate:
\begin{align*}
\mathcal{E}_k(t):=&\,\frac{1}{2}\big(\|\dot\Delta_k u\|_{L^2}^2+  \frac{1}{\gamma}\|\dot\Delta_k z\|_{L^2}^2\big)+\eta_1 \bigg( \int_{\mathbb R^d}\dot\Delta_k u\cdot\nabla\dot\Delta_k z{\rm d}x+ \frac{2\mu+\lambda}{2\gamma} \|\nabla\dot\Delta_k z\|_{L^2}^2  \bigg), \\ 
\mathcal{D}_k(t):=&\,(\mu+\lambda)\|{\rm div}\,\dot\Delta_k u\|_{L^2}^2+\mu \|\nabla \dot\Delta_k u\|_{L^2}^2+\eta_1 \big(\|\nabla\dot\Delta_k z\|_{L^2}^2-\gamma\|{\rm div}\,\dot\Delta_k u\|_{L^2}^2      \big),  
\end{align*}
where $\eta_1$ is a small constant to be chosen later.
Adding \eqref{G3.7} and \eqref{G3.8} together yields
\begin{align}\label{G3.9}
\frac{{\rm d}}{{\rm d}t}\mathcal{E}_{k}(t)+\mathcal{D}_k(t)
\lesssim\,& \bigg( \|\dot\Delta_k  (f(a) (\mu\Delta u+(\mu+\lambda)\nabla{\rm div}\,u )       )  \|_{L^2}+\|\dot\Delta_k(f(a)\nabla z)\|_{L^2}   \nonumber\\
&+\|\dot\Delta_k(z{\rm div}\,u)\|_{L^2}  +\|\nabla\dot\Delta_k(z{\rm div}\,u)\|_{L^2}+\|[u\cdot\nabla,\dot\Delta_k] u\|_{L^2}\nonumber\\
&+\|[u\cdot\nabla,\dot\Delta_k] z\|_{L^2}+\sum_{j=1}^d\|[u\cdot\nabla,\partial_j\dot\Delta_k]z\|_{L^2}\nonumber\\
& +\|{\rm div}\, u\|_{L^\infty}\|\dot\Delta_k(u,z,\nabla z)\|_{L^2}\bigg) \|\dot\Delta_k(u,z,\nabla z)\|_{L^2}.
\end{align}
Then,  choosing  $\eta_1>0$ sufficiently small and applying Young's inequality, we obtain  
\begin{align}\label{G3.10}
\mathcal{E}_k(k)\backsim \|\dot\Delta_k(u,z,\nabla z)\|_{L^2}^2,
\end{align}
and
\begin{align}\label{G3.11}
\quad \min\{1,2^{2k}\} \|\dot\Delta_k(u,z,\nabla z)\|_{L^2}^2\lesssim D_k(t). 
\end{align}
Substituting \eqref{G3.10}--\eqref{G3.11} into \eqref{G3.9}, we   infer that
\begin{align}\label{G3.12}
&\|\dot\Delta_k(u,z,\nabla z)\|_{L_t^\infty(L^2)} +\min\{1,2^{2k}\}\|\dot\Delta_k (u,z,\nabla z) \|_{L_t^1(L^2)}\nonumber\\
&\quad \lesssim \|\dot\Delta_k(u_0,z_0,\nabla z_0)\|_{L^2}+\|\dot\Delta_k  (f(a) (\mu\Delta u+(\mu+\lambda)\nabla{\rm div}\,u )   )  \|_{L_t^1(L^2)}\nonumber\\
&\qquad +\|\dot\Delta_k(f(a)\nabla z)\|_{L_t^1(L^2)}+\|\dot\Delta_k(z{\rm div}\,u)\|_{L_t^1(L^2)}+\|\nabla\dot\Delta_k(z{\rm div}\,u)\|_{L_t^1(L^2)}\nonumber\\
&\qquad +\|[u\cdot\nabla,\dot\Delta_k] u\|_{L^1(L^2)}+\|[u\cdot\nabla,\dot\Delta_k] z\|_{L_t^1(L^2)}\nonumber\\
&\qquad +\sum_{j=1}^d\|[u\cdot\nabla,\partial_j\dot\Delta_k]z\|_{L_t^1(L^2)}+\|{\rm div}\,u\|_{L_t^1(L^\infty)}\|\dot\Delta_k(u,z,\nabla z)\|_{L_t^\infty(L^2)}.
\end{align}
Multiplying \eqref{G3.12} by $2^{k(\frac{d}{2}-1)}$, taking the supremum over $[0,t]$, and summing over $k\in\mathbb{Z}$, we arrive at
\begin{align*} 
&\|u\|_{\widetilde L_t^\infty(\dot B_{2,1}^{\frac{d}{2}-1})}+\|z\|_{\widetilde L_t^\infty(\dot B_{2,1}^{\frac{d}{2}-1}\cap \dot B_{2,1}^{\frac{d}{2}})} +\|u\|_{L_t^1(\dot B_{2,1}^{\frac{d}{2}+1}+\dot B_{2,1}^{\frac{d}{2}-1})}+ \|z\|_{L_t^1(\dot B_{2,1}^{\frac{d}{2}+1}+\dot B_{2,1}^{\frac{d}{2}})}\nonumber\\
&\quad\lesssim  \|u_0\|_{\dot B_{2,1}^{\frac{d}{2}-1}}+\|z_0\|_{\dot B_{2,1}^{\frac{d}{2}-1}\cap \dot B_{2,1}^{\frac{d}{2}}}+\|   f(a) (\mu\Delta u+(\mu+\lambda)\nabla{\rm div}\,u   )  \|_{L_t^1(\dot B_{2,1}^{\frac{d}{2}-1})}\nonumber\\
&\quad\quad+\|f(a)\nabla z\|_{L_t^1(\dot B_{2,1}^{\frac{d}{2}-1})}+\|z{\rm div}\,u\|_{L_t^1(\dot B_{2,1}^{\frac{d}{2}-1}\cap \dot B_{2,1}^{\frac{d}{2}})}+\sum_{k\in \mathbb Z} 2^{k(\frac{d}{2}-1)} \|[u\cdot\nabla,\dot\Delta_k] u\|_{L^1(L^2)}\nonumber\\
&\quad\quad+\sum_{k\in \mathbb Z} 2^{k(\frac{d}{2}-1)} \|[u\cdot\nabla,\dot\Delta_k] z\|_{L^1(L^2)}+\sum_{k\in \mathbb Z} 2^{k(\frac{d}{2}-1)} \sum_{j=1}^d \|[u\cdot\nabla,\partial_j\dot\Delta_k]z\|_{L_t^1(L^2)}\nonumber\\
&\quad\quad+\|{\rm div}u\|_{L_t^1(L^\infty)}\Big( \|u\|_{\widetilde L_t^\infty(\dot B_{2,1}^{\frac{d}{2}-1})}+\|z\|_{\widetilde L_t^\infty(\dot B_{2,1}^{\frac{d}{2}-1}\cap \dot B_{2,1}^{\frac{d}{2}})}        \Big),
\end{align*}
which combined with the embedding $\dot B_{2,1}^{\frac{d}{2}}(\mathbb R^d) \hookrightarrow L^\infty(\mathbb{R}^d)$, the smallness assumption \eqref{G3.3}, and \eqref{A.3}--\eqref{A.3-1} in Lemma \ref{LA.5}, yields
\begin{align} \label{G3.13}
&\|u\|_{\widetilde L_t^\infty(\dot B_{2,1}^{\frac{d}{2}-1})}+\|z\|_{\widetilde L_t^\infty(\dot B_{2,1}^{\frac{d}{2}-1}\cap \dot B_{2,1}^{\frac{d}{2}})} +\|u\|_{L_t^1(\dot B_{2,1}^{\frac{d}{2}+1}+\dot B_{2,1}^{\frac{d}{2}-1})}+ \|z\|_{L_t^1(\dot B_{2,1}^{\frac{d}{2}+1}+\dot B_{2,1}^{\frac{d}{2}})}\nonumber\\
&\quad\lesssim  \|u_0\|_{\dot B_{2,1}^{\frac{d}{2}-1}}+\|z_0\|_{\dot B_{2,1}^{\frac{d}{2}-1}\cap \dot B_{2,1}^{\frac{d}{2}}}+\|   f(a) (\mu\Delta u+(\mu+\lambda)\nabla{\rm div}\,u   )  \|_{L_t^1(\dot B_{2,1}^{\frac{d}{2}-1})}\nonumber\\
&\quad\quad+\|f(a)\nabla z\|_{L_t^1(\dot B_{2,1}^{\frac{d}{2}-1})}+\|z{\rm div}\,u\|_{L_t^1(\dot B_{2,1}^{\frac{d}{2}-1}\cap \dot B_{2,1}^{\frac{d}{2}})}+ \delta \|u\|_{L_t^1(\dot B_{2,1}^{\frac{d}{2}+1})}. 
\end{align}

To eliminate $\delta\|u\|_{L_t^1(\dot B_{2,1}^{\frac{d}{2}+1})}$, we require the maximal regularity estimate of $u$ in the high-frequency regime. For the equation \eqref{A3}$_2$, using Lemma \ref{LA.7} gives
\begin{align}\label{G3.14}
\|u\|_{\widetilde L_t^\infty(\dot B_{2,1}^{\frac{d}{2}-1})}^h +\|u\|_{  L_t^1(\dot B_{2,1}^{\frac{d}{2}+1})}^h \lesssim&\,  \|u_0\|_{\dot B_{2,1}^{\frac{d}{2}-1}}^h+ \|   f(a) (\mu\Delta u+(\mu+\lambda)\nabla{\rm div}\,u   )  \|_{L_t^1(\dot B_{2,1}^{\frac{d}{2}-1})}^h\nonumber\\
&+\|f(a)\nabla z\|_{L_t^1(\dot B_{2,1}^{\frac{d}{2}-1})}^h+\delta \|u\|_{L_t^1(\dot B_{2,1}^{\frac{d}{2}+1})}.
\end{align}
Combining \eqref{G3.13} and \eqref{G3.14}, we end up with
\begin{align}\label{G3.15}
&\|u\|_{\widetilde L_t^\infty(\dot B_{2,1}^{\frac{d}{2}-1})}+\|z\|_{\widetilde L_t^\infty(\dot B_{2,1}^{\frac{d}{2}-1}\cap \dot B_{2,1}^{\frac{d}{2}})} +\|u\|_{L_t^1(\dot B_{2,1}^{\frac{d}{2}+1})}+ \|z\|_{L_t^1(\dot B_{2,1}^{\frac{d}{2}+1}+\dot B_{2,1}^{\frac{d}{2}})}\nonumber\\
&\quad\lesssim  \|u_0\|_{\dot B_{2,1}^{\frac{d}{2}-1}}+\|z_0\|_{\dot B_{2,1}^{\frac{d}{2}-1}\cap \dot B_{2,1}^{\frac{d}{2}}}+\|   f(a) (\mu\Delta u+(\mu+\lambda)\nabla{\rm div}\,u   )  \|_{L_t^1(\dot B_{2,1}^{\frac{d}{2}-1})}\nonumber\\
&\quad\quad+\|f(a)\nabla z\|_{L_t^1(\dot B_{2,1}^{\frac{d}{2}-1})}+\|z{\rm div}\,u\|_{L_t^1(\dot B_{2,1}^{\frac{d}{2}-1}\cap \dot B_{2,1}^{\frac{d}{2}})}.     
\end{align}
Below we estimate the nonlinear terms on the right-hand side of \eqref{G3.15} one by one. From Lemma \ref{LA.4} and the inequality \eqref{A.2} in Lemma \ref{LA.5}, it follows that
\begin{align}\label{G3.16}
 \|   f(a) (\mu\Delta u+(\mu+\lambda)\nabla{\rm div}\,u   )  \|_{L_t^1(\dot B_{2,1}^{\frac{d}{2}-1})}\lesssim&\, \|a\|_{\widetilde L_t^\infty(\dot B_{2,1}^{\frac{d}{2}})}\|u\|_{  L_t^1(\dot B_{2,1}^{\frac{d}{2}+1})}\nonumber\\
 \lesssim&\,\delta \|u\|_{  L_t^1(\dot B_{2,1}^{\frac{d}{2}+1})},\\ \label{G3.17}
 \|f(a)\nabla z\|_{L_t^1(\dot B_{2,1}^{\frac{d}{2}-1})}\lesssim&\, \|a\|_{\widetilde L_t^\infty(\dot B_{2,1}^{\frac{d}{2}-1}\cap \dot B_{2,1}^{\frac{d}{2}})} \|z\|_{L_t^1(\dot B_{2,1}^{\frac{d}{2}+1}+\dot B_{2,1}^{\frac{d}{2}}   ) }\nonumber\\
 \lesssim&\, \delta \|z\|_{L_t^1(\dot B_{2,1}^{\frac{d}{2}+1}+\dot B_{2,1}^{\frac{d}{2}}   ) },\\\label{G3.18}
 \|z{\rm div}\,u\|_{L_t^1(\dot B_{2,1}^{\frac{d}{2}-1}\cap \dot B_{2,1}^{\frac{d}{2}})}\lesssim&\, \|z\|_{\widetilde L_t^\infty(\dot B_{2,1}^{\frac{d}{2}-1}\cap \dot B_{2,1}^{\frac{d}{2}})}  \|u\|_{  L_t^1(\dot B_{2,1}^{\frac{d}{2}+1})}\nonumber\\
 \lesssim&\,\delta \|u\|_{  L_t^1(\dot B_{2,1}^{\frac{d}{2}+1})}.
\end{align}
Plugging the estimates \eqref{G3.16}--\eqref{G3.18} into \eqref{G3.15}, we consequently obtain
\begin{align*} 
&\|u\|_{\widetilde L_t^\infty(\dot B_{2,1}^{\frac{d}{2}-1})}+\|z\|_{\widetilde L_t^\infty(\dot B_{2,1}^{\frac{d}{2}-1}\cap \dot B_{2,1}^{\frac{d}{2}})} +\|u\|_{L_t^1(\dot B_{2,1}^{\frac{d}{2}+1})}+ \|z\|_{L_t^1(\dot B_{2,1}^{\frac{d}{2}+1}+\dot B_{2,1}^{\frac{d}{2}})}\nonumber\\
&\quad\lesssim  \|u_0\|_{\dot B_{2,1}^{\frac{d}{2}-1}}+\|z_0\|_{\dot B_{2,1}^{\frac{d}{2}-1}\cap \dot B_{2,1}^{\frac{d}{2}}}+\delta \Big(\|u\|_{L_t^1(\dot B_{2,1}^{\frac{d}{2}+1})}+ \|z\|_{L_t^1(\dot B_{2,1}^{\frac{d}{2}+1}+\dot B_{2,1}^{\frac{d}{2}})}\Big), 
\end{align*}
which together with smallness of $\delta$ leads to \eqref{G3.6}.
\end{proof}

\subsection{Estimate of $\gamma a-z$ and $a$}
Notice that the linear term ${\rm div}\,u$ in \eqref{A3}$_1$ prevents us from obtaining the estimate of $\|a\|_{\widetilde L_t^\infty(\dot B_{2,1}^{\frac{d}{2}-1})}$, since $\|u\|_{L_t^1(\dot B_{2,1}^{\frac{d}{2}})}$ cannot be controlled directly. To overcome this difficulty, we study the equation \eqref{A3-2} concerning the new variable $\omega = \gamma a - z$.
\begin{lem}\label{L3.3}
Let $(a,u,z)$ be a strong solution to the Cauchy problem \eqref{A3}--\eqref{A3-1} on $[0,T)\times \mathbb R^d$. Then, under the condition \eqref{G3.3}, it holds that
\begin{align}
   \|\gamma a-z\|_{\widetilde L_t^\infty(\dot B_{2,1}^{\frac{d}{2}-1} )}   \lesssim &\,   \|(a_0,u_0,z_0)\|_{\dot B_{2,1}^{\frac{d}{2}-1}}+\| z_0\|_{\dot B_{2,1}^{\frac{d}{2} }} ,\label{g333a}\\
   \|\gamma a-z\|_{\widetilde L_t^\infty(\dot B_{2,1}^{\frac{d}{2} } )}   \lesssim&\,   \|(u_0,z_0)\|_{\dot B_{2,1}^{\frac{d}{2}-1 }}+ \|(a_0,z_0)\|_{\dot B_{2,1}^{\frac{d}{2} }}.\label{G3.19}
\end{align}
\end{lem}
\begin{proof}
The equation \eqref{A3-2} can be reformulated as
\begin{align*} 
\partial_t   \omega+ u\cdot\nabla\omega   =-\gamma(a-z){\rm div}\,u .
\end{align*}
By applying Lemma \ref{LA.6} and the estimate \eqref{G3.6}, one has
\begin{align*}
\|\omega\|_{\widetilde L_t^\infty(\dot B_{2,1}^{\frac{d}{2}-1})} \lesssim&\,  e^{C\mathcal{D}(t)}  \Big(\|\omega_0\|_{ \dot B_{2,1}^{\frac{d}{2}-1}  } + \|\gamma(a-z){\rm div}\,u\|_{L_t^1( \dot B_{2,1}^{\frac{d}{2}-1})}\Big)\nonumber\\
\lesssim&\, \|\omega_0\|_{ \dot B_{2,1}^{\frac{d}{2}-1}  } +\|\gamma(a-z)\|_{\widetilde L_t^\infty(\dot B_{2,1}^{\frac{d}{2}-1} )} \|u\|_{L_t^1(\dot B_{2,1}^{\frac{d}{2}+1})} \nonumber\\
\lesssim&\, \|\omega_0\|_{ \dot B_{2,1}^{\frac{d}{2}-1} } +\delta\|\omega\|_{\widetilde L_t^\infty(\dot B_{2,1}^{\frac{d}{2}-1} )}+\delta \|z\|_{\widetilde L_t^\infty(\dot B_{2,1}^{\frac{d}{2}-1} )},  
\end{align*}
and
\begin{align*}
\|\omega\|_{\widetilde L_t^\infty(  \dot B_{2,1}^{\frac{d}{2}})} \lesssim&\,  e^{C\mathcal{D}(t)}  \Big(\|\omega_0\|_{   \dot B_{2,1}^{\frac{d}{2}} } + \|\gamma(a-z){\rm div}\,u\|_{L_t^1(  \dot B_{2,1}^{\frac{d}{2}} )}\Big)\nonumber\\
\lesssim&\, \|\omega_0\|_{   \dot B_{2,1}^{\frac{d}{2}} } +\|\gamma(a-z) \|_{\widetilde L_t^\infty(  \dot B_{2,1}^{\frac{d}{2}})} \|u\|_{L_t^1(\dot B_{2,1}^{\frac{d}{2}+1})} \nonumber\\
\lesssim&\, \|\omega_0\|_{   \dot B_{2,1}^{\frac{d}{2}} } +\delta\|\omega\|_{\widetilde L_t^\infty(  \dot B_{2,1}^{\frac{d}{2}})}+\delta\|z\|_{\widetilde L_t^\infty(\dot B_{2,1}^{\frac{d}{2} } )}.
\end{align*}
 Due to the smallness of $\delta$ and uniform estimate \eqref{G3.6}, we derive \eqref{g333a} and \eqref{G3.19}.   
\end{proof}

\begin{lem}\label{L3.4}
 Let $(a,u,z)$ be a strong solution to the Cauchy problem \eqref{A3}--\eqref{A3-1} on $[0,T)\times \mathbb R^d$. Then, under the condition \eqref{G3.3}, it holds that
\begin{align} \label{G3.20}
 & \|a \|_{\widetilde L_t^\infty(\dot B_{2,1}^{\frac{d}{2}-1}\cap \dot B_{2,1}^{\frac{d}{2}})}   \lesssim    \|(a_0,z_0)\|_{\dot B_{2,1}^{\frac{d}{2}-1}\cap  \dot B_{2,1}^{\frac{d}{2}}}+\|u_0\|_{\dot B_{2,1}^{\frac{d}{2}-1}}.
\end{align}   
\end{lem}
\begin{proof}
By combining the estimates \eqref{G3.6} and \eqref{G3.19} and applying the triangle inequality
of the norm $ \|\cdot \|_{\widetilde L_t^\infty(\dot B_{2,1}^{\frac{d}{2}-1}\cap \dot B_{2,1}^{\frac{d}{2}})} $, we have
\begin{align*}
  \|a \|_{\widetilde L_t^\infty(\dot B_{2,1}^{\frac{d}{2}-1}\cap \dot B_{2,1}^{\frac{d}{2}})} \lesssim &\, \|\gamma a-z\|_{\widetilde L_t^\infty(\dot B_{2,1}^{\frac{d}{2}-1}\cap \dot B_{2,1}^{\frac{d}{2}})}+\|z \|_{\widetilde L_t^\infty(\dot B_{2,1}^{\frac{d}{2}-1}\cap \dot B_{2,1}^{\frac{d}{2}})}\nonumber\\
  \lesssim&\,  \|(a_0,z_0)\|_{\dot B_{2,1}^{\frac{d}{2}-1}\cap  \dot B_{2,1}^{\frac{d}{2}}}+\|u_0\|_{\dot B_{2,1}^{\frac{d}{2}-1}}.
\end{align*}
Thus, we complete the proof of \eqref{G3.20}.
\end{proof}

\begin{proof}[Proof of Proposition \ref{P3.1}]
It follows from \eqref{G3.6} and \eqref{G3.20} that
\begin{align*}
\|(a,z)\|_{\widetilde L_t^\infty(\dot B_{2,1}^{\frac{d}{2}-1} \cap \dot B_{2,1}^{\frac{d}{2}})} + \|u\|_{\widetilde L_t^\infty(\dot B_{2,1}^{\frac{d}{2}-1})} + \mathcal{D}(t) \leq C' \mathcal{X}_0,
\end{align*}
for some uniform constant $C' > 0$. Furthermore, combining this with \eqref{G3.19}, we obtain
\begin{align*}
\| \gamma a - z\|_{\widetilde L_t^\infty(\dot B_{2,1}^{\frac{d}{2}-1})} \leq&\, e^{C C' \mathcal{X}_0}\Big( \|(a_0,u_0, z_0)\|_{\dot B_{2,1}^{\frac{d}{2}-1}}+\| z_0 \|_{\dot B_{2,1}^{\frac{d}{2}}}  \Big),\\
\| \gamma a - z\|_{\widetilde L_t^\infty(\dot B_{2,1}^{\frac{d}{2}})} \leq&\, e^{C C' \mathcal{X}_0} \Big( \|(u_0, z_0)\|_{\dot B_{2,1}^{\frac{d}{2}-1}}+\|(a_0,z_0)\|_{\dot B_{2,1}^{\frac{d}{2}}}  \Big).
\end{align*}
Thus, the uniform estimates \eqref{g333}--\eqref{G3.5} are obtained, which completes the proof of Proposition~\ref{P3.1}.
\end{proof}

\subsection{Proof of Theorem \ref{T1.2}} With Proposition \ref{P3.1} in hand, this subsection is devoted to proving  the global existence
of $(a,u,z)$.

Factually, Theorem \ref{T1.1} establishes the existence of a maximal existence time $T_0$ such that the Cauchy problem \eqref{A3}--\eqref{A3-1} admits a unique strong solution $(a,u,z)$ satisfying \eqref{TA.2}. We define the new energy functional:
\begin{align}\label{G3.21}
\mathcal{X}(t):=\|(a,z)\|_{\widetilde L_t^\infty(\dot B_{2,1}^\frac{d}{2})}+\|u\|_{\widetilde L_t^\infty(\dot B_{2,1}^{\frac{d}{2}-1})}+\|u\|_{L_t^1(\dot B_{2,1}^{\frac{d}{2}+1})}.    
\end{align}
Then, we set
\begin{align}\label{G3.22}
T^*:=\sup \{t\in[0,T_0)\, | \,  \mathcal{X}(t)\leq \delta   \}.   
\end{align}
It is clear that $T^* \in (0, T_0]$. We aim to prove that $T^* = T_0$. Suppose, for contradiction, that $T^* < T_0$. Then we choose 
\begin{align}\label{G3.23}
\delta_0 := \min\Big\{\frac{\delta}{2C_0}, 1\Big\},
\end{align}
%where $\delta_0$ is as defined in \eqref{TB.1}.
where $C_0$ is as defined in Proposition~\ref{P3.1}.
From Proposition \ref{P3.1}, we have
\begin{align}\label{G3.24}
\mathcal{X}(t)\leq&\, \|(a,z)\|_{\widetilde L_t^\infty(\dot B_{2,1}^{\frac{d}{2}-1}\cap \dot B_{2,1}^{\frac{d}{2}})}+\|u\|_{\widetilde L_t^\infty(\dot B_{2,1}^{\frac{d}{2}-1})}+\mathcal{D}(t)\nonumber\\
\leq&\, C_0 \Big(\|(a_0,z_0)\|_{\dot B_{2,1}^{\frac{d}{2}-1}\cap  \dot B_{2,1}^{\frac{d}{2}}}+   \|u_0\|_{\dot B_{2,1}^{\frac{d}{2}-1}}  \Big)\nonumber\\
\leq&\, \frac{1}{2}\delta,
\end{align}
for all $t\in (0,T^*)$. By leveraging the time continuity property and \eqref{G3.24}, we obtain 
\begin{align*} \mathcal{X}(T^*) \leq \frac{\delta}{2}, 
\end{align*} 
which contradicts the definition of $T^*$ given in \eqref{G3.22}. Therefore, $T^*$ cannot be the maximal existence time, and thus $T_0 = T^*$. 
By employing the uniform estimate \eqref{G3.24} and a standard continuity argument, we   readily conclude that $T^* = T_0 = +\infty$.   Therefore, $(a,u,z)$ is the unique global solution to the Cauchy problem \eqref{A3}--\eqref{A3-1}. The proof of Theorem \ref{T1.2} is thus completed. \hfill $\square$

\section{Low much Number limit of the  compressible NST system \eqref{A4}}
In this section, we aim to rigorously prove Theorem \ref{T1.4},  which characterizes the low Mach number limit of the compressible NST system \eqref{A4} as it converges to the incompressible inhomogeneous Navier-Stokes equations \eqref{A5}.

In fact, Theorem \ref{T1.2} addresses the system \eqref{A4} with $\varepsilon=1$. 
Recall
\begin{align}\label{G4.1-1}
(a,u,z)(t,x)= \varepsilon (a^\varepsilon,u^\varepsilon,z^\varepsilon)(\varepsilon^2 t,\varepsilon x),
\end{align}
and 
\begin{align}\label{G4.2-1}
(a_0,u_0,z_0)(x) = \varepsilon (a^\varepsilon_0,u^\varepsilon_0,z^\varepsilon_0)(\varepsilon x).
\end{align}
%where $a = \rho - 1$ and $z = P - 1$.
%For any $\varepsilon\in(0,1)$, it is straightforward to verify that $(a^\varepsilon,u^\varepsilon,z^\varepsilon)$ satisfies the scaled compressible Navier-Stokes-Transport system \eqref{A4}.
Applying \eqref{scale} from Lemma \ref{LA.2} and combining it with \eqref{G4.1-1}--\eqref{G4.2-1} yields  
\begin{align*}
\|(a^\varepsilon_0,u^\varepsilon_0,z^\varepsilon_0) \|_{ \dot B_{2,1}^{\frac{d}{2}-1} }\backsim&\, \|(a_0 ,u_0,z_0 ) \|_{\dot B_{2,1}^{\frac{d}{2}-1} },\\
\varepsilon\|(a_0^\varepsilon, z_0^\varepsilon) \|_{ \dot B_{2,1}^{\frac{d}{2}}}\backsim&\, \|(a_0 , z_0 ) \|_{ \dot B_{2,1}^{\frac{d}{2}} },\\
\|\gamma a_0^\varepsilon-z_0^\varepsilon\|_{ \dot B_{2,1}^{\frac{d}{2}-1}  }\backsim&\, \|\gamma a_0-z_0\|_{ \dot B_{2,1}^{\frac{d}{2}-1} },\\
\varepsilon\|\gamma a_0^\varepsilon-z_0^\varepsilon\|_{ \dot B_{2,1}^{\frac{d}{2} }  }\backsim&\, \|\gamma a_0-z_0\|_{ \dot B_{2,1}^{\frac{d}{2} } },
\end{align*}
and
\begin{align*}
\|(a^\varepsilon,u^\varepsilon,z^\varepsilon) \|_{\widetilde L_t^\infty(\dot B_{2,1}^{\frac{d}{2}-1})}\backsim&\, \|(a ,u,z ) \|_{\widetilde L_t^\infty(\dot B_{2,1}^{\frac{d}{2}-1})},\\
\varepsilon\|(a^\varepsilon, z^\varepsilon) \|_{\widetilde L_t^\infty(\dot B_{2,1}^{\frac{d}{2}-1})}\backsim&\, \|(a , z ) \|_{\widetilde L_t^\infty(\dot B_{2,1}^{\frac{d}{2}-1})},\\
\|\gamma a^\varepsilon-z^\varepsilon\|_{\widetilde L_t^\infty(\dot B_{2,1}^{\frac{d}{2}-1} )}\backsim&\, \|\gamma a-z\|_{\widetilde L_t^\infty(\dot B_{2,1}^{\frac{d}{2}-1} )} ,\\
\varepsilon\|\gamma a^\varepsilon-z^\varepsilon\|_{\widetilde L_t^\infty(\dot B_{2,1}^{\frac{d}{2} } )}\backsim&\, \|\gamma a-z\|_{\widetilde L_t^\infty(\dot B_{2,1}^{\frac{d}{2} } )} ,\\
 \|u^\varepsilon\|_{L_t^1(\dot B_{2,1}^{\frac{d}{2}+1})}\backsim&\, \|u \|_{L_t^1(\dot B_{2,1}^{\frac{d}{2}+1})},\\
\|z^\varepsilon\|_{L_t^1(\dot B_{2,1}^{\frac{d}{2}+1} )}^{\ell,\varepsilon}+\frac{1}{\varepsilon}\|z^\varepsilon\|_{L_t^1(\dot B_{2,1}^{\frac{d}{2}} )}^{h,\varepsilon} \backsim&\, \|z\|_{L_t^1(\dot B_{2,1}^{\frac{d}{2}+1}+\dot B_{2,1}^{\frac{d}{2}})}.
\end{align*}
Hence, we obtain that the solution $(a^\varepsilon,u^\varepsilon,z^\varepsilon)$ is uniformly bounded in the space  $\Xi^{\varepsilon}$  defined via the norm:
\begin{align*} 
\|(a^\varepsilon,u^\varepsilon,z^\varepsilon)\|_{\Xi ^\varepsilon}:= &\,   \|(a^\varepsilon,z^\varepsilon )\|_{\widetilde L_t^\infty (\dot B_{2,1}^{\frac{d}{2}-1} )}+\varepsilon\|(a^\varepsilon,z^\varepsilon )\|_{\widetilde L_t^\infty (\dot B_{2,1}^{\frac{d}{2}} )} +\|u^\varepsilon\|_{\widetilde L_t^\infty(\dot B_{2,1}^{\frac{d}{2}-1})} \nonumber\\
& \quad+\|u^\varepsilon\|_{L_t^1(\dot B_{2,1}^{\frac{d}{2}+1})}+\|z^\varepsilon\|_{L_t^1(\dot B_{2,1}^{\frac{d}{2}+1} )}^{\ell,\varepsilon}+\frac{1}{\varepsilon}\|z^\varepsilon\|_{L_t^1(\dot B_{2,1}^{\frac{d}{2}} )}^{h,\varepsilon}. 
\end{align*}
Therefore, under the condition \eqref{TD.1}, we obtain the uniform estimate \eqref{TD.2}, and it follows that
\begin{align}\label{G4.3}
\|(a^\varepsilon,u^\varepsilon,z^\varepsilon)\|_{\Xi ^\varepsilon} \lesssim \|(a_0^\varepsilon,z^\varepsilon_0)\|_{\dot B_{2,1}^{\frac{d}{2}-1} }+   \|u_0^\varepsilon\|_{\dot B_{2,1}^{\frac{d}{2}-1}}+ \varepsilon\|(a_0^\varepsilon,z^\varepsilon_0)\|_{\dot B_{2,1}^{\frac{d}{2} } },    \end{align}
as stated in \eqref{TD.3}.
By  adapting the compactness arguments of Lions and Masmoudi in \cite{LM-CRASPSM-1999} and utilizing   \eqref{TD.2}--\eqref{TD.3}, we justify the low Mach number limit of the NST system \eqref{A4} in the sense of distributions below.

\begin{proof}[Proof of Theorem \ref{T1.4}]
Consider a family $(a_0^\varepsilon, u_0^\varepsilon, z_0^\varepsilon)$ of initial data  satisfying \eqref{TD.1} and such that  $\mathbb{P} u_0^\varepsilon \rightharpoonup v_0$ as $\varepsilon \to 0$. Since
\begin{align*}
\|(a_0^\varepsilon, z_0^\varepsilon)\|_{\dot B_{2,1}^{\frac{d}{2}-1}}^{h,\eps} \lesssim \varepsilon \|(a_0^\varepsilon, z_0^\varepsilon)\|_{\dot B_{2,1}^{\frac{d}{2}}}^{h,\eps},
\end{align*}
it follows that the sequence $(a_0^\varepsilon, u_0^\varepsilon, z_0^\varepsilon)$ is uniformly bounded in $\dot B_{2,1}^{\frac{d}{2}-1} \times \dot B_{2,1}^{\frac{d}{2}-1} \times \dot B_{2,1}^{\frac{d}{2}-1}$. Furthermore, \eqref{G4.3} implies that the corresponding solutions $(a^\varepsilon, u^\varepsilon, z^\varepsilon)$ are bounded in $\mathcal{C}_b(\mathbb{R}^+; \dot B_{2,1}^{\frac{d}{2}-1})$. Therefore, there exists a sequence $\{\varepsilon_n\}_{n \in \mathbb{N}}$ tending to $0$ such that
\begin{align}\label{G4.4}
(a^{\varepsilon_n}_0, u^{\varepsilon_n}_0, z^{\varepsilon_n}_0) &\rightharpoonup (a, v, z) \quad \text{weakly in}\quad \dot B_{2,1}^{\frac{d}{2}-1}, \\ \label{G4.5}
(a^{\varepsilon_n}, u^{\varepsilon_n}, z^{\varepsilon_n}) &\rightharpoonup (a, v, z) 
\quad  \text{weakly-$*$ in}\quad L^\infty(\mathbb{R}^+; \dot B_{2,1}^{\frac{d}{2}-1}),\quad 
\end{align}
 \text{as}  $\varepsilon_n \to 0$. Certainly, from \eqref{G4.4}, it follows that $\mathbb P u_0 = v_0$.

By leveraging the interpolation \eqref{A.1}, one has
\begin{align*}
 \|z^{\varepsilon_n}\|_{L^2(\mathbb R^+; \dot B_{2,1}^{\frac{d}{2}})} \lesssim&\,  \Big(  \|z^{\varepsilon_n}\|_{L^\infty(\mathbb R^+; \dot B_{2,1}^{\frac{d}{2}-1})}^{\ell,\varepsilon_n}  \Big)^{\frac{1}{2}}  \Big( \|z^{\varepsilon_n}\|_{L^1(\mathbb R^+; \dot B_{2,1}^{\frac{d}{2}+1})}^{\ell,\varepsilon_n}\Big)^{\frac{1}{2}}\nonumber\\ &+\Big(\varepsilon_n\|z^{\varepsilon_n}\|_{L^\infty(\mathbb R^+; \dot B_{2,1}^{\frac{d}{2}})}^{h,\varepsilon_n}\Big)\Big( \frac{1}{\varepsilon_n}    \|z^{\varepsilon_n}\|_{L^1(\mathbb R^+; \dot B_{2,1}^{\frac{d}{2}})}^{h,\varepsilon_n}\Big)^\frac{1}{2}\nonumber\\
 \lesssim&\, \|(a^{\varepsilon_n},u^{\varepsilon_n},z^{\varepsilon_n})\|_{\Xi ^{\varepsilon_n}}, 
\end{align*}
which implies that $\{z^{{\varepsilon_n}}\}$ is bounded in $L^2(\mathbb R^+; \dot B_{2,1}^{\frac{d}{2}})$.   Notice that $P^{\varepsilon_n} = 1 + \varepsilon_n z^{\varepsilon_n}$, we directly get $P^{\varepsilon_n} \to 1$ strongly as $\varepsilon_n \to 0$. To justify that ${\rm div}\,u=0$, we rewrite \eqref{A4}$_3$ as  
\begin{align}  \label{G4.6}
\gamma{\rm div}\, u^{\varepsilon_n} = -\varepsilon_n {\rm div}\,(z^{\varepsilon_n} u^{\varepsilon_n})-\varepsilon_n(\gamma-1)z^{\varepsilon_n}{\rm div}\,u^{\varepsilon_n} - \varepsilon_n \partial_t z^{\varepsilon_n},
\end{align}  
where $\gamma>1$.
Because 
\begin{align*}
 \|u^{\varepsilon_n}\|_{L^2(\mathbb R^+; \dot B_{2,1}^{\frac{d}{2}})} \lesssim&\,   \|u^{\varepsilon_n}\|_{L^\infty(\mathbb R^+; \dot B_{2,1}^{\frac{d}{2}-1})}    ^{\frac{1}{2}}   \|u^{\varepsilon_n}\|_{L^1(\mathbb R^+; \dot B_{2,1}^{\frac{d}{2}+1})}  ^{\frac{1}{2}} \lesssim  \|(a^{\varepsilon_n},u^{\varepsilon_n},z^{\varepsilon_n})\|_{\Xi ^{\varepsilon_n}} ,  
\end{align*}
the $u^{\varepsilon_n}$ is bounded in $L^2(\mathbb R^+; \dot B_{2,1}^{\frac{d}{2}})$. Then, the first term and second term on the right-hand side of \eqref{G4.6} is $\mathcal{O}(\varepsilon_n)$ in $L^1(\mathbb R^+;\dot B_{2,1}^{\frac{d}{2}-1})$. Regarding the last term in \eqref{G4.6}, it converges to $0$ in the sense of distributions, thanks to \eqref{G4.5}. Consequently, ${\rm div}\,u^{\varepsilon_n}\rightharpoonup 0$ as $\varepsilon_n \to 0$, which implies that  ${\rm div}\,v=0$. Furthermore, \eqref{A5}$_3$ is satisfied and $\mathbb Qu^{\varepsilon_n}\rightharpoonup 0$ as $\varepsilon_n \to 0$.

Although \eqref{A4}$_1$ is formally identical to the transport equation in \eqref{A4}$_3$, the situation for \eqref{A4}$_1$ is fundamentally different. Specifically, $a^{\varepsilon_n}$ lacks a dissipative structure (i.e., $a^{\varepsilon_n} \notin L^2(\mathbb R^+; \dot B_{2,1}^{\frac{d}{2}})$), which implies that ${\rm div}\,(a^{\varepsilon_n} u^{\varepsilon_n}) \notin L^1(\mathbb R^+; \dot B_{2,1}^{\frac{d}{2}-1})$. As a consequence, the previous argument cannot be applied to establish the convergence of \eqref{A4}$_1$ in the sense of distributions. To overcome this difficulty, we consider the new variable $\omega^{\varepsilon_n} := \gamma a^{\varepsilon_n} - z^{\varepsilon_n}$, derived from \eqref{A4}$_1$ and \eqref{A4}$_3$, which satisfies
\begin{align}\label{G4.7}
\partial_t \omega^{\varepsilon_n} + {\rm div}\,(\omega^{\varepsilon_n }u^{\varepsilon_n})-(\gamma-1)z^{\varepsilon_n}{\rm div}\,u^{\varepsilon_n}= 0.
\end{align}
Similar to \cite[Section 6]{LSZ-2025-arXiv}, we only require the strong compactness properties of $\omega^{\varepsilon_n}$ and $u^{\varepsilon_n}$. More precisely, since $u^{\varepsilon_n}$ is uniformly bounded in $L_t^2(\dot B_{2,1}^{\frac{d}{2}})\hookrightarrow L_t^2(L^\infty)$, by the uniqueness of limits, we conclude that  
\begin{equation}\label{G4.8}   
\left\{  
\begin{aligned}  
& u^{\varepsilon_n}\rightharpoonup v\quad\text{weakly-$*$ in}\quad L^2(\mathbb R^+,L^\infty)\quad  \text{for}\quad d\geq3,\\
& u^{\varepsilon_n}\rightharpoonup v\quad \text{weakly in}\quad L^2(\mathbb R^+,\dot H^1)\quad   \text{for}\quad d=2,
\end{aligned}  
\right.  
\end{equation} 
 \text{as} $ \varepsilon_n \to 0$. %  \quad\text{as} \quad\varepsilon_n \to 0\quad
On the other hand, from \eqref{TD.2} and \eqref{G4.7}, we observe that $\omega^{\varepsilon_n}$ and $\partial_t\omega^{\varepsilon_n}$ are uniformly bounded in $L_t^\infty(\dot B_{2,1}^{\frac{d}{2}-1})$ and $L_t^2(\dot B_{2,1}^{\frac{d}{2}-2})$, respectively. By virtue of the compact embeddings $\dot B_{2,1}^{\frac{d}{2}-1} \hookrightarrow L_{\rm loc}^2$ for $d \geq 3$ and $\dot B_{2,1}^{\frac{d}{2}-1} \hookrightarrow H_{\rm loc}^{-1}$ for $d = 2$, and by applying the Aubin–Lions lemma together with the Cantor diagonal argument, there exists a limit function $\varrho$ such that
\begin{equation}\label{G4.9}   
\left\{  
\begin{aligned}  
& \omega^{\varepsilon_n}\rightarrow\varrho\quad\text{strongly in}\quad \mathcal{C}([0,T],L^2_{\rm loc})\quad   \text{for}\quad d\geq3,\\
& \omega^{\varepsilon_n}\rightarrow \varrho \quad\text{strongly in}\quad \mathcal{C}([0,T],H^{-1}_{\rm loc})\,\,\,\,\, \text{for}\quad d=2,
\end{aligned}  
\right.  
\end{equation} 
as $\varepsilon_n \to 0 $ 
 for any finite $T>0$. By combining \eqref{G4.7} with \eqref{G4.8}--\eqref{G4.9}, we establish the validity of \eqref{A5}$_1$ in the sense of distributions. By the uniqueness of the limits, together with \eqref{G4.4} and \eqref{G4.5}, it follows that $\varrho = \gamma a - z$ and $\varrho_0 = \gamma a_0 - z_0$.

For the remaining convergence of \eqref{A4}$_2$ toward \eqref{A5}$_2$, note that $a^{\varepsilon_n}$ lacks a dissipative structure in this case, which differs from the settings in \cite{Danchin-2002} on the isentropic Navier-Stokes equations. Therefore, a new method needs to be developed to address the challenges arising from the absence of such a dissipative structure.
Inspired by  \cite[Section 6]{LSZ-2025-arXiv}, we apply the operator $\mathbb P$ to \eqref{A4}$_2$ to obtain
\begin{align}\label{G4.10}
\partial_t \mathbb P m^{\varepsilon_n}-\bar\mu\Delta\mathbb P m^{\varepsilon_n}+\mathbb P{\rm div}\, (m^{\varepsilon_n}\otimes u^{\varepsilon_n})=0,   
\end{align}
where the momentum $m^{\varepsilon_n}:=(1+\varepsilon_n a^{\varepsilon_n})u^{\varepsilon_n}$.  From \eqref{G4.5} and the boundedness of $u^{\varepsilon_n}$ in $L^2(\mathbb R^+,\dot B_{2,1}^{\frac{d}{2}})$, it follows that $\mathbb P m^{\varepsilon_n}$ and $\partial_t \mathbb P m^{\varepsilon_n}$ are uniformly bounded in $L^2(\mathbb R^+,\dot B_{2,1}^{\frac{d}{2}})$ and $L^1(\mathbb R^+,\dot B_{2,1}^{\frac{d}{2}-1})$, respectively. By the Aubin-Lions lemma, we obtain  
\begin{align}\label{G4.11}
\mathbb P m^{\varepsilon_n} \to v'\quad  \text{strongly in} \quad L^2(0,T;L_{\rm loc}^2),
\end{align}
as $\varepsilon_n \to 0$,  for any $T > 0$.
To show that $v^\prime = v$, we rewrite $\mathbb{P} m^{\varepsilon_n}$ as  
\begin{align}\label{G4.12}
 \mathbb P m^{\varepsilon_n}=\mathbb P u^{\varepsilon_n}+\varepsilon_n  \mathbb{P}(a^{\varepsilon_n}u^{\varepsilon_n}).   
\end{align}
From the boundedness of the $\mathbb P$ operator and the fact that $ a^{\varepsilon_{n}}u^{\varepsilon_n} $ is uniformly bounded in $L^2(0,T;\dot B_{2,1}^{\frac{d}{2}-1})$, it follows that $\varepsilon_n\mathbb P(a^{\varepsilon_{n}}u^{\varepsilon_n})$ 
is $\mathcal O(\varepsilon_n)$ in $L^2(0,T;\dot B_{2,1}^{\frac{d}{2}-1})$. By exploiting the compact embedding $\dot B_{2,1}^{\frac{d}{2}-1}(\mathbb{R}^d) \hookrightarrow L_{\rm loc}^2(\mathbb{R}^d)$ for $d \geq 3$ and the continuous embedding $\dot B_{2,1}^{\frac{d}{2}-1}(\mathbb{R}^d) \hookrightarrow L_{\rm loc}^{2}(\mathbb{R}^d)$ for $d = 2$, we conclude that
\begin{align}\label{G4.13}
\varepsilon_n \mathbb{P}(a^{\varepsilon_n}u^{\varepsilon_n}) \to 0\quad\text{strongly in} \quad L^2(0,T;L_{\rm loc}^2),  \quad  \text{as} \quad \varepsilon_n \to 0.
\end{align}
Due to the uniqueness of the limits, it holds that
\begin{align}\label{strongly}
 \mathbb P u^{\varepsilon_n}\rightarrow v\quad \text{strongly in} \quad L^2(0,T;L_{\rm loc}^2), \quad \text{as} \quad \varepsilon_n \to 0,
\end{align}
which together with \eqref{G4.11}, \eqref{G4.12} and \eqref{G4.13} gives rise to 
\begin{align*} 
\mathbb P m^{\varepsilon_n} \to v\quad \text{strongly in} \quad L^2(0,T;L_{\rm loc}^2), \quad \text{as} \quad \varepsilon_n \to 0,
\end{align*}
and $v^\prime=v$.

In the end, we are left only with the proof that $\mathbb P{\rm div}\,(m^{\varepsilon_n} \otimes u^{\varepsilon_n}) \rightharpoonup \mathbb P(v \cdot \nabla v)$  as $\varepsilon_n \to 0$. We decompose $\mathbb P m^{\varepsilon_n}$ into
\begin{align}\label{G4.15}
 \mathbb P{\rm div}\,(m^{\varepsilon_n} \otimes u^{\varepsilon_n})= \mathbb P(u^{\varepsilon_n}\cdot\nabla u^{\varepsilon_n})+\varepsilon_n\mathbb P (a^{\varepsilon_n}u^{\varepsilon_n}\otimes u^{\varepsilon_n})+\mathbb P({\rm div}\,{u}^{\varepsilon_n} u^{\varepsilon_n}).
\end{align}
The first term on the right-hand side of \eqref{G4.15} can be treated 
similarly to the approach outlined in \cite[Section 4]{DH-MATHANN-2016} and \cite[Section 5]{Danchin-2018}. We observe that
\begin{align*} 
u^{\varepsilon_n} \cdot \nabla u^{\varepsilon_n} = \frac{1}{2} \nabla |\mathbb{Q} u^{\varepsilon_n}|^2 + \mathbb{P} u^{\varepsilon_n} \cdot \nabla u^{\varepsilon_n} + \mathbb{Q} u^{\varepsilon_n} \cdot \nabla \mathbb{P} u^{\varepsilon_n}.
\end{align*}
Projecting the first term onto divergence-free vector fields yields $0$, and since ${\rm div}\,v=0$, we have $\mathbb P v = v$. Therefore, it suffices to show
\begin{align} \label{G4.16}
\mathbb P( \mathbb{P} u^{\varepsilon_n} \cdot \nabla u^{\varepsilon_n}) \rightharpoonup \mathbb P(  v\cdot\nabla v)=\mathbb P(\mathbb Pv\cdot\nabla v),\quad  \mathbb P(   \mathbb{Q} u^{\varepsilon_n} \cdot \nabla \mathbb{P} u^{\varepsilon_n})\rightharpoonup 0\quad \text{as} \quad \varepsilon_n \to 0.
\end{align}
By utilizing \eqref{G4.8} and \eqref{strongly}, we   derive \eqref{G4.16}, which implies the weak convergence  
\begin{align}\label{G4.17}
 \mathbb P(  u^{\varepsilon_n} \cdot \nabla u^{\varepsilon_n}) \rightharpoonup    \mathbb P(v\cdot\nabla v)\quad \text{as} \quad \varepsilon_n \to 0.
\end{align}
Thanks to the uniform regularity estimates for $a^{\varepsilon_n}$ and $u^{\varepsilon_n}$, we  directly conclude that
\begin{align}\label{G4.18}
 \varepsilon_n \mathbb{P}(a^{\varepsilon_n}
 u^{\varepsilon_n}\otimes u^{\varepsilon_n}) \to 0 \quad \text{as} \quad \varepsilon_n \to 0.   
\end{align}
Combining \eqref{G4.15} with \eqref{G4.17} and \eqref{G4.18}, it suffices to show that $\mathbb P({\rm div}\,{u}^{\varepsilon_n} u^{\varepsilon_n})\rightharpoonup 0$ as $\varepsilon_n \to 0$, which hinges on establishing that ${\rm div}\,{u}^{\varepsilon_n} \rightarrow 0$. To this
end, we aim to establish the dispersive estimates of $z^{\varepsilon_n}$ and $\mathbb Qu^{\varepsilon_n}$ as $\varepsilon_n \to 0$. Motivated by \cite{Danchin-2002} and \cite[Section 6]{LSZ-2025-arXiv}, we consider the transformation $q^{\varepsilon_n} := \Lambda^{-1}{\rm div}\,\mathbb{Q}u^{\varepsilon_n}$ and rewrite the equations \eqref{A4}$_2$--\eqref{A4}$_3$ as
\begin{equation}\label{G4.19}   
\left\{  
\begin{aligned}  
&  \partial_tq^{\varepsilon_n}-\frac{1}{\varepsilon_n} \Lambda z^{\varepsilon_n}=G_1^{\varepsilon_n} ,\\
&\frac{1}{\gamma}\partial_t z^{\varepsilon_n}+\frac{1}{ \varepsilon_n}\Lambda q^{\varepsilon_n}= G_2^{\varepsilon_n},\\
&(q^{\varepsilon_n},z^{\varepsilon_n})|_{t=0}=(\Lambda^{-1}{\rm div}\,\mathbb Q u_0^{\varepsilon_n},z_0^{\varepsilon_n}),
\end{aligned}  
\right.  
\end{equation} 
with
\begin{equation*}    
\left\{  
\begin{aligned}  
 G_1^{\varepsilon_n}:=& \,\Delta q^{\varepsilon_n}-\Lambda^{-1}{\rm div}\,\Big(u^{\varepsilon_n}\cdot\nabla u^{\varepsilon_n}+\frac{\varepsilon_na^{\varepsilon_n}}{1+\varepsilon_na^{\varepsilon_n}}\big[\bar\mu \Delta u^{\varepsilon_n}+(\bar\mu+\bar\lambda)\nabla{\rm div}\,u^{\varepsilon_n}\big]\Big)\nonumber\\
&-\Lambda^{-1}{\rm div}\,\Big(\frac{a^{\varepsilon_n}}{1+\varepsilon_na^{\varepsilon_n}}\nabla z^{\varepsilon_n}\Big),\\
G_2^{\varepsilon_n}:=&\,- \frac{1}{\gamma}{\rm div}\,(z^{\varepsilon_n} u^{\varepsilon_n})-\frac{(\gamma-1)}{\gamma}z^{\varepsilon_n}{\rm div}\,u^{\varepsilon_n}.
\end{aligned}  
\right.  
\end{equation*} 

For the case $d\geq3$, by applying the dispersive estimates from Lemma \ref{LA.8} to the system \eqref{G4.19}, we   derive that 
\begin{align*}
&\|(q^{\varepsilon_n},z^{\varepsilon_n})\|_{L^{\frac{2p}{p-2}}_t(\dot B_{p,1}^{\frac{d-1}{p}-\frac{1}{2}})}\nonumber\\
\lesssim&\, \varepsilon_n^{\frac{p-2}{2p}} \|(q^{\varepsilon_n}_0,z^{\varepsilon_n}_0)\|_{\dot B_{2,1}^{\frac{d}{2}-1}}+   \varepsilon^{\frac{p-2}{2p}} \| (G_1^{\varepsilon_n},G_2^{\varepsilon_n})\|_{L^1_t(\dot B_{2,1}^{\frac{d}{2}-1})}\nonumber\\
\lesssim&\,\varepsilon_n^{\frac{p-2}{2p}}\Big( \|(u^{\varepsilon_n}_0,z^{\varepsilon_n}_0)\|_{\dot B_{2,1}^{\frac{d}{2}-1}}+\|a^{\varepsilon_n}\|_{\widetilde L_t^\infty(\dot B_{2,1}^{\frac{d}{2}-1})} \|u^{\varepsilon_n}\|_{L_t^1(\dot B_{2,1}^{\frac{d}{2}+1})} +\|a^{\varepsilon_n}\|_{\widetilde L_t^\infty(\dot B_{2,1}^{\frac{d}{2}-1})}\|z^{\varepsilon_n}\|_{L_t^1(\dot B_{2,1}^{\frac{d}{2}+1})}^{\ell,\varepsilon_n} \nonumber\\
&+\|a^{\varepsilon_n}\|_{\widetilde L_t^\infty(\dot B_{2,1}^{\frac{d}{2}})}\|z^{\varepsilon_n}\|_{L_t^1(\dot B_{2,1}^{\frac{d}{2} })}^{h,\varepsilon_n} + \|u^{\varepsilon_n}\|_{\widetilde L_t^2(\dot B_{2,1}^{\frac{d}{2}})}^2+\|z^{\varepsilon_n}\|_{L^2(\mathbb R^+; \dot B_{2,1}^{\frac{d}{2}})}\|u^{\varepsilon_n}\|_{L^2(\mathbb R^+; \dot B_{2,1}^{\frac{d}{2}})}\Big)\nonumber\\
\lesssim&\, \varepsilon_n^{\frac{p-2}{2p}}\Big(  \|(u^{\varepsilon_n}_0,z^{\varepsilon_n}_0)\|_{\dot B_{2,1}^{\frac{d}{2}-1}}+\|(a^{\varepsilon_n},
u^{\varepsilon_n},z^{\varepsilon_n})\|_{\Xi ^{\varepsilon_n}}^2\Big)\nonumber\\ 
\lesssim&\, \varepsilon_n^{\frac{p-2}{2p}},
\end{align*}
for any $p \in( 2,\infty)$ and $t > 0$.
Therefore, we further get
\begin{align}\label{G4.20}
\|(\mathbb Q u^{\varepsilon_n},z^{\varepsilon_n})\|_{L^{\frac{2p}{p-2}}_t(\dot B_{p,1}^{\frac{d-1}{p}-\frac{1}{2}})} 
\lesssim \varepsilon_n^{\frac{p-2}{2p}},   
\end{align}
for  any $p \in(2,\infty)$ and $t > 0$.
By using \eqref{TD.3}, \eqref{G4.20} and \eqref{A.1}, one has
\begin{align}\label{G4.21}
&\|(\mathbb Q u^{\varepsilon_n},z^{\varepsilon_n})^{\ell,\varepsilon_n} \|_{{L_t^2}(\dot B_{q,1}^{\frac{d+1}{q}-\frac{1}{2}})}\nonumber\\
&\quad \lesssim \Big(\|(\mathbb Q u^{\varepsilon_n},z^{\varepsilon_n}) \|^{\ell,\varepsilon_n}_{{L_t^{\frac{2p}{p-2}}}(\dot B_{p,1}^{\frac{d-1}{p}-\frac{1}{2}})}\Big)^{\frac{p}{p+2}}   \Big(\|(\mathbb Q u^{\varepsilon_n},z^{\varepsilon_n}) \|^{\ell,\varepsilon_n}_{{L_t^1}(\dot B_{2,1}^{ \frac{d}{2}+1})}\Big)^{\frac{2}{p+2}}  
\lesssim \varepsilon_n^{\frac{q-2}{2q} },
\end{align}
for any $q=\frac{p+2}{2}\in (2,\infty)$ and $t>0$.
Consequently, we derive  from \eqref{TD.3} and \eqref{G4.21} that
\begin{align*}
\|(\mathbb Q u^{\varepsilon_n},z^{\varepsilon_n})  \|_{{L_t^2}(\dot B_{q,1}^{\frac{d+1}{q}-\frac{1}{2}})}   \lesssim&\,  \|(\mathbb Q u^{\varepsilon_n},z^{\varepsilon_n}) ^{\ell,\varepsilon_n} \|_{{L_t^2}(\dot B_{q,1}^{\frac{d+1}{q}-\frac{1}{2}})}+  \|(\mathbb Q u^{\varepsilon_n},z^{\varepsilon_n})  \|_{{L_t^2}(\dot B_{q,1}^{\frac{d+1}{q}-\frac{1}{2}})}^{h,\varepsilon_n}\nonumber\\
\lesssim&\, \|(\mathbb Q u^{\varepsilon_n},z^{\varepsilon_n})^{\ell,\varepsilon_n}  \|_{{L_t^2}(\dot B_{q,1}^{\frac{d+1}{q}-\frac{1}{2}})} +\varepsilon_n^{\frac{q-2}{2q}}\|(\mathbb Q u^{\varepsilon_n},z^{\varepsilon_n})  \|_{{L_t^2}(\dot B_{2,1}^{\frac{d }{2}})}^{h,\varepsilon_n}\nonumber\\
\lesssim&\,\varepsilon_n^{\frac{q-2}{2q} },
\end{align*}
for any $q\in(2,\infty)$ and $t>0$.

For the case $d=2$, by taking $r = \frac{4p}{p-2}$ in Lemma \ref{LA.8} for any $p \in (2, \infty]$, we obtain
\begin{align*}
 \|(\mathbb Q u^{\varepsilon_n},z^{\varepsilon_n})\|_{L^{\frac{4p}{p-2}}_t(\dot B_{p,1}^{\frac{3}{2p}-\frac{3}{4}})} 
\lesssim&\, \varepsilon_n^{\frac{p-2}{4p}} \|(q^{\varepsilon_n}_0,z^{\varepsilon_n}_0)\|_{\dot B_{2,1}^{\frac{d}{2}-1}}+   \varepsilon^{\frac{p-2}{4p}} \| (G_1^{\varepsilon_n},G_2^{\varepsilon_n})\|_{L^1_t(\dot B_{2,1}^{\frac{d}{2}-1})}\nonumber\\
 \lesssim&\, \varepsilon_n^{\frac{p-2}{4p}}\Big(  \|(u^{\varepsilon_n}_0,z^{\varepsilon_n}_0)\|_{\dot B_{2,1}^{\frac{d}{2}-1}}+\|(a^{\varepsilon_n},
 u^{\varepsilon_n},z^{\varepsilon_n})\|_{\Xi ^{\varepsilon_n}}^2\Big)\nonumber\\ 
\lesssim&\,   \varepsilon_n^{\frac{p-2}{4p}}, 
\end{align*}
which together with \eqref{A.1} gives
\begin{align*}
 &\|(\mathbb Q u^{\varepsilon_n},z^{\varepsilon_n})^{\ell,\varepsilon_n}\|_{L^{2}_t(\dot B_{q,1}^{\frac{5}{2q}-\frac{1}{4}})}\nonumber\\
 &\quad \lesssim \Big( \|(\mathbb Q u^{\varepsilon_n},z^{\varepsilon_n})^{\ell,\varepsilon_n}\|_{L^{\frac{4p}{p-2}}_t(\dot B_{p,1}^{\frac{3}{2p}-\frac{3}{4}})}\Big)^{\frac{2p}{3p+2}}
  \Big(\|(\mathbb Q u^{\varepsilon_n},z^{\varepsilon_n}) \|^{\ell,\varepsilon_n}_{{L_t^1}(\dot B_{2,1}^{ \frac{d}{2}+1})}\Big)^{\frac{p+2}{3p+2}} 
 \lesssim \varepsilon_n^{\frac{q-2}{4q}},
\end{align*}
for any $q=\frac{6p+4}{p+6}\in (2,6)$ and $t>0$.
Thus, we have
\begin{align*}
 &\|(\mathbb Q u^{\varepsilon_n},z^{\varepsilon_n}) \|_{L^{2}_t(\dot B_{q,1}^{\frac{5}{2q}-\frac{1}{4}})} 
  \lesssim   \|(\mathbb Q u^{\varepsilon_n},z^{\varepsilon_n})^{\ell,\varepsilon_n}\|_{L^{2}_t(\dot B_{q,1}^{\frac{5}{2q}-\frac{1}{4}})}+\varepsilon_n^{\frac{q-2}{4q}} \|(\mathbb Q u^{\varepsilon_n},z^{\varepsilon_n})\|_{L^{2}_t(\dot B_{2,1}^{1 })}^{h,\varepsilon_n} 
 \lesssim \varepsilon_n^{\frac{q-2}{4q}},
\end{align*}
for any $q\in (2,6)$ and $t>0$. Hence, we have   established the estimate \eqref{TD.4}. Furthermore, by combining \eqref{G4.15}, \eqref{G4.17}, and \eqref{G4.18}, we deduce that
\begin{align*}
    \mathbb P{\rm div}\,(m^{\varepsilon_n} \otimes u^{\varepsilon_n}) \rightharpoonup \mathbb P(v\cdot\nabla v) \quad \text{as} \quad \varepsilon_n \to 0.
\end{align*}
Therefore, the equation \eqref{A5}$_2$ holds in the sense of distributions, and the proof of Theorem \ref{T1.4} is completed.
\end{proof}

\section{Optimal time decay rates  of  the compressible NST system \eqref{A3} }

\subsection{The upper bound estimate of solutions under the boundedness condition}
In this subsection, we aim to prove Theorem \ref{T1.5}. To begin with, we establish the upper bound of the decay estimates of $(u,z)$. First, in the case where the initial data  $ (u_0^\ell,z_0^\ell)\in{\dot B_{2,\infty}^{\sigma_0}}$, we analyze the propagation of the $\dot B_{2,\infty}^{\sigma_0}$ norm of $(u,z)$ in the low-frequency regime. For convenience, we define 
\begin{align}\label{G5.1}
 \mathcal{E}(t):=\|(a,z)\|_{\widetilde L_t^\infty (\dot B_{2,1}^{\frac{d}{2}-1}\cap B_{2,1}^{\frac{d}{2}})}+\|u\|_{\widetilde L_t^\infty(\dot B_{2,1}^{\frac{d}{2}-1})} +\|u\|_{L_t^1(\dot B_{2,1}^{\frac{d}{2}+1})}+\|z\|_{L_t^1(\dot B_{2,1}^{\frac{d}{2}+1}+\dot B_{2,1}^{\frac{d}{2}})}.    
\end{align}
\begin{lem}\label{L5.1}
Let $(a, u, z)$ be the global strong solution to the Cauchy problem \eqref{A3}--\eqref{A3-1} established in  Theorem \ref{T1.2}. Then, under the assumptions of Theorem \ref{T1.5}, we have  
\begin{align}\label{G5.2}
\mathcal{E}_{\ell,\sigma_0}(t) := \|(u,z)\|^\ell_{\widetilde L_t^\infty(\dot B_{2,\infty}^{\sigma_0})} + \|(u,z)\|^\ell_{L_t^1(\dot B_{2,\infty}^{\sigma_0+2})} \leq C\delta_{*},
\end{align}
for all $t > 0$, where $\delta_*$ is defined in \eqref{TE.4} and $C > 0$ is a constant independent of $t$.
\end{lem}
\begin{proof}
It follows from \eqref{G3.7}--\eqref{G3.12} that the following estimates can be derived:
\begin{align}\label{G5.3}
&\|\dot\Delta_k (u,z)\|_{L_t^\infty(L^2)}+2^{2k}\|\dot \Delta_k(u,z)\|_{L_t^1(L^2)}\nonumber\\
&\quad \lesssim     \|\dot\Delta_k (u_0,z_0)\|_{ L^2}+\|\dot\Delta_k(u\cdot\nabla u)\|_{L_t^1(L^2)}+\| \dot\Delta_k(  f(a)  \nabla{\rm div}\,u   )  \|_{L_t^1(L^2 )}\nonumber\\
&\qquad +\| \dot\Delta_k(  f(a) \Delta u   )  \|_{L_t^1(L^2 )}+\|\dot\Delta_k(f(a)\nabla z)\|_{L_t^1(L^2)}\nonumber\\
&\qquad +2^k\|\dot \Delta_k(z u)\|_{L_t^1(L^2)}+\|\dot\Delta_k(z{\rm div}\,u)\|_{L_t^1(L^2)},
\end{align}
for any $k \leq 0$, where we have   utilized the fact that $2^k \lesssim 1$ when $k \leq 0$.
Multiplying \eqref{G5.3} by $2^{k\sigma_0}$ and taking the supremum on  both $[0,t]$ and   $k \leq 0$, we have
\begin{align}\label{G5.4}
&\|(u,z)\|_{\widetilde L_t^\infty(\dot B_{2,\infty}^{ \sigma_0})}^\ell+\|(u,z)\|_{L_t^1(\dot B_{2,\infty}^{ \sigma_0+2})}^\ell\nonumber\\
&\quad\lesssim  \|(u_0,z_0)\|_{\dot B_{2,\infty}^{ \sigma_0}}^\ell +\|u\cdot\nabla u\|_{L_t^1(\dot B_{2,\infty}^{ \sigma_0})}^\ell +  \|f(a)\nabla{\rm div}\, u\|_{L_t^1(\dot B_{2,\infty}^{ \sigma_0})}^\ell\nonumber\\
&\quad\quad+\|f(a)\Delta u\|_{L_t^1(\dot B_{2,\infty}^{ \sigma_0})}^\ell+\|f(a)\nabla z\|_{L_t^1(\dot B_{2,\infty}^{ \sigma_0})}^\ell+\|zu\|_{L_t^1(\dot B_{2,\infty}^{ \sigma_0+1})}^\ell+\|z{\rm div}\,u\|_{L_t^1(\dot B_{2,\infty}^{ \sigma_0})}^\ell.
\end{align}
From Lemmas \ref{LA.2}--\ref{LA.4}, we have
\begin{align}\label{G5.5}
 \|u\cdot\nabla u\|_{L_t^1(\dot B_{2,\infty}^{ \sigma_0})} \lesssim&\, \|u\|_{ L_t^2(\dot B_{2,1}^{\frac{d}{2}})} \|u\|_{L_t^2(\dot B_{2,\infty}^{\sigma_0+1})}\nonumber\\
 \lesssim&\,    \|u\|_{ \widetilde L_t^\infty(\dot B_{2,1}^{\frac{d}{2}-1})}^{\frac{1}{2}}\|u\|_{ L_t^1(\dot B_{2,1}^{\frac{d}{2}+1})}^{\frac{1}{2}} \|u\|_{ \widetilde L_t^\infty(\dot B_{2,\infty}^{\sigma_0})}^{\frac{1}{2}} \|u\|_{ L_t^1(\dot B_{2,\infty}^{ \sigma_0+2})}^{\frac{1}{2}} \nonumber\\
 \lesssim&\, \mathcal{E}(t)\Big(\|u^\ell\|_{ \widetilde L_t^\infty(\dot B_{2,\infty}^{\sigma_0})}  +  \|u^h\|_{ \widetilde L_t^\infty(\dot B_{2,1}^{\frac{d}{2}-1})}  +\|u^\ell\|_{ L_t^1(\dot B_{2,\infty}^{ \sigma_0+2})} +\|u^h\|_{ L_t^1(\dot B_{2,1}^{ \frac{d}{2}+1})}     \Big)\nonumber\\
 \lesssim&\, \mathcal{E}(t)\big( \mathcal{E}_{\ell,\sigma_0}(t)+  \mathcal{E}(t)     \big).
\end{align}
Similarly, for the remaining coupling terms associated with the density $a$, 
where no dissipative structure exists, we derive
\begin{align}\label{G5.6}
&\, \|f(a)\nabla{\rm div}\, u\|_{L_t^1(\dot B_{2,\infty}^{ \sigma_0})}+\|f(a)\Delta u\|_{L_t^1(\dot B_{2,\infty}^{ \sigma_0})} \nonumber\\
\lesssim&\, \|f(a)\|_{\widetilde L_t^\infty(\dot B_{2,1}^\frac{d}{2})}\|u\|_{L_t^1(\dot B_{2,\infty}^{\sigma_0+2})}\nonumber\\
\lesssim&\, \|a\|_{\widetilde L_t^\infty(\dot B_{2,1}^\frac{d}{2})}\Big( \|u^\ell\|_{L_t^1(\dot B_{2,\infty}^{\sigma_0+2})}+\|u^h\|_{L_t^1(\dot B_{2,1}^{\frac{d}{2}+1})} \Big)\nonumber\\
\lesssim&\,    \mathcal{E}(t)\big( \mathcal{E}_{\ell,\sigma_0}(t)+  \mathcal{E}(t)     \big).
\end{align}
Through a direct calculation, one has
\begin{align}\label{G5.7}
&\|zu\|_{L_t^1(\dot B_{2,\infty}^{\sigma_0+1})}+\|z{\rm div}\,
u\|_{L_t^1(\dot B_{2,\infty}^{ \sigma_0})}\nonumber\\
\lesssim&\, \| z^\ell u\|_{L_t^1(\dot B_{2,\infty}^{\sigma_0+1})}+   \|z^hu\|_{L_t^1(\dot B_{2,\infty}^{\sigma_0+1 })}+\|z{\rm div}\,u^\ell\|_{L_t^1(\dot B_{2,\infty}^{ \sigma_0})}+\|z{\rm div}\,u^h\|_{L_t^1(\dot B_{2,\infty}^{ \sigma_0})} \nonumber\\
\lesssim&\, \|u\|_{\widetilde L_t^\infty(\dot B_{2,1}^{\frac{d}{2}-1})}\|z\|_{L_t^1(\dot B_{2,\infty}^{\sigma_0+2})}^\ell+
\|u\|_{\widetilde L_t^\infty(\dot B_{2,1}^{\frac{d}{2}-1})}\|z\|_{L_t^1(\dot B_{2,\infty}^{\sigma_0+2})}^h\nonumber\\
&+\|z\|_{\widetilde L_t^\infty(\dot B_{2,1}^{\frac{d}{2}-1})}\|u\|_{L_t^1(\dot B_{2,\infty}^{\sigma_0+2})}^\ell+\|z\|_{\widetilde L_t^\infty(\dot B_{2,1}^{\frac{d}{2}})}\|u\|^h_{L_t^1(\dot B_{2,\infty}^{\sigma_0+1})} \nonumber\\
\lesssim&\, \mathcal{E}(t) \big( \mathcal{E}_{\ell,\sigma_0}(t)+  \mathcal{E}(t)     \big).
\end{align}
To handle the nonlinear term $f(a)\nabla z$, we begin by decomposing $f(a)$ as  
\begin{align*}  
f(a) = -\frac{a}{1+a} = -a - a f(a),  
\end{align*}  
which gives
\begin{align*}
\|f(a)\|_{\widetilde L_t^\infty(\dot B_{2,1}^{\frac{d}{2}-1})}\lesssim&\,    \| a\|_{\widetilde L_t^\infty(\dot B_{2,1}^{\frac{d}{2}-1})} +\|a\|_{\widetilde L_t^\infty(\dot B_{2,1}^{\frac{d}{2}-1})}\|f(a)\|_{\widetilde L_t^\infty(\dot B_{2,1}^{\frac{d}{2} })}\nonumber\\
\lesssim&\, \| a\|_{\widetilde L_t^\infty(\dot B_{2,1}^{\frac{d}{2}-1})} \Big(1+\| a\|_{\widetilde L_t^\infty(\dot B_{2,1}^{\frac{d}{2} })}\Big)\nonumber\\
\lesssim&\,  (\mathcal{E}(t)+1 )\mathcal{E}(t).
\end{align*}
Therefore, we further get
\begin{align}\label{G5.8}
\|f(a)\nabla z\|_{L_t^1(\dot B_{2,\infty}^{ \sigma_0})}\lesssim&\,\|f(a) \|_{\widetilde L_t^\infty(\dot B_{2,1}^{ \frac{d}{2}-1})}  \|z\|^\ell_{L_t^1(\dot B_{2,\infty}^{\sigma_0+2})}   +\|f(a) \|_{\widetilde L_t^\infty(\dot B_{2,1}^{ \frac{d}{2}})}  \|z\|^h_{L_t^1(\dot B_{2,\infty}^{\sigma_0+1})}\nonumber\\
\lesssim&\,(1+\mathcal{E}(t)) \mathcal{E}(t)\big( \mathcal{E}_{\ell,\sigma_0}(t)+  \mathcal{E}(t)     \big).
\end{align}
Putting the estimates \eqref{G5.5}--\eqref{G5.8} into \eqref{G5.4}, one arrives at 
\begin{align*}
 \mathcal{E}_{\ell,\sigma_0}(t)\lesssim   \|(u_0,z_0)\|_{\dot B_{2,\infty}^{ \sigma_0}}^\ell +  (1+\mathcal{E}(t))\mathcal{E}(t)\big( \mathcal{E}_{\ell,\sigma_0}(t)+  \mathcal{E}(t)  \big), 
\end{align*}
which  together with the smallness of $\mathcal{E}(t)$ leads to
\begin{align*}
 \mathcal{E}_{\ell,\sigma_0}(t)\lesssim   \|(u_0,z_0)\|_{\dot B_{2,\infty}^{ \sigma_0}}^\ell + \|(a_0,z_0)\|_{\dot B_{2,1}^{\frac{d}{2}-1}\cap\dot B_{2,1}^{\frac{d}{2}}}+\|u_0\|_{\dot B_{2,1}^{\frac{d}{2}-1}}\lesssim  \delta_*.
\end{align*}
Hence, we obtain \eqref{G5.2}, thereby completing the proof of Lemma \ref{L5.1}.
\end{proof}

We now establish the time-weighted estimates for the low-frequency and high-frequency components of $(u,z)$ separately, as stated in Lemmas \ref{L5.2}--\ref{L5.3} below. 

For $M > \max\left\{\frac{1}{2}\big(\frac{d}{2} + 1 - \sigma_0\big), 1\right\}$ and $\tau\in (0,t],$ inspired by \cite{LS-SIMA-2023}, we introduce a time-weighted functional $\mathcal{E}_{M}(t)$ as:   
\begin{align}\label{G5.9}
\mathcal{E}_{M}(t) :=&\, \|\tau^M u(\tau)\|_{\widetilde L_t^\infty(\dot B_{2,1}^{\frac{d}{2}-1})}+ \|\tau^M z(\tau)\|_{\widetilde L_t^\infty(\dot B_{2,1}^{\frac{d}{2}-1})}^\ell+ \|\tau^M z(\tau)\|_{\widetilde L_t^\infty(\dot B_{2,1}^{\frac{d}{2}})}^h\nonumber\\
&+\|\tau^M u(\tau)\|_{L_t^1(\dot B_{2,1}^{\frac{d}{2}+1})}+\|\tau^M z(\tau)\|_{L_t^1(\dot B_{2,1}^{\frac{d}{2}+1})}^\ell+\|\tau^M z(\tau)\|_{L_t^1(\dot B_{2,1}^{\frac{d}{2}})}^h.
\end{align}

\begin{lem}\label{L5.2}
Let $(a,u,z)$ be the  global strong solution to the Cauchy problem \eqref{A3}--\eqref{A3-1} established in  Theorem \ref{T1.2}. Then, under the assumptions of Theorem \ref{T1.5},  for any $t>0$ and  $M > \max\left\{\frac{1}{2}\big(\frac{d}{2} + 1 - \sigma_0\big), 1\right\}$, it holds that
\begin{align} \label{G5.10}
 &\|\tau^M (u,z)\|_{\widetilde L_t^\infty(\dot B_{2,1}^{{\frac{d}{2}}-1})}^\ell+\|\tau^M (u,z)\|_{L_t^1(\dot B_{2,1}^{\frac{d}{2}+1})}^\ell\nonumber\\
 &\quad\lesssim   \big(\eta+\mathcal{E}(t)\big) \mathcal{E}_{M}(t)+\frac{\mathcal{E}(t)+\mathcal{E}_{\ell,\sigma_0}(t)}{\eta} t^{M-\frac{1}{2}(\frac{d}{2}-1-\sigma_0)},  
\end{align}
where $\eta>0$ is a constant to be determined later. Here, $\mathcal{E}(t)$, $\mathcal{E}_{\ell,\sigma_0}(t)$, and $\mathcal{E}_M(t)$ are defined in  \eqref{G5.1}, \eqref{G5.2} 
and \eqref{G5.9}, respectively.
\end{lem}

\begin{proof}
Similar to the derivation analysis in \eqref{G5.3}, multiplying \eqref{G3.9} by $t^{M}$ yields  
\begin{align}\label{G5.11}
&\frac{{\rm d}}{{\rm d}t}\big(t^M \mathcal{E}_k(t)\big)+t^M 2^{2k}\mathcal{E}_k(t)-Mt^{M-1} \mathcal{E}_k(t)  \nonumber\\
&\quad \lesssim t^M\sqrt{\mathcal{E}_k(t)}\big( \|\dot\Delta_k(u\cdot\nabla u)\|_{L^2}+\|\dot\Delta_k(f(a)\nabla{\rm div}\,u)\|_{L^2}\nonumber\\
&\qquad +\|\dot\Delta_k(f(a)\Delta u)\|_{L^2} +\|\dot\Delta_k(f(a)\nabla z)\|_{L^2} + \|\dot\Delta_k(z{\rm div}\,u)\|_{L^2}\nonumber\\
&\qquad +\|[u\cdot\nabla,\dot\Delta_k]z\|_{L^2}+\|{\rm div}\,u\|_{L^\infty}\|\dot\Delta_k(u,z)\|_{L^2}\big),
\end{align}
for any $k\leq 0$. Here, $\mathcal{E}_k(t)\backsim \|\dot\Delta_k(u,z)\|_{L^2}^2$.
By dividing \eqref{G5.11} by $\big( \mathcal{E}_k(t) + \epsilon^2 \big)^{\frac{1}{2}}$ for $\epsilon > 0$, integrating the resulting inequality over $[0, t]$, and then taking the limit as $\epsilon \to 0$, one has
\begin{align}\label{G5.12}
&t^M \|\dot\Delta_k (u,z)\|_{L^2}+2^{2k}\int_0^t \tau^M \|\dot\Delta_k (u,z)\|_{L^2}{\rm d}\tau  -\int_0^t \tau^{M-1}\|\dot\Delta_k(u,z)\|_{L^2}{\rm d}\tau \nonumber\\
&\quad  \lesssim\int_0^t\tau^M\big(\|\dot\Delta_k(u\cdot\nabla u)\|_{L^2}+\|\dot\Delta_k(f(a)\nabla{\rm div}\, u)\|_{L^2}\nonumber\\
& \qquad \qquad + \|\dot\Delta_k(f(a)\Delta u)\|_{L^2}   +\|\dot\Delta(f(a)\nabla z)\|_{L^2}    \big){\rm d}\tau\nonumber\\
& \qquad  + \int_0^t \tau^{M}\big(\|\dot\Delta_k(z{\rm div}\,u)\|_{L^2} + \|[u\cdot\nabla ,\dot\Delta_k]z\|_{L^2}+\|{\rm div}\,u\|_{L^\infty}\|\dot\Delta_k(u,z)\|_{L^2}\big){\rm d}\tau, 
\end{align}
for any $k\leq 0$.
Multiplying \eqref{G5.12} by $2^{k(\frac{d}{2}-1)}$, taking the supremum over $[0,t]$, and summing over all $k \leq 0$,  we find that
\begin{align}\label{G5.13}
&\|\tau^M (u,z)\|_{\widetilde L_t^\infty(\dot B_{2,1}^{\frac{d}{2}-1})}^\ell+\|\tau^M (u,z)\|_{  L_t^1(\dot B_{2,1}^{\frac{d}{2}+1})}^\ell\nonumber\\
&\quad \lesssim\|\tau^M u\cdot \nabla u\|_{L_t^1(\dot B_{2,1}^{\frac{d}{2}-1})}^\ell+\|\tau^Mf(a)\nabla{\rm div}\,u\|_{L_t^1(\dot B_{2,1}^{\frac{d}{2}-1})}^\ell+\|\tau^Mf(a)\Delta u\|_{L_t^1(\dot B_{2,1}^{\frac{d}{2}-1})}^\ell\nonumber\\
&\qquad +\|\tau^Mf(a)\nabla z\|_{L_t^1(\dot B_{2,1}^{\frac{d}{2}-1})}^\ell+ \|\tau^M z{\rm div}\,u\|_{L_t^1(\dot B_{2,1}^{\frac{d}{2}-1})}^\ell+\sum_{k\leq 0}2^{k(\frac{d}{2}-1)}\|[u\cdot\nabla ,\dot\Delta_k]\tau^M z\|_{L_t^1(L^2)}\nonumber\\
&\qquad +\int_0^t\|\tau^Mu\|_{\dot B_{2,1}^{\frac{d}{2}+1}}\|(u,z)\|_{\dot B_{2,1}^{\frac{d}{2}-1}}^\ell{\rm d}\tau+  \int_0^t \tau^{M-1}\|(u,z)\|_{\dot B_{2,1}^{\frac{d}{2}-1}}^\ell {\rm d}\tau.    
\end{align}
Similar to the estimates \eqref{G3.16}--\eqref{G3.18}, applying Lemmas \ref{LA.2}--\ref{LA.4} yields 
\begin{align}\label{G5.14}
 \|\tau^M u\cdot \nabla u\|_{L_t^1(\dot B_{2,1}^{\frac{d}{2}-1})} 
 \lesssim
 &\,  \|u\|_{\widetilde L_t^\infty(\dot B_{2,1}^{\frac{d}{2}-1})}\| \tau^Mu\|_{L_t^1(\dot B_{2,1}^{\frac{d}{2}+1})} 
 \lesssim   \mathcal{E}(t)\mathcal{E}_{M}(t),\\ \label{G5.15}
 \|\tau^Mf(a)\nabla{\rm div}\,u\|_{L_t^1(\dot B_{2,1}^{\frac{d}{2}-1})} \lesssim&\,\|a\|_{\widetilde L_t^\infty(\dot B_{2,1}^\frac{d}{2})}\| \tau^Mu\|_{L_t^1(\dot B_{2,1}^{\frac{d}{2}+1})} \lesssim   \mathcal{E}(t)\mathcal{E}_{M}(t),\\\label{G5.16}
  \|\tau^Mf(a)\Delta u\|_{L_t^1(\dot B_{2,1}^{\frac{d}{2}-1})}\lesssim&\,\|a\|_{\widetilde L_t^\infty(\dot B_{2,1}^\frac{d}{2})}\| \tau^Mu\|_{L_t^1(\dot B_{2,1}^{\frac{d}{2}+1})} \lesssim   \mathcal{E}(t)\mathcal{E}_{M}(t),\\
 \label{G5.17}
 \|\tau^Mf(a)\nabla z\|_{L_t^1(\dot B_{2,1}^{\frac{d}{2}-1})}\lesssim&\, \|a\|_{\widetilde L_t^\infty(\dot B_{2,1}^{\frac{d}{2}-1}\cap \dot B_{2,1}^{\frac{d}{2}})} \|\tau^Mz\|_{L_t^1(\dot B_{2,1}^{\frac{d}{2}+1}+\dot B_{2,1}^{\frac{d}{2}}   ) }  \lesssim \mathcal{E}(t)\mathcal{E}_{M}(t), \\ \label{G5.18}
 \|\tau^M z{\rm div}\,u\|_{L_t^1(\dot B_{2,1}^{\frac{d}{2}-1})}\lesssim&\,\|z\|_{\widetilde L_t^\infty(\dot B_{2,1}^{\frac{d}{2}-1}\cap \dot B_{2,1}^{\frac{d}{2}})}  \|\tau^Mu\|_{  L_t^1(\dot B_{2,1}^{\frac{d}{2}+1})}\lesssim \mathcal{E}(t)\mathcal{E}_{M}(t), \\ \label{G5.19}
 \int_0^t\|\tau^Mu\|_{\dot B_{2,1}^{\frac{d}{2}+1}}\|(u,z)\|_{\dot B_{2,1}^{\frac{d}{2}-1}}{\rm d}\tau\lesssim&\, \|(u,z)\|_{L_t^\infty(\dot B_{2,1}^{\frac{d}{2}-1})}\|\tau^Mu\|_{  L_t^1(\dot B_{2,1}^{\frac{d}{2}+1})}\lesssim \mathcal{E}(t)\mathcal{E}_{M}(t).
\end{align}
From \eqref{A.3}, we have
\begin{align}\label{G5.20}
 \sum_{k\leq 0}2^{k(\frac{d}{2}-1)}\|[u\cdot\nabla ,\dot\Delta_k]\tau^M z\|_{L_t^1(L^2)}\lesssim \|\tau^M u \|_{L_t^1(\dot B_{2,1}^{\frac{d}{2}+1})}  \|  z \|_{L_t^\infty(\dot B_{2,1}^{\frac{d}{2}-1})} \lesssim \mathcal{E}(t)\mathcal{E}_{M}(t).
\end{align}

To handle the last term on the right-hand side of \eqref{G5.13}, we decompose it into the following parts:
\begin{align*}
\int_0^t \tau^{M-1}\|(u,z)\|_{\dot B_{2,1}^{\frac{d}{2}-1}}  {\rm d}\tau\lesssim   \int_0^t \tau^{M-1}\|(u^\ell,z^\ell)\|_{\dot B_{2,1}^{\frac{d}{2}-1}}  {\rm d}\tau  +\int_0^t \tau^{M-1}\|(u^h,z^h)\|_{\dot B_{2,1}^{\frac{d}{2}-1}}  {\rm d}\tau.
\end{align*}
On the one hand, for the low-frequency component, it follows from the interpolation inequality \eqref{A.1} that
 \begin{align}\label{G5.21}
&\,\int_0^t \tau^{M-1}\|(u^\ell,z^\ell)\|_{\dot B_{2,1}^{\frac{d}{2}-1}}  {\rm d}\tau \nonumber\\
\lesssim&\, \int_0^t \tau^{M-1} \|(u^\ell,z^\ell)\|_{\dot B_{2,\infty}^{\sigma_0}}^{1-\zeta_1} \|(u^\ell,z^\ell)\|_{\dot B_{2,1}^{\frac{d}{2}+1}}^{\zeta_1} {\rm d}\tau \nonumber\\
\lesssim&\, \bigg(\int_0^t \tau^{M-\frac{1}{1-\zeta_1}}{\rm d}\tau  \bigg)^{1-\zeta_1} \|(u^\ell,z^\ell)\|_{\widetilde L_t^\infty(\dot B_{2,\infty}^{\sigma_0})}^{1-\zeta_1} \|(\tau^M u^\ell,\tau^M z^\ell)\|_{L_t^1(\dot B_{2,1}^{\frac{d}{2}+1})}^{\zeta_1}\nonumber\\
\lesssim&\,\Big( t^{M-\frac{1}{2}(\frac{d}{2}-1-\sigma_0)} \|(u,z)\|^\ell_{\widetilde L_t^\infty(\dot B_{2,\infty}^{\sigma_0})}\Big)^{1-\zeta_1} \Big(\|(\tau^Mu,\tau^M z)\|^\ell_{L_t^1(\dot B_{2,1}^{\frac{d}{2}+1})}\Big)^{\zeta_1}\nonumber\\
\lesssim&\, \eta {\mathcal{E}_M(t)}+\frac{\mathcal{E}_{\ell,\sigma_0}}{\eta} t^{M-\frac{1}{2}(\frac{d}{2}-1-\sigma_0)},
\end{align}
where the constant $\zeta_1=\frac{{d}/{2}-1-\sigma_0}{{d}/{2}+1-\sigma_0}\in (0,1)$. 
On the other hand, for the high-frequency component, 
exploiting the  interpolation inequality \eqref{A.1} and the dissipative properties of $(u^h,z^h)$ leads to  
\begin{align} \label{G5.22}
&\int_0^t \tau^{M-1}\|(u^h,z^h)\|_{\dot B_{2,1}^{\frac{d}{2}-1}}  {\rm d}\tau\nonumber\\
\lesssim&\,    \Big( t^{M-\frac{1}{2}(\frac{d}{2}-1-\sigma_0)} \|(u^h,z^h)\|_{\widetilde L_t^\infty(\dot B_{2,1}^{\frac{d}{2}-1})}\Big)^{1-\zeta_1} \Big(\|(\tau^Mu,\tau^M z)\|^h_{L_t^1(\dot B_{2,1}^{\frac{d}{2}-1})}\Big)^{\zeta_1}\nonumber\\
\lesssim&\,    \Big( t^{M-\frac{1}{2}(\frac{d}{2}-1-\sigma_0)} \Big(\|u\|_{\widetilde L_t^\infty(\dot B_{2,1}^{\frac{d}{2}-1})}^h +\| z\|^h_{\widetilde L_t^\infty(\dot B_{2,1}^{\frac{d}{2}})}\Big)\Big)^{1-\zeta_1} \Big(\|\tau^Mu \|^h_{L_t^1(\dot B_{2,1}^{\frac{d}{2}-1})}+\| \tau^M z \|^h_{L_t^1(\dot B_{2,1}^{\frac{d}{2} })}\Big)^{\zeta_1}\nonumber\\
\lesssim&\, \eta {\mathcal{E}_M(t)}+\frac{\mathcal{E}(t)}{\eta} t^{M-\frac{1}{2}(\frac{d}{2}-1-\sigma_0)}. 
\end{align}
Combining the estimates \eqref{G5.21} and \eqref{G5.22} together, one reaches
\begin{align} \label{G5.23}
\int_0^t \tau^{M-1}\|(u,z)\|_{\dot B_{2,1}^{\frac{d}{2}-1}}  {\rm d}\tau\lesssim    \eta {\mathcal{E}_M(t)}+\frac{\mathcal{E}(t)+\mathcal{E}_{\ell,\sigma_0}(t)}{\eta} t^{M-\frac{1}{2}(\frac{d}{2}-1-\sigma_0)}. 
\end{align}

Plugging the estimates \eqref{G5.14}--\eqref{G5.20} and \eqref{G5.23} into \eqref{G5.13}, we consequently derive the desired estimate \eqref{G5.10}. 
\end{proof}

\begin{lem}\label{L5.3}
Let $(a,u,z)$ be the  global strong solution to the Cauchy problem \eqref{A3}--\eqref{A3-1} established in  Theorem \ref{T1.2}. Then, under the assumptions of Theorem \ref{T1.5},  for any $t>0$ and  $M > \max\left\{\frac{1}{2}\big(\frac{d}{2} + 1 - \sigma_0\big), 1\right\}$, it holds that
\begin{align}  \label{G5.24}
 &\|\tau^M u\|_{\widetilde L_t^\infty(\dot B_{2,1}^{{\frac{d}{2}}-1})}^h+\|\tau^M  z\|_{\widetilde L_t^\infty(\dot B_{2,1}^{{\frac{d}{2}}})}^h+\|\tau^M u\|_{L_t^1(\dot B_{2,1}^{\frac{d}{2}+1})}^h+\|\tau^M u\|_{L_t^1(\dot B_{2,1}^{\frac{d}{2}})}^h\nonumber\\
 &\quad\lesssim    \big(\eta+\mathcal{E}(t)\big) \mathcal{E}_{M}(t)+\frac{\mathcal{E}(t) }{\eta} t^{M-\frac{1}{2}(\frac{d}{2}-1-\sigma_0)},  
\end{align}
where $\eta>0$ is a constant to be determined later. Here, $\mathcal{E}(t)$, $\mathcal{E}_{\ell,\sigma_0}(t)$, and $\mathcal{E}_M(t)$ are defined in  \eqref{G5.1}, \eqref{G5.2} 
and \eqref{G5.9}, respectively.
\end{lem}

\begin{proof}
For any $k \geq -1$, multiplying \eqref{G3.9} by $t^M$, we deduce that
\begin{align*} 
&\frac{{\rm d}}{{\rm d}t}\big(t^M \mathcal{E}_k(t)\big)+t^M  \mathcal{E}_k(t)-Mt^{M-1} \mathcal{E}_k(t)  \nonumber\\
&\quad \lesssim t^M\sqrt{\mathcal{E}_k(t)}\bigg(  \|\dot\Delta_k(f(a)\nabla{\rm div}\,u)\|_{L^2}+\|\dot\Delta_k(f(a)\Delta u)\|_{L^2} +\|\dot\Delta_k(f(a)\nabla z)\|_{L^2} \nonumber\\
&\qquad +2^{k} \|\dot\Delta_k(z{\rm div}\,u)\|_{L^2} +\|[u\cdot\nabla,\dot\Delta_k] u\|_{L^2}+\|[u\cdot\nabla,\dot\Delta_k] z\|_{L^2}\nonumber\\
&\qquad +\sum_{j=1}^d\|[u\cdot\nabla,\partial_j\dot\Delta_k]z\|_{L^2}+\|{\rm div}\,u\|_{L^\infty}\|\dot\Delta_k(u,\nabla z)\|_{L^2}\bigg),
\end{align*}
where we have utilized the fact that $2^k \gtrsim 1$. Here, $\mathcal{E}_k(t) \backsim  \|\dot\Delta_k(u,\nabla z)\|_{L^2}^2$. Performing direct computations on the above inequality gives
\begin{align*} 
& t^M \|\dot\Delta_k (u,\nabla z)\|_{L^2}+ \int_0^t \tau^M \|\dot\Delta_k (u,\nabla z)\|_{L^2}{\rm d}\tau  -\int_0^t \tau^{M-1}\|\dot\Delta_k(u,\nabla z)\|_{L^2}{\rm d}\tau  \nonumber \\
&\quad  \lesssim\int_0^t\tau^M\big\{ \|\dot\Delta_k(f(a)\nabla{\rm div}\, u)\|_{L^2}+ \|\dot\Delta_k(f(a)\Delta u)\|_{L^2}   +\|\dot\Delta(f(a)\nabla z)\|_{L^2}\nonumber\\
  &\qquad \qquad +2^{k} \|\dot\Delta_k(z{\rm div}\,u)\|_{L^2}     \big\}{\rm d}\tau  + \int_0^t \tau^{M}\bigg(  \|[u\cdot\nabla,\dot\Delta_k] (u,z)\|_{L^2}\nonumber \\
& \qquad +\sum_{j=1}^d\|[u\cdot\nabla,\partial_j\dot\Delta_k]z\|_{L^2}+\|{\rm div}\,u\|_{L^\infty}\|\dot\Delta_k(u,\nabla z)\|_{L^2}\bigg){\rm d}\tau, 
\end{align*}
for any $k\geq-1$, which implies that 
\begin{align} \label{G5.25}
&\|\tau^M u\|_{\widetilde L_t^\infty(\dot B_{2,1}^{\frac{d}{2}-1})}^h+\|\tau^M z\|_{\widetilde L_t^\infty(\dot B_{2,1}^{\frac{d}{2}})}^h+\|\tau^M u\|_{  L_t^1(\dot B_{2,1}^{\frac{d}{2}-1})}^h+\|\tau^M z\|_{  L_t^1(\dot B_{2,1}^{\frac{d}{2}})}^h\nonumber\\
&\quad \lesssim\|\tau^Mf(a)\nabla{\rm div}\,u\|_{L_t^1(\dot B_{2,1}^{\frac{d}{2}-1})}^h+\|\tau^Mf(a)\Delta u\|_{L_t^1(\dot B_{2,1}^{\frac{d}{2}-1})}^h+\|\tau^Mf(a)\nabla z\|_{L_t^1(\dot B_{2,1}^{\frac{d}{2}-1})}^h\nonumber\\
&\qquad   + \|\tau^M z{\rm div}\,u\|_{L_t^1(\dot B_{2,1}^{\frac{d}{2}})}^h+\int_0^t\|\tau^Mu\|_{\dot B_{2,1}^{\frac{d}{2}+1}}\|(u,\nabla z)\|_{\dot B_{2,1}^{\frac{d}{2}-1}}^h{\rm d}\tau+  \int_0^t \tau^{M-1}\|(u,\nabla z)\|_{\dot B_{2,1}^{\frac{d}{2}-1}}^h {\rm d}\tau\nonumber\\
&\qquad   +\sum_{k\geq -1}2^{k(\frac{d}{2}-1)}\bigg(\|[u\cdot\nabla ,\dot\Delta_k](\tau^M u,\tau^M z)\|_{L_t^1(L^2)}+\sum_{j=1}^d\|[u\cdot\nabla,\partial_j\dot\Delta_k]\tau^M z\|_{L^1_t(L^2)}\bigg) .    
\end{align}
Similar to  \eqref{G5.15}--\eqref{G5.17}, we get 
\begin{align*} 
& \|\tau^Mf(a)\nabla{\rm div}\,u\|_{L_t^1(\dot B_{2,1}^{\frac{d}{2}-1})}^h+\|\tau^Mf(a)\Delta u\|_{L_t^1(\dot B_{2,1}^{\frac{d}{2}-1})}^h+\|\tau^Mf(a)\nabla z\|_{L_t^1(\dot B_{2,1}^{\frac{d}{2}-1})}^h    \lesssim \mathcal{E}(t)\mathcal{E}_{M}(t).
\end{align*}
By a direct calculation, it holds that
\begin{align*}
&\|\tau^M z{\rm div}\,u\|_{L_t^1(\dot B_{2,1}^{\frac{d}{2}})}^h+\int_0^t\|\tau^Mu\|_{\dot B_{2,1}^{\frac{d}{2}+1}}\|(u,\nabla z)\|_{\dot B_{2,1}^{\frac{d}{2}-1}}^h{\rm d}\tau  \nonumber\\
\lesssim\,& \|z\|_{\widetilde L_t^\infty(\dot B_{2,1}^{\frac{d}{2}})}\|\tau^Mu\|_{L_t^1(\dot B_{2,1}^{\frac{d}{2}+1})}+\Big( 
\|u\|_{L_t^\infty(\dot B_{2,1}^{\frac{d}{2}-1})}^h+ \|z\|_{L_t^\infty(\dot B_{2,1}^{\frac{d}{2} })}^h\Big)\|\tau^Mu\|_{L_t^1(\dot B_{2,1}^{\frac{d}{2}+1})} \nonumber\\
\lesssim\,&    \mathcal{E}(t)\mathcal{E}_{M}(t).
\end{align*}
According to Lemma \ref{LA.5}, we have
\begin{align*}
&\sum_{k\geq -1}2^{k(\frac{d}{2}-1)}\bigg(\|[u\cdot\nabla ,\dot\Delta_k](\tau^M u,\tau^M z)\|_{L_t^1(L^2)}+\sum_{j=1}^d\|[u\cdot\nabla,\partial_j\dot\Delta_k]z\|_{L^1_t(L^2)}\bigg)    \nonumber\\
\lesssim\,& \|\tau^M u\|_{L^1_t(\dot B_{2,1}^{\frac{d}{2}+1})}\Big( \|  u \|_{\widetilde L^\infty_t(\dot B_{2,1}^{\frac{d}{2}-1})}+\| z \|_{\widetilde L^\infty_t(\dot B_{2,1}^{\frac{d}{2}-1}\cap\dot B_{2,1}^{\frac{d}{2} } )}\Big)\nonumber\\
\lesssim\,&    \mathcal{E}(t)\mathcal{E}_{M}(t).
\end{align*}
Similar to \eqref{G5.22}, we conclude that 
\begin{align*} 
\int_0^t\tau^{M-1} \|(u,\nabla z)\|_{\dot B_{2,1}^{\frac{d}{2}-1}}^h{\rm d}\tau\lesssim &\,  \Big( t^{M-\frac{1}{2}(\frac{d}{2}-1-\sigma_0)} \|u \|^h_{\widetilde L_t^\infty(\dot B_{2,1}^{\frac{d}{2}-1})}\Big)^{1-\zeta_2} \Big(\|\tau^Mu\|^h_{L_t^1(\dot B_{2,1}^{\frac{d}{2}+1})}\Big)^{\zeta_2}\nonumber\\
&+\Big( t^{M-\frac{1}{2}(\frac{d}{2}-1-\sigma_0)} \|z \|^h_{\widetilde L_t^\infty(\dot B_{2,1}^{\frac{d}{2} })}\Big)^{1-\zeta_2} \Big(\|\tau^M z\|^h_{L_t^1(\dot B_{2,1}^{\frac{d}{2}})}\Big)^{\zeta_2}\nonumber\\
\lesssim&\, \eta  {\mathcal{E}_M(t)}+\frac{\mathcal{E}(t)}{\eta} t^{M-\frac{1}{2}(\frac{d}{2}-1-\sigma_0)},  
\end{align*}
where the constant $\zeta_2=\frac{ {d}/{2}-1-\sigma_0}{{d}/{2}+1-\sigma_0}\in (0,1)$.
Putting all the above estimates into \eqref{G5.25}, we eventually obtain
\begin{align}\label{G5.26}
&\|\tau^M u\|_{\widetilde L_t^\infty(\dot B_{2,1}^{\frac{d}{2}-1})}^h+\|\tau^M z\|_{\widetilde L_t^\infty(\dot B_{2,1}^{\frac{d}{2}})}^h+\|\tau^M u\|_{  L_t^1(\dot B_{2,1}^{\frac{d}{2}-1})}^h+\|\tau^M z\|_{  L_t^1(\dot B_{2,1}^{\frac{d}{2}})}^h  \nonumber\\
&\quad\lesssim   \big(\eta+\mathcal{E}(t)\big) \mathcal{E}_{M}(t)+\frac{\mathcal{E}(t) }{\eta} t^{M-\frac{1}{2}(\frac{d}{2}-1-\sigma_0)},  
\end{align}

Finally, we aim to estimate $\|\tau^M u\|_{L_t^1(\dot B_{2,1}^{\frac{d}{2}+1})}^h$. Multiplying \eqref{A3}$_2$ by $t^M$, we arrive at 
\begin{align}\label{G5.27}
&\partial_t(t^M u)-\mu\Delta (t^M u)-(\mu+\lambda)\nabla{\rm div}\,(t^M u)\nonumber\\
&\quad =Mt^{M-1}u -\nabla(t^M z)-t^M u\cdot\nabla u-t^Mf(a)\nabla z \nonumber\\
&\qquad  +t^Mf(a)(\mu\Delta u+(\mu+\lambda)\nabla{\rm div}\,u).
\end{align}
Applying Lemma \ref{LA.7} to \eqref{G5.27}, we easily infer 
\begin{align*}
&\|\tau^M u\|_{\widetilde L^\infty_t(\dot B_{2,1}^{\frac{d}{2}-1})}^h+\|\tau^M u\|_{L^1_t(\dot B_{2,1}^{\frac{d}{2}+1})}^h\nonumber\\
&\quad \lesssim   \|\tau^M z\|_{L_t^1(\dot B_{2,1}^{\frac{d}{2}})} ^h+\|\tau^M u\cdot\nabla u\|_{L_t^1(\dot B_{2,1}^{\frac{d}{2}-1})} ^h  +\|\tau^M f(a) \nabla z\|_{L_t^1(\dot B_{2,1}^{\frac{d}{2}-1})} ^h\nonumber\\
& \qquad   +\|\tau^M  f(a)\Delta u\|_{L_t^1(\dot B_{2,1}^{\frac{d}{2}-1})} ^h +\|\tau^M  f(a)\nabla{\rm div}\, u\|_{L_t^1(\dot B_{2,1}^{\frac{d}{2}-1})} ^h+\int_0^t\tau^{M-1} \|u\|_{\dot B_{2,1}^{\frac{d}{2}-1}}^h{\rm d}\tau,
\end{align*}
which, combined with the nonlinear estimates on the right-hand side of \eqref{G5.25} and the estimate \eqref{G5.26}, yields
\begin{align}\label{G5.28}
\|\tau^M u\|_{\widetilde L^\infty_t(\dot B_{2,1}^{\frac{d}{2}-1})}^h+\|\tau^M u\|_{L^1_t(\dot B_{2,1}^{\frac{d}{2}+1})}^h\lesssim      \big(\eta+\mathcal{E}(t)\big) \mathcal{E}_{M}(t)+\frac{\mathcal{E}(t) }{\eta} t^{M-\frac{1}{2}(\frac{d}{2}-1-\sigma_0)}.
\end{align}
Thus, the combination of \eqref{G5.26} and \eqref{G5.28} gives rise to \eqref{G5.24}.
\end{proof}

With Lemmas \ref{L5.1}--\ref{L5.3} established, we are now in a position to prove Theorem \ref{T1.5}.

\begin{proof}[Proof of Theorem \ref{T1.5}]

From \eqref{G5.10} and \eqref{G5.24}, we have
\begin{align*} 
\mathcal{E}_{M}(t)\lesssim  \big(\eta+\mathcal{E}(t)\big) \mathcal{E}_{M}(t)+\frac{\mathcal{E}(t)+\mathcal{E}_{\ell,\sigma_0}(t)}{\eta} t^{M-\frac{1}{2}(\frac{d}{2}-1-\sigma_0)},    
\end{align*}
for any $t>0$ and  $M > \max\left\{\frac{1}{2}\big(\frac{d}{2} + 1 - \sigma_0\big), 1\right\}$. 
By choosing the constant $\eta$ sufficiently small and exploiting the smallness of $\mathcal{E}(t)$ together with Lemma \ref{L5.1}, one gets
\begin{align}\label{G5.29}
 \mathcal{E}_{M}(t)\lesssim   \delta_{*} t^{M-\frac{1}{2}(\frac{d}{2}-1-\sigma_0)},
\end{align}
for any $t>0$, where $\delta_*$ is given by \eqref{TE.4}.
Combining \eqref{G5.9} and \eqref{G5.29} gives
\begin{align}\label{G5.30}
\|u (t)\|_{\dot B_{2,1}^{\frac{d}{2}-1}}+\|z(t)\|_{\dot B_{2,1}^{\frac{d}{2}-1}\cap\dot B_{2,1}^{\frac{d}{2}}}\lesssim&\, \|u(t)\|_{{\dot B_{2,1}^{\frac{d}{2}-1}}}+ \|z(t)\|_{\dot B_{2,1}^{\frac{d}{2}-1}}^\ell+  \|z(t)\|_{\dot B_{2,1}^{\frac{d}{2}}}^h\nonumber\\
\lesssim&\, \delta_*(1+t)^{-\frac{1}{2}(\frac{d}{2}-1-\sigma_0)},
\end{align}
for any $t\geq 1$.
Using \eqref{A.1} and \eqref{G5.30}, we notice that
\begin{align*}
\|(u,z)(t)\|_{\dot B_{2,1}^{\sigma}}\lesssim&\,  \|(u^\ell,z^\ell)(t)\|_{\dot B_{2,1}^{\sigma}} +   \|(u^h,z^h)(t)\|_{\dot B_{2,1}^{\sigma}}\nonumber\\
\lesssim&\,  \|(u^\ell,z^\ell)(t)\|_{\dot B_{2,\infty}^{\sigma_0}}^{  \frac{{ {d}/{2}-1-\sigma}}{ {d}/{2}-1-\sigma_0} }  \|(u^\ell,z^\ell)(t)\|_{\dot B_{2,1}^{\frac{d}{2}-1}}^{  \frac{\sigma-\sigma_0}{ {d}/{2}-1-\sigma_0} }+\|u(t)\|^h_{\dot B_{2,1}^{\frac{d}{2}-1}}+\|z(t)\|^h_{\dot B_{2,1}^{\frac{d}{2}}}\nonumber\\
\lesssim&\, \delta_*^{  \frac{{ {d}/{2}-1-\sigma}}{ {d}/{2}-1-\sigma_0} }\Big(\delta_*\|(u^\ell,z^\ell)(t)\|_{\dot B_{2,1}^{\frac{d}{2}-1}}\Big)^{ \frac{\sigma-\sigma_0}{\ {d}/{2}-1-\sigma_0}}+\delta_{*} (1+t)^{-\frac{1}{2}( {d}/{2}-1-\sigma_0)}\nonumber\\
\lesssim&\, \delta_* (1+t)^{-\frac{1}{2}(\sigma-\sigma_0)},
\end{align*}
for any $\sigma \in \big(\sigma_0, \frac{d}{2} - 1\big]$ with $\sigma_0 \in \big[-\frac{d}{2}, \frac{d}{2} - 2\big)$. Hence, \eqref{TE.2} follows. 

To establish \eqref{TE.3}, we estimate $\|u(t)\|^h_{\dot B_{2,1}^{\frac{d}{2}+1}}$ below. From \eqref{TB.4}, it holds that there exists a time $t_*>0$ such that
\begin{align*}
 \|u(t_*)\| _{\dot B_{2,1}^{\frac{d}{2}+1}}\lesssim  \|(a_0,z_0)\|_{\dot B_{2,1}^{\frac{d}{2}-1}\cap  \dot B_{2,1}^{\frac{d}{2}}}+\|u_0\|_{\dot B_{2,1}^{\frac{d}{2}-1}} \lesssim \delta_*.  
\end{align*}
For notational simplicity, we assume without loss of generality that $t_* = 1$. Then, applying Lemma \ref{LA.7} to \eqref{G5.27} on  $[1,t]$ yields  
\begin{align*}
&\|\tau^M u(t)\|_{\widetilde L^\infty(1,t;\dot B_{2,1}^{\frac{d}{2}+1})}^h\nonumber\\
\lesssim\,& \|u(1)\|^h_{\dot B_{2,1}^{\frac{d}{2}+1}} +\|\tau^M z\|_{\widetilde L^\infty(1,t;\dot B_{2,1}^{\frac{d}{2}-1})}^h +\|\tau^{M } u\cdot\nabla u\|_{\widetilde L^\infty(1,t;\dot B_{2,1}^{\frac{d}{2}-1})}^h \nonumber\\
& +\|\tau^{M } f(a) \nabla z\|_{\widetilde L^\infty(1,t;\dot B_{2,1}^{\frac{d}{2}-1})}^h+\|\tau^{M} f(a)\Delta u\|_{\widetilde L^\infty(1,t;\dot B_{2,1}^{\frac{d}{2}-1})}^h \nonumber\\
&+\|\tau^{M} f(a)\nabla{\rm div}\,u\|_{\widetilde L^\infty(1,t;\dot B_{2,1}^{\frac{d}{2}-1})}^h+\|\tau^{M-1 } u\|_{\widetilde L^\infty(1,t;\dot B_{2,1}^{\frac{d}{2}-1})}^h  \nonumber\\
\lesssim \,&\|u(1)\|^h_{\dot B_{2,1}^{\frac{d}{2}+1}} +\mathcal{E}_{M}(t)+\Big( \|u\|_{\widetilde L^\infty(1,t;\dot B_{2,1}^{\frac{d}{2}-1})}  + \|a\|_{\widetilde L^\infty(1,t;\dot B_{2,1}^{\frac{d}{2} })}\Big)\|\tau^M u\|_{\widetilde L^\infty(1,t;\dot B_{2,1}^{\frac{d}{2}+1})} ^h  \nonumber\\
&+\Big( \|u\|_{\widetilde L^\infty(1,t;\dot B_{2,1}^{\frac{d}{2}-1})}  + \|a\|_{\widetilde L^\infty(1,t;\dot B_{2,1}^{\frac{d}{2} })}\Big)\|\tau^M u\|_{\widetilde L^\infty(1,t;\dot B_{2,1}^{\frac{d}{2}-1})} ^\ell  \nonumber\\
&+\|a\|_{\widetilde L^\infty(1,t;\dot B_{2,1}^{\frac{d}{2}})}\|\tau^M z\|_{\widetilde L^\infty(1,t;\dot B_{2,1}^{\frac{d}{2}})}\nonumber\\
&+\Big(1+\|a\|_{\widetilde L^\infty(1,t;\dot B_{2,1}^{\frac{d}{2}})}\Big)\|a\|_{\widetilde L^\infty(1,t;\dot B_{2,1}^{\frac{d}{2}-1})}\|\tau^M z\|_{\widetilde L^\infty(1,t;\dot B_{2,1}^{\frac{d}{2}})}\nonumber\\
\lesssim\,& \|u(1)\|^h_{\dot B_{2,1}^{\frac{d}{2}+1}}+\delta\|\tau^M u(t)\|_{\widetilde L^\infty(1,t;\dot B_{2,1}^{\frac{d}{2}+1})}^h+\delta_* t^{M-\frac{1}{2}(\frac{d}{2}-1-\sigma_0)},
\end{align*}
for any $t\geq 1$, which together with the smallness of $\delta$ gives rise to  
\begin{align}\label{G5.31}
\|u(t)\|_{\dot B_{2,1}^{\frac{d}{2}+1} }^h\lesssim \delta_* (1+t)^{ -\frac{1}{2}(\frac{d}{2}-1-\sigma_0)},    
\end{align}
for any $t\geq 1$.  Hence, \eqref{TE.3} follows from \eqref{G5.30} and \eqref{G5.31}. The proof of Theorem \ref{T1.5} is completed.
\end{proof}

\subsection{The upper bound estimate of solutions under the  smallness condition}

In this subsection, we establish time-decay estimates for higher-order spatial derivatives of $(u,z)$ under the assumption that $\|(u_0^\ell,z_0^\ell)\|_{\dot B_{2,\infty}^{\sigma_0}}$ is sufficiently small, where $\sigma_0 \in \big[-\frac{d}{2}, \frac{d}{2}-2\big)$. In contrast to Theorem \ref{T1.5}, an additional assumption on $a_0^\ell \in \dot B_{2,\infty}^{\sigma_0}$ is imposed to address technical difficulties caused by the absence of dissipation in the density $a$. The proof of Theorem \ref{T1.6} is primarily based on the method presented in \cite[Section 5.2]{Danchin-2018}.

Define the time-weighted energy functional as follows:
\begin{align}\label{G5.32}
\mathcal{Z}(t):=&\,\sup_{\sigma\in[\sigma_0+\zeta,\frac{d}{2}+1]} \|\langle\tau\rangle^{\frac{1}{2}(\sigma-\sigma_0)} (u,z)\|_{L_t^\infty(\dot B_{2,1}^{\sigma})}^\ell+\|\langle\tau\rangle^\alpha u\|_{\widetilde L_t^\infty(\dot B_{2,1}^{\frac{d}{2}-1})}^h\nonumber\\
&+\|\langle\tau\rangle^\alpha z\|_{\widetilde L_t^\infty(\dot B_{2,1}^{\frac{d}{2}})}^h+\| \tau ^\alpha u\|_{\widetilde L_t^\infty(\dot B_{2,1}^{\frac{d}{2}+1})}^h,
\end{align}
where  $\langle \tau \rangle:=\sqrt{1+\tau ^2}$ and $\alpha:=\frac{1}{2}(\frac{d}{2}+1-\sigma_0 )$, 
and $\zeta\in(0,1]$ is a sufficiently small constant.

Before establishing the uniform-in-time estimate of $\mathcal{Z}(t)$, it is essential to ensure the uniformly regular evolution of $a$. Since 
the density $a$ lacks a dissipative structure, we cannot directly obtain the estimate of $\|a\|_{\widetilde L_t^\infty(\dot B_{2,\infty}^{\sigma_0})}$; hence, a new method must be developed. By utilizing the auxiliary 
 mode $\omega = \gamma a - z$, we can derive the bounds for $\|\omega\|_{\widetilde L_t^\infty(\dot B_{2,1}^{\sigma})}$ and $\|a\|_{\widetilde L_t^\infty(\dot B_{2,1}^{\sigma})}$ with $\sigma \in \big(\sigma_0, \frac{d}{2}\big]$, respectively.

\begin{lem}\label{L5.4}
Let $(a,u,z)$ be the global strong solution to the Cauchy problem \eqref{A3}--\eqref{A3-1} established in  Theorem \ref{T1.2}. Then, under the assumptions of Theorem \ref{T1.6}, it holds that for any $t>0$,
\begin{align}  \label{G5.33}
 \|a\|_{\widetilde L_t^\infty(\dot B_{2,1}^{\sigma})}+ \|\gamma a-z\|_{\widetilde L_t^\infty(\dot B_{2,1}^{\sigma })}\lesssim&\, \delta_*,  \quad \sigma\in\Big(\sigma_0,\frac{d}{2} \Big],  
\end{align}
where $\delta_{*}$ is defined by   \eqref{TE.4}. 
\end{lem}    
 
\begin{proof}
 By applying Lemma \ref{LA.6} to the equation $  
\partial_t   \omega+ u\cdot\nabla\omega   =\gamma(a-z){\rm div}\,u 
 $, one has
\begin{align*}
 \|\omega\|_{\widetilde L_t^\infty(\dot B_{2,\infty}^{\sigma_0 })}\lesssim&\, e^{C \mathcal{D}(t)}  \Big(\|\omega_0\|_{ \dot B_{2,\infty}^{\sigma_0}  } + \|\gamma( a-z){\rm div}\,u\|_{L_t^1( \dot B_{2,\infty}^{\sigma_0})}\Big)\nonumber\\   
 \lesssim&\, \|\omega_0\|_{ \dot B_{2,\infty}^{\sigma_0}  } + \|\omega\|_{\widetilde L_t^\infty(\dot B_{2,\infty}^{\sigma_0})}\|u\|_{L_t^1(\dot B_{2,1}^{\frac{d}{2}+1})}+\|z\|_{\widetilde L_t^\infty(\dot B_{2,\infty}^{\sigma_0})}\|u\|_{L_t^1(\dot B_{2,1}^{\frac{d}{2}+1})}\nonumber\\
 \lesssim&\,\|\omega_0\|_{ \dot B_{2,\infty}^{\sigma_0}  }+\delta\|\omega\|_{\widetilde L_t^\infty(\dot B_{2,\infty}^{\sigma_0 })}+\delta\|z\|_{\widetilde L_t^\infty(\dot B_{2,\infty}^{\sigma_0 })},
\end{align*}
which together with the smallness of $\delta$, \eqref{TB.4} and \eqref{G5.2} yields
\begin{align}\label{G5.34}
\|\omega\|_{\widetilde L_t^\infty(\dot B_{2,\infty}^{\sigma_0 })}\lesssim &\,    \|\omega_0 \|_{ \dot B_{2,\infty}^{\sigma_0}}+\|z^\ell\|_{\widetilde L_t^\infty(\dot B_{2,\infty}^{\sigma_0 })}+\|z \|_{\widetilde L_t^\infty(\dot B_{2,1}^{\frac{d}{2} })}^h\nonumber\\
 \lesssim&\, \|(a_0^\ell,u_0^\ell,z_0^\ell)\|_{ \dot B_{2,\infty}^{\sigma_0}}+\|u_0\|_{\dot B_{2,1}^{\frac{d}{2}-1}}^h+\|(a_0,z_0)\|_{\dot B_{2,1}^{\frac{d}{2}}}^h\nonumber\\
 \lesssim&\, \delta_*.
\end{align}
Using \eqref{TB.4}, \eqref{G5.2} and \eqref{G5.34}, we find that
\begin{align}\label{G5.35}
  \|a\|_{\widetilde L_t^\infty(\dot B_{2,\infty}^{\sigma_0 })}\lesssim   \|\omega\|_{\widetilde L_t^\infty(\dot B_{2,\infty}^{\sigma_0 })} +\|z^\ell\|_{\widetilde L_t^\infty(\dot B_{2,\infty}^{\sigma_0 })}+\|z \|_{\widetilde L_t^\infty(\dot B_{2,1}^{\frac{d}{2} })}^h\lesssim \delta_*.
\end{align}
Adding \eqref{G5.34} and \eqref{G5.35} up, we eventually get
\begin{align*}
  \|a\|_{\widetilde L_t^\infty(\dot B_{2,1}^{\sigma})}+\|\omega\|_{\widetilde L_t^\infty(\dot B_{2,1}^{\sigma })}\lesssim &\,\|a\|_{\widetilde L_t^\infty(\dot B_{2,\infty}^{\sigma_0})}^\ell+\|\omega\|_{\widetilde L_t^\infty(\dot B_{2,\infty}^{\sigma_0})}^\ell + \|a\|_{\widetilde L_t^\infty(\dot B_{2,1}^{\frac{d}{2}})}^h+\|\omega\|_{\widetilde L_t^\infty(\dot B_{2,1}^{\frac{d}{2}})}^h \lesssim \delta_*.
\end{align*}
Thus, we  derive \eqref{TF.2}, thereby completing the proof of Lemma \ref{L5.4}.
\end{proof}
 
Next, we provide the estimates of $(u, z)$ in both the low and high frequency regimes, respectively.

\begin{lem} \label{L5.5}
Let $(a,u,z)$ be the global strong solution to the Cauchy problem \eqref{A3}--\eqref{A3-1} established in  Theorem \ref{T1.2}. Then, under the assumptions of Theorem \ref{T1.6}, it holds that for any $t>0$,
\begin{align}  \label{G5.36}
 \sup_{\sigma\in[\sigma_0+\zeta,\frac{d}{2}+1]} \big\|\langle\tau\rangle^{\frac{1}{2}(\sigma-\sigma_0)} (u,z)\big\|_{L_t^\infty(\dot B_{2,1}^{\sigma})}^\ell\lesssim  \delta_*+\delta_*\mathcal{Z}(t)+\mathcal{Z}^2(t),  
\end{align}
where $\delta_{*}$ and $\mathcal{Z}(t)$ are defined by  \eqref{TE.4} and \eqref{G5.32}, respectively.
\end{lem}
\begin{proof}
Applying Gr\"{o}nwall's inequality to \eqref{G3.9} and noting that $2^k \lesssim 1$ for all $k \leq 0$, we deduce   
\begin{align*} 
\|\dot\Delta_k (u,z)\|_{L^2} 
\lesssim&\,  e^{-2^{2k}t}   \|\dot \Delta_k (u_0,z_0)\|_{L^2}+\int_0^t e^{-2^{2k}(t-\tau)} \big(\|\dot\Delta_k(u\cdot\nabla u)\|_{L^2}+\|\dot\Delta_k(f(a)\Delta u)\|_{L^2}\big){\rm d}\tau\nonumber\\
& +\int_0^t e^{-2^{2k}(t-\tau)} \big( \|\dot\Delta_k(f(a)\nabla{\rm div}\, u)\|_{L^2}+\|\dot\Delta_k(f(a)\nabla z)\|_{L^2} +2^k\|\dot\Delta_k(z u)\|_{L^2}       \big){\rm d}\tau\nonumber\\
&+\int_0^t e^{-2^{2k}(t-\tau)} \|\dot\Delta_k(z{\rm div}\,u)\|_{L^2} {\rm d}\tau,
\end{align*}
 which implies that 
\begin{align} \label{G5.37}
\|(u,z)\|_{\dot B_{2,1}^\sigma}^\ell\lesssim&\,\int_0^t \langle t-\tau\rangle^{-\frac{1}{2}(\sigma-\sigma_0)} \Big(\|u\cdot\nabla u\|_{\dot B_{2,\infty}^{\sigma_0}}^\ell+\|f(a)\Delta u\|_{\dot B_{2,\infty}^{\sigma_0}}^\ell+\|f(a)\nabla {\rm div}\, u\|_{\dot B_{2,\infty}^{\sigma_0}}^\ell     \Big) {\rm d}\tau \nonumber\\
&+\int_0^t\langle t-\tau\rangle^{-\frac{1}{2}(\sigma-\sigma_0)} \Big(\|f(a)\nabla z\|_{\dot B_{2,\infty}^{\sigma_0}}^\ell +\|zu\|_{\dot B_{2,\infty}^{\sigma_0+1}}^\ell
 +\|z{\rm div}\,u\|_{\dot B_{2,\infty}^{\sigma_0 }}^\ell\Big){\rm d}\tau \nonumber\\
&+\langle t\rangle^{-\frac{1}{2}(\sigma-\sigma_0)}\|(u_0,z_0)\|_{\dot B_{2,\infty}^{\sigma_0}}^\ell,   
\end{align}
for any $\sigma>\sigma_0$,  where we have used the following fact(\!\!\cite[Lemma 2.35]{BCD-Book-2011}):
\begin{align*}
\sup_{t>0}\sum_{k\in\mathbb Z}t^{\sigma-\sigma_0} 2^{k({\sigma-\sigma_0})}e^{-ct2^{2k}}\lesssim 1.   
\end{align*}
Let us first address the dissipative structure coupling terms $u\cdot\nabla u$ and $zu$. Following a similar approach to that in \cite[Section 4.2]{LS-SIMA-2023}, we   obtain
\begin{align}\label{G5.38}
&\int_0^t \langle t-\tau\rangle^{-\frac{1}{2}(\sigma-\sigma_0)} \Big(\|u\cdot\nabla u\|_{\dot B_{2,\infty}^{\sigma_0}}^\ell+ \|z u\|_{\dot B_{2,\infty}^{\sigma_0+1}}^\ell+\|z{\rm div}\,u\|_{\dot B_{2,\infty}^{\sigma_0 }}^\ell  \Big)   {\rm d}\tau\nonumber\\  
&\quad\lesssim      ( \mathcal{Z}^2(t)+\delta_*)\langle t\rangle^{-\frac{1}{2}(\sigma-\sigma_0)}.
\end{align}
For conciseness, the details are omitted. 

Unlike the compressible isentropic Navier-Stokes system \cite{DX-ARMA-2017,Xj-CMP-2019,XX-JDE-2021}, where all nonlinear terms in $(a,u)$ have dissipative structures and no loss of spatial derivatives, the terms $f(a)\nabla z$, $f(a)\Delta u$, and $f(a)\nabla{\rm div}\,u$ in our case cause derivative loss due to the lack of dissipation in $a$, requiring careful analysis.  
To estimate the remaining terms on the right-hand side of \eqref{G5.27}, we analyze the cases $t \leq 2$ and $t \geq 2$ separately. 

For $t \leq 2$, we utilize \eqref{G5.6}, \eqref{G5.8} and $\langle t\rangle\backsim 1$ to derive 
\begin{align} \label{G5.39}
&\int_0^t \langle t-\tau\rangle^{-\frac{1}{2}(\sigma-\sigma_0)} \Big(\|f(a)\Delta u\|_{\dot B_{2,\infty}^{\sigma_0}}^\ell+\|f(a)\nabla{\rm div}\,u\|_{\dot B_{2,\infty}^{\sigma_0}}^\ell+ \|f(a)\nabla z\|_{\dot B_{2,\infty}^{\sigma_0}}^\ell \Big)   {\rm d}\tau  \nonumber\\
&\quad\lesssim   (1+\mathcal{E}(t)) \mathcal{E}(t)\big( \mathcal{E}_{\ell,\sigma_0}(t)+  \mathcal{E}(t)     \big) \langle t\rangle^{-\frac{1}{2}(\sigma-\sigma_0)} \lesssim  ( \mathcal{Z}^2(t)+\delta_*)\langle t\rangle^{-\frac{1}{2}(\sigma-\sigma_0)}.
\end{align}
For $t \geq 2$, we divide the integral into two parts:
\begin{align}\label{G5.40}
&\int_0^t \langle t-\tau\rangle^{-\frac{1}{2}(\sigma-\sigma_0)} \Big(\|f(a)\Delta u\|_{\dot B_{2,\infty}^{\sigma_0}}^\ell+\|f(a)\nabla{\rm div}\,u\|_{\dot B_{2,\infty}^{\sigma_0}}^\ell+ \|f(a)\nabla z\|_{\dot B_{2,\infty}^{\sigma_0}}^\ell \Big)   {\rm d}\tau   \nonumber\\
=&\,\bigg(\int_0^1+\int_1^t\bigg) \langle t-\tau\rangle^{-\frac{1}{2}(\sigma-\sigma_0)} \Big(\|f(a)\Delta u\|_{\dot B_{2,\infty}^{\sigma_0}}^\ell+\|f(a)\nabla{\rm div}\,u\|_{\dot B_{2,\infty}^{\sigma_0}}^\ell+ \|f(a)\nabla z\|_{\dot B_{2,\infty}^{\sigma_0}}^\ell \Big)   {\rm d}\tau \nonumber\\
\equiv:&\, I_1+I_2.
\end{align}
It follows from Lemma \ref{LA.3} and the fact that $\langle t - \tau \rangle \sim \langle t \rangle$ for all $0 \leq \tau \leq 1$ that
\begin{align}\label{G5.41}
I_1\lesssim&\,  \langle t \rangle^{-\frac{1}{2}(\sigma-\sigma_0)}\int_0^1 \Big(\|a\|_{\dot B_{2,1}^{\frac{d}{2}}} \|u\|_{\dot B_{2,\infty}^{\sigma_0+2}}+ \|a\|_{\dot B_{2,1}^{\frac{d}{2}}}\|z\|_{\dot B_{2,\infty}^{\sigma_0+1}}  \Big) {\rm d}\tau \nonumber\\
\lesssim&\, \big(\mathcal{Z}^2(t)+\delta_*   \big)\langle t \rangle^{-\frac{1}{2}(\sigma-\sigma_0)} .
\end{align}
Notice that $\sigma\in \big(\sigma_0,\frac{d}{2}+1\big]$ with $\sigma_0\in\big [-\frac{d}{2},\frac{d}{2}-2\big)$, which means
\begin{align*}
 \sigma-\sigma_0 \in\Big(0, \frac{d}{2}+1-\sigma_0\Big] \quad\text{and}\quad \frac{1}{2} \Big(\frac{d}{2} +1-\sigma_0  \Big )\in \Big(\frac{3}{2} ,d+1\Big].   
\end{align*}
From \eqref{TE.3}, \eqref{G5.32}, \eqref{G5.33} and \eqref{G5.35}, together with Lemmas \ref{LA.3} and \ref{LA.4} and Lemma \ref{LA.9},  one has
\begin{align}\label{G5.42}
I_2   \lesssim&\,  \int_0^t   \langle t-\tau\rangle^{-\frac{1}{2}(\sigma-\sigma_0)}\|a\|_{\dot B_{2,\infty}^{\sigma_0+1}} \Big(\|u\|^\ell_{\dot B_{2,1}^{\frac{d}{2}+1}}+\|u\|^h_{\dot B_{2,1}^{\frac{d}{2}+1}}\Big) {\rm d}\tau  \nonumber\\ 
&+\int_0^t   \langle t-\tau\rangle^{-\frac{1}{2}(\sigma-\sigma_0)}\Big(\|a\|_{\dot B_{2,\infty}^{\sigma_0}} \|\nabla z\|^\ell_{\dot B_{2,1}^{\frac{d}{2} }}+\|a\|_{\dot B_{2,\infty}^{\sigma_0+1}} \|\nabla z\|^h_{\dot B_{2,1}^{\frac{d}{2}-1}}\Big) {\rm d}\tau\nonumber\\
\lesssim&\, \delta_*\big(1+\mathcal{Z}(t)\big)\int_0^t\langle t-\tau\rangle^{-\frac{1}{2}(\sigma-\sigma_0)} \langle\tau \rangle^{-\frac{1}{2}(\frac{d}{2}+1-\sigma_0)} {\rm d}\tau \nonumber\\
\lesssim&\, \delta_*\big(1+\mathcal{Z}(t)\big)\langle t\rangle^{-\frac{1}{2}(\sigma-\sigma_0)}.
\end{align}
Substituting the estimates \eqref{G5.41} and \eqref{G5.42} into \eqref{G5.40} and combining the result with \eqref{G5.39}, we consequently obtain   
\begin{align}\label{G5.43}
&\int_0^t \langle t-\tau\rangle^{-\frac{1}{2}(\sigma-\sigma_0)} \Big(\|f(a)\Delta u\|_{\dot B_{2,\infty}^{\sigma_0}}^\ell+\|f(a)\nabla{\rm div}\,u\|_{\dot B_{2,\infty}^{\sigma_0}}^\ell+ \|f(a)\nabla z\|_{\dot B_{2,\infty}^{\sigma_0}}^\ell \Big)   {\rm d}\tau    \nonumber\\
&\quad \lesssim (\delta_*+\delta_*\mathcal{Z}(t)+\mathcal{Z}^2(t))\langle t\rangle^{-\frac{1}{2}(\sigma-\sigma_0)},
\end{align}
 for any $t>0$.
Plugging \eqref{G5.38} and \eqref{G5.43} into \eqref{G5.37} yields
\begin{align*}
\|(u,z)\|_{\dot B_{2,1}^\sigma}^\ell\lesssim      (\delta_*+\delta_*\mathcal{Z}(t)+\mathcal{Z}^2(t))\langle t\rangle^{-\frac{1}{2}(\sigma-\sigma_0)},
\end{align*}
for any $t>0$, thereby completing the proof of \eqref{G5.36}.
\end{proof}

\begin{lem} \label{L5.6}
Let $(a,u,z)$ be the global strong solution to the Cauchy problem \eqref{A3}--\eqref{A3-1} provided by Theorem \ref{T1.2}. Then, under the assumptions of Theorem \ref{T1.6}, it holds that for any $t>0$,
\begin{align}   \label{G5.44}
  \|\langle\tau\rangle^\alpha u\|_{\widetilde L_t^\infty(\dot B_{2,1}^{\frac{d}{2}-1})}^h+\|\langle\tau\rangle^\alpha z\|_{\widetilde L_t^\infty(\dot B_{2,1}^{\frac{d}{2}})}^h+\| \tau ^\alpha u\|_{\widetilde L_t^\infty(\dot B_{2,1}^{\frac{d}{2}+1})}^h\lesssim  \delta_*+\delta_*\mathcal{Z}(t),  
\end{align}
where $\delta_{*}$ and $\mathcal{Z}(t)$ are defined by                 \eqref{TE.4} and \eqref{G5.32}, respectively, $\alpha:=\frac{1}{2}(\frac{d}{2}+1-\sigma_0 )$.
\end{lem}
\begin{proof}
For any $k \geq -1$, it follows from \eqref{G3.9}--\eqref{G3.11} that
\begin{align*}
&\|\dot\Delta_k(u,\nabla z)\|_{L^2}\nonumber\\
&\quad  \lesssim e^{-t}\|\dot\Delta_k(u_0,\nabla z_0)\|_{L^2} +\int_0^t e^{-(t-\tau)} \big( \|\dot\Delta_k(f(a)\Delta u)\|_{L^2}+ \|\dot\Delta_k(f(a)\nabla{\rm div}\,u)\|_{L^2}  \big) {\rm d}\tau\nonumber\\
&\qquad+\int_0^t e^{-(t-\tau)} \big( \|\dot\Delta_k(f(a)\nabla z)\|_{L^2}+ 2^{k}\|\dot\Delta_k(z {\rm div}\,u)\|_{L^2} +\| \dot\Delta_k( u\cdot\nabla u)\|_{L^2} \big) {\rm d}\tau \nonumber\\
&\qquad+\int_0^te^{-(t-\tau)}\big(\|[u\cdot\nabla,\dot\Delta_k] z\|_{L^2}    +\sum_{j=1}^d\|[u\cdot\nabla,\partial_j\dot\Delta_k]z\|_{L^2}+2^k\|{\rm div}\,u\|_{L^\infty}\|\dot\Delta_k z\|_{L^2}\big){\rm d}\tau.
\end{align*}
Applying the time-weighted estimate to the above inequality, one has
\begin{align}\label{G5.45}
&\|\langle \tau\rangle^\alpha(u,\nabla z)\|_{\widetilde L_t^\infty(\dot B_{2,1}^{\frac{d}{2}-1})}^h\nonumber\\
 &\quad  \lesssim \|(u_0,\nabla z_0)\|_{\dot B_{2,1}^{\frac{d}{2}-1}}^h +\sum_{k\geq-1} \sup_{\tau\in[0,t]}\langle\tau\rangle^\alpha\int_0^t e^{-(\tau-s)}2^{k(\frac{d}{2}-1)}\sum_{i=1}^3 J_{i,k}\,{\rm d}s, 
\end{align}
where $J_{i,k}$ $(i=1,2,3)$ are given by
\begin{align*} 
J_{1,k}:=&\,\|\dot\Delta_k(u\cdot\nabla u)\|_{L^2}+2^k\|\dot\Delta_k(z{\rm div}\,u)\|_{L^2}+  \|[u\cdot\nabla,\dot\Delta_k] z\|_{L^2}    \nonumber\\
&+\sum_{j=1}^d\|[u\cdot\nabla,\partial_j\dot\Delta_k]z\|_{L^2}+2^k\|{\rm div}\,u\|_{L^\infty}\|\dot\Delta_k z\|_{L^2},\\
J_{2,k}:=&\, \|\dot\Delta_k(f(a)\Delta u)\|_{L^2}+\|\dot\Delta_k(f(a)\nabla{\rm div}\, u)\|_{L^2}, \\ 
J_{3,k}:=&\,\|\dot\Delta_k(f(a)\nabla z)\|_{L^2}.
\end{align*}
The nonlinear terms in $J_{1,k}$ can be handled similarly to those in \cite[Lemma 4.5]{LS-SIMA-2023}; here, we only present the resulting estimate:
\begin{align}\label{G5.46}
 \sum_{k\geq-1} \sup_{\tau\in[0,t]}\langle\tau\rangle^\alpha\int_0^t e^{-(\tau-s)}2^{k(\frac{d}{2}-1)}  J_{1,k}\,{\rm d}s \lesssim \delta_*+\delta_*\mathcal{Z}(t).
\end{align}
Since the density $a$ lacks a dissipative structure, the terms $J_{2,k}$ and $J_{3,k}$ need to be handled with particular care.
For $J_{2,k}$, following an argument similar to that in \cite[Lemma 4.5]{LNZ-2025-preprint}, we obtain  
\begin{align}\label{G5.47}
 \sum_{k\geq-1} \sup_{\tau\in[0,t]}\langle\tau\rangle^\alpha\int_0^t e^{-(\tau-s)}2^{k(\frac{d}{2}-1)}  J_{2,k}\,{\rm d}s \lesssim \delta_*+\delta_*\mathcal{Z}(t).
\end{align}

Therefore, it suffices to analyze the remaining term $J_{3,k}$.
We examine two cases: $t \leq 2$ and $t \geq 2$. For the case   $t < 2$,  we have 
\begin{align}\label{G5.48}
&\sum_{k\geq -1}\sup_{\tau\in[0,t]}\langle\tau \rangle^\alpha \int_0^\tau e^{-(\tau-s)} 2^{k(\frac{d}{2}-1)} J_{3,k} \,{\rm d}s\nonumber\\
\lesssim&\, \int_0^t \Big(\|f(a)\nabla z^\ell\|_{\dot B_{2,1}^{\frac{d}{2}-1}} +\|f(a)\nabla z^h\|_{\dot B_{2,1}^{\frac{d}{2}-1}} \Big) {\rm d}s\nonumber\\
\lesssim&\,  \Big( 1+\|a\|_{\widetilde L^\infty_t(\dot B_{2,1}^{\frac{d}{2}})}\Big)\|a\|_{\widetilde L^\infty_t(\dot B_{2,1}^{\frac{d}{2}-1})} \|z\|^{\ell}_{L_t^1(\dot B_{2,1}^{\frac{d}{2}+1})}  + \| a\|_{\widetilde L_t^\infty(\dot B_{2,1}^{\frac{d}{2}})} \|z\|^h_{  L_t^1(\dot B_{2,1}^{\frac{d}{2}})}  \nonumber\\ 
\lesssim&\,  \mathcal{E}^2(t).
\end{align}
For the case where $t \geq 2$, we divide the time interval $[0,t]$ into two subintervals: $[0, 1]$ and $[1, t]$. For the interval $[0,1]$, a direct calculation yields
\begin{align}\label{G5.49}
\sum_{k\geq -1}\sup_{\tau\in[2,t]}\langle\tau \rangle^\alpha \int_0^1 e^{-(\tau-s)} 2^{k(\frac{d}{2}-1)} J_{3,k} \,\mathrm{d}s  \lesssim   \mathcal{E}^2(1).
\end{align}
For the interval $[1,t]$, it follows from \eqref{G5.32} and Lemmas \ref{LA.3}--\ref{LA.4} that
\begin{align}\label{G5.50}
& \sum_{k\geq -1}\sup_{\tau\in[2,t]}\langle\tau \rangle^\alpha \int_1^\tau e^{-(\tau-s)} 2^{k(\frac{d}{2}-1)} J_{3,k}\, \mathrm{d}s\nonumber\\  
\lesssim&\, \Big( \| f(a)\|_{\widetilde L_t^\infty(\dot B_{2,1}^{\frac{d}{2}-1})} \|\langle\tau\rangle^\alpha   z^\ell\|_{\widetilde L_t^\infty(\dot B_{2,1}^{\frac{d}{2}+1})}+ \| a\|_{\widetilde L_t^\infty(\dot B_{2,1}^{\frac{d}{2}})} \|\langle\tau\rangle^\alpha  z^h\|_{\widetilde L_t^\infty(\dot B_{2,1}^{\frac{d}{2}})}\Big)\sup_{\tau\in [2,t]} \int_1^\tau e^{-(\tau-s)}s^{-\alpha} {\rm d}s\nonumber\\
\lesssim&\,  \mathcal{E}(t)\big(1+ \mathcal{E}(t)      \big) \Big(  \|\langle\tau\rangle^\alpha   z^\ell\|_{\widetilde L_t^\infty(\dot B_{2,1}^{\frac{d}{2}+1})}+ \|\langle\tau\rangle^\alpha  z^h\|_{\widetilde L_t^\infty(\dot B_{2,1}^{\frac{d}{2}})}\Big) \nonumber\\
 \lesssim&\, \delta_*  \mathcal{Z}(t),
\end{align}
where we have used the fact that
\begin{align*}
 \| \langle\tau\rangle^\alpha   z^\ell\|_{\widetilde L_t^\infty(\dot B_{2,1}^{\frac{d}{2}+1})}  \lesssim&\,   \big\| \langle\tau\rangle^{\frac{1}{2}(\frac{d}{2}+1-\sigma_0)}  z\big\|^\ell_{\widetilde L_t^\infty(\dot B_{2,1}^{\frac{d}{2}+1})}\lesssim \mathcal{Z}(t).
\end{align*}
Inserting the estimates  \eqref{G5.46}--\eqref{G5.50} into \eqref{G5.45}, we  conclude that
\begin{align}\label{G5.51}
 \|\langle \tau\rangle^\alpha(u,\nabla z)\|_{\widetilde L_t^\infty(\dot B_{2,1}^{\frac{d}{2}-1})}^h\lesssim \delta_*+\delta_*\mathcal{Z}(t),  
\end{align}
for any $t>0$. For the remaining $\dot B_{2,1}^{\frac{d}{2}+1}$-regularity estimate of $u$, multiplying \eqref{G5.27} by $t^\alpha$ and using Lemma \ref{LA.7} and \eqref{G5.51}, we deduce that 
\begin{align}\label{G5.52}
&\|\tau^\alpha u(t)\|_{\widetilde L^\infty_t(\dot B_{2,1}^{\frac{d}{2}+1})}^h\nonumber\\
\lesssim&\,  \|\langle\tau\rangle^{\alpha} z\|_{\widetilde L^\infty_t(\dot B_{2,1}^{\frac{d}{2}})}^h +\|\langle\tau\rangle^{\alpha} u\cdot\nabla u\|_{  \widetilde L^\infty_t(\dot B_{2,1}^{\frac{d}{2}-1})}^h+\|\langle\tau\rangle^{\alpha} f(a) \nabla z\|_{\widetilde L^\infty_t(\dot B_{2,1}^{\frac{d}{2}-1})}^h \nonumber\\
& +\|\langle\tau\rangle^{\alpha} f(a)\Delta u\|_{\widetilde L^\infty_t(\dot B_{2,1}^{\frac{d}{2}-1})}^h +\|\langle\tau\rangle^{\alpha} f(a)\nabla{\rm div}\,u\|_{\widetilde L_t^\infty(\dot B_{2,1}^{\frac{d}{2}-1})}^h+\|\langle\tau\rangle^{\alpha } u\|_{\widetilde L^\infty_t(\dot B_{2,1}^{\frac{d}{2}-1})}^h  \nonumber\\
\lesssim&\,\delta_*+\delta_*\mathcal{Z}(t),
\end{align}
for any $t>0$. Combining \eqref{G5.51} with \eqref{G5.52}, we directly obtain \eqref{G5.44}.
\end{proof}

\begin{proof}[Proof of Theorem \ref{T1.5}]
It follows from \eqref{G5.32}, \eqref{G5.36} and \eqref{G5.44} that
\begin{align*}
\mathcal{Z}(t)\lesssim  \delta_*+\delta_*\mathcal{Z}(t
)+\mathcal{Z}^2(t),   
\end{align*}
for any $t>0$.   From the smallness assumption \eqref{TF.1}, we directly obtain that $\delta_*$ is sufficiently small. This implies $\mathcal{Z}(t) \lesssim \delta_*$. Hence, the estimates \eqref{TF.3}--\eqref{TF.4} are valid, and the proof of Theorem \ref{T1.6} is completed.
\end{proof}

\subsection{The lower bound estimate of solutions} 
In this subsection, we investigate the lower bound of the decay estimates of $(u,z)$ associated with the Cauchy problem \eqref{A3}--\eqref{A3-1}.   

To begin with, we consider the following linearized problem 
(the  solution is still denoted by $(u,z)$ for notation simplicity) associated with the subsystem \eqref{A3}$_2$--\eqref{A3}$_3$:  
 \begin{equation} \label{G5.53}
\left\{  
\begin{aligned}  
& \partial_t u-\mu\Delta u-(\mu+\lambda)\nabla {\rm div}\,u+ \nabla z =0,\\
& \partial_t z +\gamma{\rm div}\,u=0, \\
&(u(0,x),z(0,x))=(u_0(x),z_0(x)).
\end{aligned}  
\right.  
\end{equation}  
Motivated by \cite{Danchin-00-IM}, we apply the so-called Hodge decomposition to the velocity field. The velocity $u$
is decomposed by introducing $m=\Lambda^{-1}{\rm div}\,u$ and $n=\Lambda^{-1}{\rm curl}\,u$. Using the identity   $\Delta = \nabla {\rm div}\, - {\rm curl}\,{\rm curl}$,  we obtain the following representation:  
\begin{align*}
u=-\Lambda^{-1}\nabla m-\Lambda^{-1}{\rm curl}\, n,    
\end{align*}
where ${\rm curl}_{ij}=\partial_ju^i-\partial_iu^j$. 
Then, the Cauchy problem \eqref{G5.53} can be rewritten as the
hyperbolic-parabolic equations
 \begin{equation} \label{G5.54}
\left\{  
\begin{aligned}  
& \partial_t m-(2\mu+\lambda)\Delta m -\Lambda z =0,\\
& \partial_t z +\gamma \Lambda m=0, \\
&(m(0,x),z(0,x))=(m_0(x),z_0(x))=(\Lambda^{-1}{\rm div}\,u_0(x),z_0(x
)),
\end{aligned}  
\right.  
\end{equation}  
and the heat equation
 \begin{equation} \label{G5.55}
\left\{  
\begin{aligned}  
& \partial_t n-\mu\Delta n =0,\\
& n(0,x)=n_0(x)=\Lambda^{-1}{\rm curl\,}u_0(x).
\end{aligned}  
\right.  
\end{equation}  

 Due to the dispersion property inherent in the hyperbolic component of the system \eqref{G5.54}, the decay theory developed in \cite{BS-2009-Adv,Bl-2016-SIMA,NS-2015-JLMS} is not applicable in this setting. Motivated by the approach in \cite{BSXZ-Adv-2024}, we derive point-wise estimates for $(u,z)$ associated with the system \eqref{G5.53}, which yield the following lemma.
\begin{lem}\label{L5.7}
Let $(u,z)$ satisfy \eqref{G5.53}. Then it holds that  
\begin{align}\label{G5.56}
|\widehat{u}(t,\xi)|+|\widehat{z}(t,\xi)|\gtrsim e^{-\max \left\{\mu+\frac{\lambda}{2},\mu \right\}|\xi|^2t}( \widehat{u}_0( \xi)|+|\widehat{z}_0( \xi)| ), \quad {\rm if}\quad |\xi|\leq  \beta,
\end{align}
where $\beta\in \big(0,\frac{2\sqrt{\gamma}}{2\mu+\lambda}\big]$ is sufficiently small.
\end{lem}
\begin{proof}
In the framework of semi-group theory, by defining $U=(m,z)^\top$, we can rewrite \eqref{G5.54} as  
\begin{align}\label{G5.57}
U_t=\mathcal{B}U,\quad U|_{t=0}=U_0,    
\end{align}
where the operator $\mathcal{B}$ is given by  
\begin{align*}\mathcal{B}= \left(
\begin{matrix}
(2\mu+\lambda)\Delta  &      \Lambda\\
-\gamma \Lambda  &     0\\    
\end{matrix}\right).
\end{align*}
Applying the Fourier transform to the system \eqref{G5.57} gives
 \begin{align}\label{G5.58}
\widehat{U}_t=\mathcal{A}(\xi)\widehat{U},\quad \widehat{U}|_{t=0}=\widehat{U}_0,    
\end{align}
with $\widehat{U}(t,\xi)=\mathcal{F}( U(t,x))$, $\xi=(\xi_1,\xi_2,\dots,\xi_d)^\top$, and $\mathcal{A}(\xi)$ is expressed by
\begin{align*}\mathcal{A}(\xi)=\left(
\begin{matrix}
-(2\mu+\lambda)|\xi|^2 &  |\xi| \\
-\gamma|\xi|          &  0
\end{matrix}\right).    
\end{align*}
Similar to \cite{HZ-IUMJ-1995}, through a direct calculation, we derive the explicit expression for the Green matrix $\mathcal{G}(t,\xi) := e^{t\mathcal{A}(\xi)}$ as
\begin{align*}\mathcal{G}(t,\xi)= \left(
\begin{matrix}
\frac{\lambda_{+}e^{\lambda_{+} t}-\lambda_{-}e^{\lambda_{-}t}}{\lambda_+-\lambda_-}  &    \Big( \frac{ e^{\lambda_{+} t}- e^{\lambda_{-}t}}{\lambda_+-\lambda_-} \Big)|\xi|\\
 -\gamma\Big(\frac{ e^{\lambda_{+} t}- e^{\lambda_{-}t}}{\lambda_+-\lambda_-}\Big) |\xi|  &     \frac{\lambda_{+}e^{\lambda_{-} t}-\lambda_{-}e^{\lambda_{+}t}}{\lambda_+-\lambda_-} \\    
\end{matrix}\right),
\end{align*}
with the eigenvalues  
\begin{align*}
\lambda_{\pm}(\xi)=-\Big({\mu+\frac{\lambda}{2}}\Big)|\xi|^2 \pm i\sqrt{\gamma|\xi|^2- {\Big(\mu+\frac{\lambda}{2}\Big)^2} |\xi|^4},  
\end{align*} 
for $|\xi|\leq \frac{2\sqrt{\gamma}}{2\mu+\lambda}$.
For simplicity of notation, we denote $q = \mu + \frac{\lambda}{2}$ and $r = \sqrt{\gamma|\xi|^2 - \left( \mu + \frac{\lambda}{2} \right)^2 |\xi|^4}$. Then, $\lambda_{\pm}(\xi) = -q|\xi|^2 \pm i r$ for $|\xi|\leq \frac{2\sqrt{\gamma}}{2\mu+\lambda}$. 

At low frequencies, where $|\xi| \leq \frac{2\sqrt{\gamma}}{2\mu + \lambda}$, we have
\begin{align*}
 \frac{\lambda_{+}e^{\lambda_{+} t}-\lambda_{-}e^{\lambda_{-}t}}{\lambda_+-\lambda_-} =&\, e^{-q |\xi|^2}\Big( {\rm cos}(rt)-q\frac{{\rm sin}(rt)}{r}|\xi|^2   \Big), \nonumber\\  
  \frac{\lambda_{+}e^{\lambda_{-} t}-\lambda_{-}e^{\lambda_{+}t}}{\lambda_+-\lambda_-} =&\, e^{-q |\xi|^2}\Big( {\rm cos}(rt)+q\frac{{\rm sin}(rt)}{r}|\xi|^2   \Big), \nonumber\\
   \frac{ e^{\lambda_{+} t}- e^{\lambda_{-}t}}{\lambda_+-\lambda_-}  =&\, e^{-q|\xi|^2}\frac{{\rm sin}(rt)}{r}.
\end{align*}
Therefore, from \eqref{G5.58}, we can define $\widehat{m}^*(t,\xi)$ and $\widehat{z}\,^*(t,\xi)$ as follows:
\begin{align}\label{G5.59}
\widehat{m}^*(t,\xi)=:&\, e^{q |\xi|^2}\widehat{m}(t,\xi)=    \Big( {\rm cos}(rt)-q\frac{{\rm sin}(rt)}{r}|\xi|^2   \Big) \widehat{m}_0+ \frac{{\rm sin}(rt)}{r}|\xi|\widehat{z}_0, \\ \label{G5.60}
\widehat{z}\,^*(t,\xi)=:&\,e^{q |\xi|^2}\widehat{z}(t,\xi)=  -  \frac{{\rm sin}(rt)}{r}|\xi|\widehat{m}_0+ \Big( {\rm cos}(rt)+q\frac{{\rm sin}(rt)}{r}|\xi|^2   \Big) \widehat{z}_0. 
\end{align}
Following a similar approach to the proof of \cite[Lemma 3.1]{BSXZ-Adv-2024}, a direct computation yields 
\begin{align}\label{G5.61}
 &|\widehat{m}^*(t,\xi)|^2+  | \widehat{z}\,^*(t,\xi)|^2\nonumber\\
 =&\, \Big( |{\rm cos}(rt)|^2+q^2\frac{|{\rm sin}(rt)|^2}{|r|^2}|\xi|^4       -q\frac{{\rm sin}(2rt) }{r}|\xi|^2 + \frac{|{\rm sin}(rt)|^2}{|r|^2}|\xi|^2    \Big) |\widehat{ m}_0(\xi)|^2 \nonumber\\
 &+ \Big( |{\rm cos}(rt)|^2+q^2\frac{|{\rm sin}(rt)|^2}{|r|^2}|\xi|^4       +q\frac{{\rm sin}(2rt) }{r}|\xi|^2 + \frac{|{\rm sin}(rt)|^2}{|r|^2}|\xi|^2    \Big) |\widehat{ z}_0(\xi)|^2\nonumber\\
 &-2q\frac{|{\rm sin}(rt)|^2}{r^2}|\xi|^3\Big( \overline{{\widehat{m}_0}(\xi)}\widehat{z}_0(\xi)+{{\widehat{m}_0}(\xi)}\overline{\widehat{z}_0(\xi)}   \Big)\nonumber\\
 \geq&\, (1-C_q|\xi|) (|\widehat{m}_0(\xi)|^2+|\widehat{z}_0(\xi)|^2)-C_q|\xi||\widehat{m}_0(\xi)||\widehat{z}_0(\xi)|\nonumber\\
 \geq&\, \frac{1}{4}(|\widehat{m}_0(\xi)|^2+|\widehat{z}_0(\xi)|^2),
\end{align}
for any $|\xi|\leq \beta:=\min\big\{\frac{1}{2C_q},\frac{2\sqrt{\gamma}}{2\mu+\lambda}\big\}$, where $C_q$ is a constant that depends only on $q$. Combining \eqref{G5.59} and \eqref{G5.60} with the estimate \eqref{G5.61}, we eventually obtain
\begin{align}\label{G5.62}
|\widehat{m}(t,\xi)|^2+   |\widehat{z}(t,\xi)|^2 \geq \frac{1}{4}e^{-q|\xi|^2}(|\widehat{m}_0(\xi)|^2+|\widehat{z}_0(\xi)|^2).
\end{align}
On the other hand, for the system \eqref{G5.55}, by leveraging the properties of the heat equation (see Appendix in \cite{BV-2022}), we arrive at
\begin{align}\label{G5.63}  
|\hat{n}(t,\xi)|^2= e^{-\mu|\xi|^2 t} |\hat{n}_0(\xi)|^2, 
\end{align}
for any $|\xi|\leq \frac{2\sqrt{\gamma}}{2\mu+\lambda}$.
Adding  \eqref{G5.62} and \eqref{G5.63} up yields
\begin{align*}
|\widehat{u}(t,\xi)|^2+    |\widehat{z}(t,\xi)|^2\geq \frac{1}{4} \max\left\{e^{-\mu|\xi|^2}, e^{-\left(\mu+\frac{\lambda}{2}\right)|\xi|^2}\right  \}(|\widehat{m}_0(\xi)|^2+|\widehat{n}_0(\xi)|^2+|\widehat{z}_0(\xi)|^2),
\end{align*}
for $|\xi| \leq \beta$. Together with the equivalence relation $|\widehat{u}(t,\xi)|^2 \sim |\widehat{m}(\xi)|^2 + |\widehat{n}(\xi)|^2$, this implies that \eqref{G5.56} holds.    
\end{proof}

Below, we present the decay estimates for the linearized system \eqref{G5.53} as follows.

\begin{lem}[Linear analysis]\label{L5.8}
  Let  $d\geq 3$.  Assume that $(u_{L},z_L)$ is the global  solution to the   Cauchy problem \eqref{G5.53}, and the initial data $ (u_0,z_0) $ satisfies
\begin{align*}
  u_0^\ell,z_0^\ell\in \dot {\mathfrak B}_{2,\infty}^{\sigma_0}\quad \text{with}\quad \sigma_0\in\Big[-\frac{d}{2},\frac{d}{2}-2\Big),
\end{align*}  
then for all $t \geq 1$, there exists two universal constants $c_6>0$ and  $C_6 > 0$ such that  
\begin{align} \label{G5.64}
c_6(1+t)^{-\frac{1}{2}(\sigma-\sigma_0)}\leq \|(u_{L},z_L)(t)\|_{\dot B_{2,1}^\sigma}\leq C_6  (1+t)^{-\frac{1}{2}(\sigma-\sigma_0)},    
\end{align}
for all $\sigma\in\big(\sigma_0,\frac{d}{2}\big]$, where $\dot{\mathfrak{B}}_{2,\infty}^{\sigma_0}$ is defined by \eqref{G1.12}.
\end{lem}
\begin{proof}
The upper bound of the decay estimate for $\|(u_L,z_L)\|_{\dot B_{2,1}^{\sigma}}$, where $\sigma\in \big(\sigma_0,\frac{d}{2}\big]$, follows directly from Theorem \ref{T1.6}; hence, we omit the details for brevity. 
Inspired by \cite{Bl-2016-SIMA} and \cite[Section 3]{BSXZ-Adv-2024}, below we only need to establish the remaining lower bound estimates of $(u_L, z_L)$ for the problem \eqref{G5.53}.

Without loss of generality, for $\{k_j\}_{j\in\mathbb{N}} \subset \mathbb{Z}$, we assume that $j = 1, 2, \dots,$  is less than $[\log_2 \beta]$ for the $\beta$ defined  in \eqref{G5.56}.
Applying  the Fourier-Plancherel theorem and  \eqref{G5.56} yields
\begin{align*}
\|(u_L,z_L)(t)\|_{\dot B_{2,1}^\sigma}\geq&\,  \|(u_L^\ell,z_L^\ell)(t)\|_{\dot B_{2,1}^\sigma}\nonumber\\
\geq&\,  \sum_{k\leq [\rm log_2\beta]} 2^{\sigma k} \|\dot\Delta_k (u_L,z_L)(t)\|_{L^2}   \nonumber\\
\geq&\, \sum_{k\leq [\rm log_2\beta]}2^{\sigma k} \|\phi (2^{-k}\cdot)(\widehat u_L,\widehat{z}_L)(t)\|_{L^2{}}\nonumber\\
\geq&\, \sum_{k\leq [\rm log_2\beta]}e^{-\frac{64\max\left\{\mu+\frac{\lambda}{2},\mu\right\} 2^{2k}t}{9}}2^{\sigma k} \|\dot\Delta_k (u_0,z_0)\|_{L^2},
\end{align*}
for any fixed $t\geq 1$.  From \eqref{G1.12}, there exists a maximal integer $k_{j_0}$ satisfying 
$k_{j_0} \leq -\frac{1}{2} \log_2(1+t).$
We claim that $k_{j_0} > -M_0 - \frac{1}{2} \log_2(1+t)$. 
Suppose, to the contrary, that there exists an integer $k_{j_0-1}$ satisfying 
$$k_{j_0-1} \leq k_{j_0} + M_0 \leq -\frac{1}{2} \log_2(1+t).$$
 Then $k_{j_0-1}$ would be a larger candidate than $k_{j_0}$ still satisfying the inequality, which contradicts the maximality of $k_{j_0}$. Noting that $2^{k_{j_0}} \sim \langle t \rangle^{-\frac{1}{2}}$, we further derive
\begin{align*}
\|(u_L,z_L)(t)\|_{\dot B_{2,1}^\sigma} \gtrsim&\,  \|(u_L^\ell,z_L^\ell)(t)\|_{\dot B_{2,1}^\sigma} \nonumber\\
\gtrsim&\, \sum_{k\leq [\rm log_2\beta]}e^{-\frac{64\max\left\{\mu+\frac{\lambda}{2}, \mu \right\} 2^{2k}t}{9}}2^{\sigma k} \|\dot\Delta_k (u_0,z_0)\|_{L^2}\nonumber\\
\gtrsim&\, e^{-\frac{64\max\left\{\mu+\frac{\lambda}{2}, \mu \right\} 2^{2 k_{j_0}} t  } {9}} 2^{(\sigma-\sigma_0)k_{j_0}} 2^{\sigma_0 k_{j_0}} \|\dot\Delta_k (u_0,z_0)\|_{L^2}\nonumber\\
\gtrsim&\,  c_02^{(\sigma-\sigma_0)k_{j_0}}\nonumber\\
\gtrsim&\, c_0(1+t)^{-\frac{1}{2}(\sigma-\sigma_0)},
\end{align*}
for all $t\geq 1$. Hence, we complete the proof of \eqref{G5.64}.
\end{proof}

To investigate the nonlinear components of the system \eqref{A3}$_2$--\eqref{A3}$_3$, we define the  variables
\begin{align*}
(u_N, z_N) := (u - u_L, z - z_L),    
\end{align*}
and subsequently analyze the following nonlinear Cauchy problem associated with $(u_N,z_N)$:  
\begin{equation}  \label{G5.65}
\left\{  
\begin{aligned}  
& \partial_t u_N - \mu\Delta u_N - (\mu+\lambda)\nabla {\rm div}\, u_N +\nabla z_N= F_1, \\  
& \partial_t z_N+\gamma{\rm div}\, u_N=F_2       ,\\
& (u_N(0,x),z_N(0,x)  )= (0,0),  
\end{aligned}  
\right.  
\end{equation}  
with
\begin{equation*}  
\left\{  
\begin{aligned}  
&  F_1:=-u\cdot\nabla u-f(a)\nabla z+f(a)(\mu\Delta u+(\mu+\lambda)\nabla {\rm div}\,u), \\  
& F_2:=-{\rm div}\,(uz)  -(\gamma-1) z{\rm div}\,u.  
\end{aligned}  
\right.  
\end{equation*}  

\begin{lem}[Nonlinear analysis]\label{L5.9}
Let $d\geq 3$. Assume that $(u_N,z_N)$ is
a solution to the nonlinear Cauchy problem \eqref{G5.65},
and the initial data $(a_0, u_0,z_0)$  defined in \eqref{A3-1}  satisfies
\begin{align*}
  a_0^\ell\in \dot B_{2,\infty}^{\sigma_0},\quad u_0^\ell,z_0^\ell\in \dot {\mathfrak B}_{2,\infty}^{\sigma_0},\quad \|(u_0^\ell,z_0^\ell)\|_{\dot B_{2,\infty}^{\sigma_0}}\leq\eps_3,\quad \text{with}\quad \sigma_0\in\Big[-\frac{d}{2},\frac{d}{2}-2\Big),
\end{align*}
where $\eps_3$  is sufficiently small positive constant. Then  it holds that for all $t > 0$,
\begin{align*} 
\|(u_N,z_N)(t) \|_{\dot B_{2,1}^\sigma}^\ell   \lesssim \delta_{*}^2  (1+t)^{-\frac{\sigma-\sigma_0}{2}},
\end{align*}
for all $\sigma\in\big(\sigma_0,\frac{d}{2} \big]$, where $\dot{\mathfrak{B}}_{2,\infty}^{\sigma_0}$ is defined via \eqref{G1.12}.
\end{lem}

\begin{proof}
As in \eqref{G3.9}--\eqref{G3.11}, we only need to analyze the low-frequency component. For $k \leq 0$, it follows that
\begin{align}\label{G5.66}
&\frac{{\rm d}}{{\rm d}t}\|\dot\Delta_k(u_N,z_N,\nabla z_N)\|_{L^2}^2 + 2^{2k}\|\dot\Delta_k(u_N,z_N,\nabla z_N)\|_{L^2}^2\nonumber\\
&\quad\lesssim  \|\dot\Delta_k(u_N,z_N,\nabla z_N)\|_{L^2} \|\dot\Delta_k (F_1,F_2)\|_{L^2},
\end{align}
where the inequality $2^k \lesssim 1$ has been utilized.
Applying Gr\"{o}nwall’s inequality to \eqref{G5.66}, together with the initial condition \eqref{G5.65}$_3$, yields  
\begin{align*} 
\|\dot\Delta_k (u_N,z_N,\nabla z_N)\|_{L^2}   \lesssim \int_0^t e^{-2^{2k}(t-\tau)}\|\dot\Delta_k (F_1,F_2)\|_{L^2}{\rm d}\tau , 
\end{align*}
which implies that 
\begin{align*}
\|(u_N,z_N)\|_{\dot B_{2,1}^{\sigma}}^\ell \lesssim&\, \int_0^t \langle t-\tau\rangle^{-\frac{1}{2}(\sigma-\sigma_0)} \|(F_1,F_2)\|_{\dot B_{2,\infty}^{\sigma_0}}^\ell {\rm d}\tau \nonumber\\    
\lesssim&\,  \int_0^t \langle t-\tau\rangle^{-\frac{1}{2}(\sigma-\sigma_0)} \Big(\|u\cdot\nabla u\|_{\dot B_{2,\infty}^{\sigma_0}}^\ell+\|f(a)\Delta u\|_{\dot B_{2,\infty}^{\sigma_0}}^\ell+\|f(a)\nabla {\rm div}\, u\|_{\dot B_{2,\infty}^{\sigma_0}}^\ell     \Big) {\rm d}\tau\nonumber\\
&+\int_0^t \langle t-\tau\rangle^{-\frac{1}{2}(\sigma-\sigma_0)} \Big( \|f(a)\nabla z\|_{\dot B_{2,\infty}^{\sigma_0}}^\ell+\|uz\|_{\dot B_{2,\infty}^{\sigma_0+1}}^\ell+\|z{\rm div}\,u\|_{\dot B_{2,\infty}^{\sigma_0 }}^\ell \Big) {\rm d}\tau,
\end{align*}
for any $\sigma\in\big( \sigma_0,\frac{d}{2} \big]$.
Taking the same estimates on the right-hand side of \eqref{G5.37} as those in Lemma \ref{L5.5}, we conclude that  
\begin{align*}
\|(u_N,z_N)\|_{\dot B_{2,1}^{\sigma}}^\ell \lesssim (\delta_*\mathcal{Z}(t) + \mathcal{Z}^2(t)) \langle t\rangle^{-\frac{1}{2}(\sigma-\sigma_0)} \lesssim \delta_*^2 (1+t)^{-\frac{1}{2}(\sigma-\sigma_0)},
\end{align*}
for any $t \geq 1$. Thus, the proof of Lemma \ref{L5.9} is completed.
\end{proof}

Finally, by making use of Lemmas \ref{L5.8}--\ref{L5.9}, we are now in a position to  prove Theorem \ref{T1.7}.

\begin{proof}[Proof of Theorem \ref{T1.7}]
Since $(u,z) = (u_L + u_N, z_L + z_N)$, by Duhamel's principle, we have  
\begin{align*}
\|(u,z)(t)\|_{\dot B_{2,1}^\sigma} \geq  \|(u^\ell,z^\ell)(t)\|_{\dot B_{2,1}^\sigma} 
\geq&\, \|(u^\ell_L,z_L^\ell)(t)\|_{\dot B^\sigma_{2,1}} -\|(u_N,z_N)(t)\|_{\dot B_{2,1}^\sigma}^\ell\nonumber\\
\geq&\, c_0 (1+t)^{-\frac{1}{2}(\sigma-\sigma_0)}-\delta_*^2(1+t)^{-\frac{1}{2}(\sigma-\sigma_0)}\nonumber\\
\geq&\, c_7(1+t)^{-\frac{1}{2}(\sigma-\sigma_0)},
\end{align*}
for any $\sigma \in \big(\sigma_0, \frac{d}{2}\big]$ and any $t \geq 1$, where $c_7 > 0$ is a universal constant. Moreover, Theorem \ref{T1.6} yields for any $\sigma\in\big(\sigma_0,\frac{d}{2}\big]$,
\begin{align*}
\|(u,z)(t)\|_{\dot B_{2,1}^\sigma} \lesssim \delta_* (1+t)^{-\frac{1}{2}(\sigma - \sigma_0)},
\end{align*}
for all $t \geq 1$. Thus, \eqref{TG.1} is established, which completes the proof of Theorem \ref{T1.7}.

\end{proof}

\appendix
\section{Analytic tools}
In this appendix, we present several useful lemmas that are frequently employed throughout the paper.

We now state several fundamental properties of Besov spaces along with product estimates. It is important to note that these properties remain valid for Chemin-Lerner type spaces, provided that the time exponent satisfies H\"{o}lder's inequality with respect to the time variable. The first lemma concerns Bernstein's inequalities.

\begin{lem}\label{LA.1}{\rm(\!\!\cite{BCD-Book-2011})}
Let $0<r<R$, $1\leq p\leq q\leq \infty$ and $m\in \mathbb{N}$. Define the ball $\mathcal{B}=\{\xi\in\mathbb{R}^{3}~| ~|\xi|\leq  R\}$ and the annulus $\mathcal{C}=\{\xi\in\mathbb{R}^{3}~|~ \lambda_1 r\leq |\xi|\leq \lambda_1 R\}$ . For any $ u\in L^p$ and $\lambda_1>0$, it holds that
\begin{equation*} \left\{
\begin{aligned}
{\rm{supp}}\, \mathcal{F}(u) \subset& \lambda_1 \mathcal{B} \Rightarrow \|D^{m}u\|_{L^q}\lesssim\lambda_1^{m+d(\frac{1}{p}-\frac{1}q{})}\|u\|_{L^p}, \nonumber\\
{\rm{supp}}\, \mathcal{F}(u) \subset& \lambda_1 \mathcal{C} \Rightarrow \lambda_1^{m}\|u\|_{L^{p}}\lesssim\|D^{m}u\|_{L^{p}}\lesssim \lambda_1^{m}\|u\|_{L^{p}}.
\end{aligned}\right.
\end{equation*}
 \end{lem}

By virtue of Lemma \ref{LA.1}, the following properties of Besov spaces can be readily established.
\begin{lem}\label{LA.2}{\rm(\!\!\cite[Chapter 2]{BCD-Book-2011})}
The following properties hold:
\begin{itemize}
\item{} For $s\in\mathbb{R}$, $1\leq p_{1}\leq p_{2}\leq \infty$ and $1\leq r_{1}\leq r_{2}\leq \infty$, it holds that
\begin{equation}\notag
\begin{aligned}
\dot{B}^{s}_{p_{1},r_{1}}\hookrightarrow \dot{B}^{s-d (\frac{1}{p_1} -\frac{1}{p_2})}_{p_{2},r_{2}}.
\end{aligned}
\end{equation}
\item{} For $1\leq p\leq q\leq\infty$, we have the   chain of continuous embedding  as follows:
\begin{equation}\nonumber
\begin{aligned}
  \dot{B}^{0}_{p,1}\hookrightarrow L^{p}\hookrightarrow \dot{B}^{0}_{p,\infty}\hookrightarrow \dot B_{q,\infty}^{\varsigma},\quad \varsigma=-d\Big(\frac{1}{p}-\frac{1}{q}\Big).
\end{aligned}
\end{equation}
\item{} If $p<\infty$, then $\dot{B}^{\frac{d}{p}}_{p,1}$ is continuously embedded in the set of continuous functions decaying to 0 at infinity.

\item{}  The following real interpolation property is satisfied for $1\leq p\leq \infty$, $s_1<s_2$, and $\theta\in (0,1)$:
\begin{align}\label{A.1}
\|u\|_{\dot B_{p,1}^{\theta s_1+(1-\theta)s_2}}\lesssim \frac{1}{\theta(1-\theta)(s_2-s_1)}\|u\|_{\dot B_{p,\infty}^{s_1}}^{\theta} \|u\|_{\dot B_{p,1}^{s_2}}^{1-\theta}  .  
\end{align}

\item{}  For any $\lambda_1 > 0$, $s \in \mathbb{R}$, and $1 \leq p, r \leq \infty$,  the  following scaling relationship holds:
\begin{align} \label{scale}
  \|u(\lambda_1\cdot)\|_{\dot B_{p,r}^s}\backsim \lambda_1^{s-\frac{d}{p}}\|u\|_{\dot B_{p,r}^s}.  
\end{align}

\item{}  For any $\varepsilon>0$, it holds that
\begin{equation}\nonumber
\begin{aligned}
H^{s+\varepsilon}\hookrightarrow \dot{B}^{s}_{2,1}\hookrightarrow \dot{H}^{s}.
\end{aligned}
\end{equation}
\item{}
Let $\Lambda^{\sigma}$ be defined by $\Lambda^{\sigma}:=(-\Delta )^{\frac{\sigma}{2}}u:=\mathcal{F}^{-1}\big{(} |\xi|^{\sigma}\mathcal{F}(u) \big{)}$ for $\sigma\in \mathbb{R}$ and $u\in{\mathcal S}^{'}_{h}(\mathbb{R}^3)$, then $\Lambda^{\sigma}$ is an isomorphism from $\dot{B}^{s}_{p,r}$ to $\dot{B}^{s-\sigma}_{p,r}$.
\item{} Let $1\leq p_{1},p_{2},r_{1},r_{2}\leq \infty$, $s_{1}\in\mathbb{R}$ and $s_{2}\in\mathbb{R}$ satisfy
\begin{align*}
 s_{2}<\frac{d}{p_2}\quad\text{\text{or}}\quad s_{2}=\frac{d}{p_2}~\text{and}~r_{2}=1.
\end{align*}
   Then the  space $\dot{B}^{s_{1}}_{p_{1},r_{1}}\cap \dot{B}^{s_{2}}_{p_{2},r_{2}}$ endowed with the norm $\|\cdot \|_{\dot{B}^{s_{1}}_{p_{1},r_{1}}}+\|\cdot\|_{\dot{B}^{s_{2}}_{p_{2},r_{2}}}$ is a Banach space and has the weak compact and Fatou properties$:$ If $u_{n}$ is a uniformly bounded sequence of $\dot{B}^{s_{1}}_{p_{1},r_{1}}\cap \dot{B}^{s_{2}}_{p_{2},r_{2}}$, then an element $u$ of $\dot{B}^{s_{1}}_{p_{1},r_{1}}\!\cap \dot{B}^{s_{2}}_{p_{2},r_{2}}$ and a subsequence $u_{n_{k}}$ exist such that $u_{n_{k}}\rightarrow u$ in $\mathcal{S}'$ and
    \begin{align*} 
    \begin{aligned}
    \|u\|_{\dot{B}^{s_{1}}_{p_{1},r_{1}}\cap \dot{B}^{s_{2}}_{p_{2},r_{2}}}\lesssim \liminf_{n_{k}\rightarrow \infty} \|u_{n_{k}}\|_{\dot{B}^{s_{1}}_{p_{1},r_{1}}\cap \dot{B}^{s_{2}}_{p_{2},r_{2}}}.
    \end{aligned}
    \end{align*}
\end{itemize}
\end{lem}

To effectively control the nonlinear terms, we require the following Morse-type product estimates.
\begin{lem} \label{LA.3}{\rm(\!\!\cite[Chapter 2]{BCD-Book-2011})}
The following statements hold:
\begin{itemize}
\item{} Let $s>0$, $1\leq p,r\leq \infty$.  Then $\dot B^{s}_{p,r}\cap L^\infty$ is an algebra and
 \begin{align*} 
\|uv\|_{\dot B_{p,r}^s}\lesssim \|u\|_{L^\infty}\|v\|_{\dot B_{p,r}^s}+\|v\|_{L^\infty}\|u\|_{\dot B_{p,r}^s},    
 \end{align*}
which in particular leads to
\begin{align*}
\|uv\|_{\dot B_{p,1}^{\frac{d}{p}}}\lesssim \|u\|_{\dot B_{p,1}^{\frac{d}{p}}}\|v\|_{\dot B_{p,1}^{\frac{d}{p}}} .    
\end{align*}

\item{} Let the real numbers $s_1$, $s_2$, $p$ and $r$ fulfill
\begin{align*}
 1\leq p, r\leq\infty, \quad s_1\leq\frac{d}{p},\quad s_2\leq\frac{d}{p},\quad s_1+s_2>0.
\end{align*}
  Then, it holds  that
 \begin{align*} 
\|uv\|_{\dot B_{p,r}^{s_1+s_2-\frac{d}{p}}}\lesssim \|u\|_{ \dot B_{p,1}^{s_1 }}\|v\|_{\dot B_{p,1}^{s_2}} .    
 \end{align*}

\item{} Let  the real numbers $s_1$, $s_2$, $p$ and $r$ fulfill
\begin{align*}
 s_1\leq \frac{d}{p},\quad s_2< \frac{d}{p},\quad s_1+s_2\geq 0.   
\end{align*}
Then, we have
 \begin{align*} 
\|uv\|_{\dot B_{p,\infty}^{s_1+s_2-\frac{d}{p}}}\lesssim \|u\|_{ \dot B_{p,1}^{s_1 }}\|v\|_{\dot B_{p,\infty}^{s_2}}.    
 \end{align*}
    \end{itemize}
\end{lem}

We present the following lemma regarding the continuity of composite functions:

\begin{lem}\label{LA.4}{\rm(\!\!\cite[Chapter 2]{BCD-Book-2011})}
Let $F:I\rightarrow \mathbb{R}$ be a smooth function such that $F(0)=0$, where $I$ is an open interval in $\mathbb{R}$ containing $0$.
Then, for any $1\leq p,r\leq \infty$, $s>0$, and $u\in \dot{B}_{p,r}^s \cap L^\infty$, it holds that $F(u) \in \dot{B}_{p,r}^s \cap L^\infty$, and
\begin{align}\label{A.2}
\|F(u)\|_{\dot{B}^{s}_{p,r}} \leq C_u \|u\|_{\dot{B}_{p,r}^s},
\end{align}
where the constant $C_u > 0$ depends only on $\|u\|_{L^\infty}$, $F'$, $p$, $s$, and   $d$.
\end{lem}

Next, we present the following commutator {estimates}.
\begin{lem}\label{LA.5}
Let $1\leq p\leq \infty$ and $-\frac{d}{p}<s\leq 1+\frac{d}{p}$. Then, it holds that
\begin{align}\label{A.3}
\sum_{k\in\mathbb Z} 2^{ks}\|[u\cdot\nabla,\dot\Delta_k]a\|_{L^p}\lesssim&\, \|u\|_{\dot B^{\frac{d}{p}+1}_{p,1}}\|a\|_{\dot B_{p,1}^s}, \\\label{A.3-1}
\sum_{k\in\mathbb Z} 2^{k(s-1)}\|[u\cdot\nabla,\partial_j\dot\Delta_k]a\|_{L^p}\lesssim&\, \|u\|_{\dot B^{\frac{d}{p}+1}_{p,1}}\|a\|_{\dot B_{p,1}^s}, \quad j=1,2,\dots,d,
\end{align}
with the commutator $[A,B]:=AB-BA$.
\end{lem}

We study the following linear transport equation:
\begin{align} \label{A.4}
\left\{\begin{aligned}
&\partial_t \rho+\rho\cdot\nabla u=f, \quad\, x\in \mathbb{R}^d,\quad t>0,  \\
& \rho(x,0)=\rho_0(x),\quad \,\,\,\,\,  x\in \mathbb{R}^d.
 \end{aligned}
 \right.
\end{align}

\begin{lem}\label{LA.6}{\rm(\!\!\!\cite[Chapter 3]{BCD-Book-2011})}
Let $ T > 0 $, $ -\frac{d}{2} < s \leq \frac{d}{2} + 1 $,
 $ 1 \leq r \leq \infty $, $ \rho_0 \in \dot{B}_{2,r}^s $, $ u \in L^1(0,T; \dot{B}_{2,1}^{\frac{d}{2}+1}) $, 
 and $ f \in L^1(0,T; \dot{B}_{2,r}^s) $. Then there exists a constant $ C > 0 $, 
 independent of $ T $ and $ \rho_0 $, such that the solution $ \rho $ to \eqref{A.4} satisfies
\begin{align}\label{A.5}
\|\rho\|_{\widetilde L_T^\infty(\dot B_{2,r}^s)} 
\leq \exp\bigg\{C \|\nabla u\|_{L_T^1(\dot B_{2,1}^{\frac{d}{2}})} \bigg\} 
\bigg( \|\rho_0\|_{\dot B_{2,r}^s} + \int_0^T \|f\|_{\dot B_{2,r}^s}  {\rm d}\tau\bigg).
\end{align}
Moreover, if $ r < \infty $, then the solution $ \rho \in \mathcal{C}([0,T]; \dot B_{2,r}^s) $.
\end{lem}

We require the optimal regularity estimate for the following Lam\'{e} system:
\begin{align} \label{A.6}
\left\{\begin{aligned}
&\partial_t u-\mu\Delta  u-(\mu+\lambda)\nabla{\rm div}\,u=g, \quad\, x\in \mathbb{R}^d,\quad t>0,  \\
& u(x,0)=u_0(x),\quad \,\,\,\quad\quad\quad\quad\quad\quad\quad  x\in \mathbb{R}^d.
 \end{aligned}
 \right.
\end{align}
\begin{lem}\label{LA.7}{\rm(\!\!\cite[Chapter 3]{BCD-Book-2011})}
Let $ T > 0 $, $ \mu > 0 $, $ 2\mu +  \lambda > 0 $, $ s \in \mathbb{R} $, $ 1 \leq p, r \leq \infty $, and $ 1 \leq \varsigma_2 \leq \varsigma_1 \leq \infty $. Suppose that $ u_0 \in \dot{B}^{s}_{p,r} $ and $ g \in \widetilde{L}^{\varsigma_2}(0,T;\dot{B}_{p,r}^{s-2+\frac{2}{\varsigma_2}}) $. Then there exists a solution $u $ to \eqref{A.6} satisfying
\begin{align*}
\min\{\mu,2\mu+\lambda\}^{\frac{1}{\varsigma_1}}\|u\|_{\widetilde{L}^{\varsigma_1}_T(\dot B^{s+\frac{2}{\varsigma_1}}_{p,r})}\lesssim \|u_0\|_{\dot B_{p,r}^s}+\min\{\mu,2\mu+\lambda\}^{\frac{1}{\varsigma_2}-1}\|g\|_{L^{\varsigma_2}_T(\dot B_{p,r}^{s-2+\frac{2}{\varsigma_2}})}.
\end{align*}
\end{lem}

To investigate the low Mach number  limit of the compressible NST system \eqref{A4}, we consider the  following linear problem:
\begin{align} \label{A.7}
\left\{\begin{aligned}
&\partial_t  b+\frac{1}{\varepsilon} \Lambda d=f,\quad\quad\quad\quad\quad x\in\mathbb R^d,\quad t>0,  \\
&  \partial_t d-\frac{1}{\varepsilon}\Lambda b=g,\quad\quad\quad\quad\quad x\in\mathbb R^d,\quad t>0,  \\
&(b,d)(0,x)=(b_0,d_0)(x),\, \quad x\in\mathbb R^d.
 \end{aligned}
 \right.
\end{align}
\begin{lem}\label{LA.8}{\rm(\! \!\!\!\cite[Proposition 2.2]{Danchin-2002})}
Let $ T > 0 $, $ s \in \mathbb{R} $, $ p \geq 2 $,   $ \frac{2}{r} \leq \min\big\{1, (d-1)\big(\frac{1}{2} - \frac{1}{p}\big)\big\} $ and $ (r, p, d) \neq (2, \infty, 3) $. 
Assume that $ b_0, d_0 \in \dot{B}_{2,1}^s $ and $ f, g \in L^1(0,T; \dot{B}_{2,1}^s) $. Let $ (b, d) $ be a solution to \eqref{A.7}. Then it holds

\begin{align*}
\|(b,d)\|_{L^r_T(\dot B_{p,1}^{s+d\big(\frac{1}{p}-\frac{1}{2}\big)+\frac{1}{r}})}\lesssim \varepsilon^{\frac{1}{r}}\Big( \|(b_0,d_0)\|_{\dot B^s_{2,1}}+\|(f,g)\|_{L_T^1(\dot B_{2,1}^s)}      \Big).   
\end{align*}
    
\end{lem}

Finally, in order to obtain the optimal time decay rates of the solutions, we need  the following inequality:
\begin{lem}\label{LA.9} {\rm(\!\!\cite[Section 5]{Danchin-2018})}
Let $0<\beta_1\leq\beta_2$. If in addition $\beta_2>1$, then, it holds that
\begin{align}\label{A.8}
\int_0^t \langle t-\tau\rangle^{-\beta_1} \langle \tau\rangle^{-\beta_2} {\rm d}\tau\lesssim \langle  t\rangle^{-\beta_1}.    
\end{align}
\end{lem}

\bigskip 
{\bf Acknowledgements:} 
Li, Ni and  Wang are supported by NSFC (Grant No. 12331007).  
And Li is also supported by the ``333 Project" of Jiangsu Province.

\vspace{2mm}

\textbf{Conflict of interest.} The authors do not have any possible conflicts of interest.

\vspace{2mm}

\textbf{Data availability statement.}
 Data sharing is not applicable to this article as no data sets were generated or analyzed during the current study.

\vspace{2mm}

\bibliographystyle{plain}

\begin{thebibliography}{aaa}
%\bibitem{AlazardADE2005}
%T. Alazard, Incompressible limit of the nonisentropic Euler equations with the
%solid wall boundary conditions, {\it Adv. Differential Equations} {\bf 10 (1)} (2005)  19--44.
	
\bibitem{AlazardARMA2006}
T. Alazard, Low Mach number limit of the full Navier-Stokes equations, {\it Arch. Ration. Mech. Anal.} {\bf 180 (1)} (2006)  1--73.

\bibitem{almgren-06}
A. S. Almgren, J. B. Bell, C. A. Rendleman and M. Zingale, Low Mach number modeling of type ia supernovae. I. Hydrodynamics, \emph{Astrophys. J.} \textbf{637} (2006) 922–936.  


\bibitem{asa}
K. Asano, On the incompressible limit of the compressible Euler equations, {\it Japan J. Appl. Math.} {\bf 4 (3)} (1987)  455--488. 
	
\bibitem{BCD-Book-2011} H. Bahouri, J. Chemin, R. Danchin,
{\it Fourier Analysis and Nonlinear Partial Differential Equations,}
Grundlehren Math. Wiss., {\bf 343},
Springer, Heidelberg, 2011.

\bibitem{BV-2022} J. Bedrossian, V. Vicol,
{\it  The Mathematical Analysis of the Incompressible Euler and Navier-Stokes Equations --- An Introduction,}
 American Mathematical Society, Providence, RI, 2022.  

\bibitem{BS-2009-Adv} C. Bjorland,  M. E. Schonbek,  
Poincar\'{e}'s inequality and diffusive evolution equations,
{\it Adv. Differential Equations} {\bf 14} (2009)  241--260.



\bibitem{Bl-2016-SIMA} L. Brandolese,  
Characterization of solutions to dissipative systems with sharp algebraic decay,
{\it SIAM J. Math. Anal.} {\bf 48 (3)} (2016)  1616--1633.

\bibitem{BSXZ-Adv-2024} L. Brandolese,  L.-Y. Shou,  J. Xu,  P. Zhang, 
Sharp decay characterization for the compressible Navier-Stokes equations,
{\it Adv. Math.} {\bf 456} (2024) Paper No. 109905.

\bibitem{BDGL-SAP-2022}
D. Bresch, B.  Desjardins, E. Grenier, C.-K. Lin, Low Mach number limit of viscous polytropic flows: formal asymptotics in the periodic case, {\it Stud. Appl. Math.} {\bf 109 (2)}  (2002)  125--149.

\bibitem{CD-ARMA-2010} F. Charve,  R. Danchin,  
A global existence result for the compressible Navier-Stokes equations in the critical $L^p$ framework,
{\it Arch. Ration. Mech. Anal.} {\bf 198 (1)} (2010)   233--271.

\bibitem{C-N-1995} J.-Y. Chemin,   N. Lerner,  
Flot de champs de vecteurs non lipschitziens et \'{e}quations de Navier-Stokes,  
{\it J. Differential Equations} {\bf  121} (1995) 314--328.

\bibitem{CMZ-CPAM-2010} Q. Chen, C. Miao,  Z. Zhang,  
Global well-posedness for compressible Navier-Stokes equations with highly oscillating initial velocity,
{\it Comm. Pure Appl. Math.} {\bf 63 (9)} (2010) 1173--1224.

\bibitem{CTWZ-JDE-2020} Q. Chen,  Z. Tan,  G. Wu,  W. Zou,  
The initial value problem for the compressible Navier-Stokes equations without heat conductivity,
{\it J. Differential Equations} {\bf 268 (9)} (2020)  5469--5490.

%\bibitem{CKLSW-2019-MCWF} A. Chertock, A. Kurganov,  M. Lukáčová-%Medvid'ová,  P. Spichtinger, B. Wiebe, 
%Stochastic Galerkin method for cloud simulation,
%{\it Math. Clim. Weather Forecast.} {\bf 5 (1) (2019)}  65--106. 



\bibitem{Danchin-00-IM} R. Danchin,  
Global existence in critical spaces for compressible Navier-Stokes equations,
{\it Invent. Math.} {\bf 141 (3)} (2000)  579--614.

\bibitem{Danchin-01-CPDE} R. Danchin,  
Local theory in critical spaces for compressible viscous and heat-conductive gases,
{\it Comm. Partial Differential Equations} {\bf 26} (2001) 1183--1233.


\bibitem{Dr-2002-AJM}
 R. Danchin,  
Zero Mach number limit for compressible flows with periodic boundary conditions,
{\it Amer. J. Math.} {\bf 124 (6)} (2002) 1153--1219.

\bibitem{Danchin-2002} R. Danchin, 
Zero Mach number limit in critical spaces for compressible Navier-Stokes equations,
{\it Ann. Sci. \'{E}cole Norm. Sup. (4)} {\bf 35 (1)} (2002) 27--75.

%\bibitem{Danchin-2007-CPDE} R. Danchin,  
%Well-posedness in critical spaces for barotropic viscous fluids with truly not constant density,
%{\it Comm. Partial Differential Equations} {\bf 32} (2007)   1373--1397.

\bibitem{Danchin-2014-ANF} R. Danchin,  
A Lagrangian approach for the compressible Navier-Stokes equations,
{\it Ann. Inst. Fourier (Grenoble)} {\bf 64 (2)} (2014) 753--791.

 \bibitem{Danchin-2018} R.  Danchin,  
Fourier analysis methods for the compressible Navier-Stokes equations, {\it in: Handbook of Mathematical Analysis in Mechanics of Viscous Fluids}, 1843--1903, Y. Giga  and  A. Novotn\'{y} (eds),
Springer, International Publishing Switzerland, 2018.


\bibitem{Dr-2024-PMP}  R. Danchin, 
Global solutions for two-dimensional viscous pressureless flows with large variations of density,
{\it Probab. Math. Phys.} {\bf 5 (1)} (2024)  55--88.  

 
 


\bibitem{DH-MATHANN-2016} R. Danchin,  L. He,  
The incompressible limit in $L^p$ type critical spaces,
{\it Math. Ann.} {\bf 366  } (2016) 1365--1402.


\bibitem{DX-ARMA-2017} R. Danchin,  J. Xu, 
Optimal time-decay estimates for the compressible Navier-Stokes equations in the critical $L^p$ framework,
{\it Arch. Ration. Mech. Anal.} {\bf 224 (1)} (2017)  53--90.

\bibitem{Ebin}
D. G. Ebin, Motion of a slightly compressible fluid,  {\it Proc. Natl. Acad. Sci. USA,} {\bf 72} (1975) 539--542.

\bibitem{fanhuang}
J. Fan, X. Huang, Global strong solutions and asymptotic behavior for arbitrarily large initial data of the 2D compressible Navier-Stokes equations with transport entropy, arXiv:2512.18976, 2025.


\bibitem{FN-ARMA-2007} E. Feireisl,  A. Novotn\'{y},
The low Mach number limit for the full Navier-Stokes-Fourier
system, {\it Arch. Ration. Mech. Anal.} {\bf 186 (1)} (2007)  77--107.

\bibitem{FN-CMP-2013} E. Feireisl,  A. Novotn\'{y},
Inviscid incompressible limits of the full Navier-Stokes-Fourier system, {\it Comm. Math. Phys.} {\bf 321 (3)} (2013)  605--628.

\bibitem{FKNZ-M3AS-2016} E. Feireisl, R. Klein,  A. Novotn\'{y},  E. Zatorska, 
On singular limits arising in the scale analysis of stratified fluid flows,
{\it Math. Models Methods Appl. Sci.} {\bf 26 (3)}  (2016)  419--443.

\bibitem{FM-2024-MA} M. Fujii,  
Low Mach number limit of the global solution to the compressible Navier-Stokes system for large data in the critical Besov space,
{\it Math. Ann.} {\bf 388 (4)} (2024)  4083--4134.

\bibitem{FK-1964-ARMA} H. Fujita, T. Kato,  
On the Navier-Stokes initial value problem. I.
{\it Arch. Rational Mech. Anal.} {\bf 16} (1964)  269--315.

\bibitem{GW-CPDE-2012} Y. Guo,  Y. Wang,  
Decay of dissipative equations and negative Sobolev spaces,
{\it Comm. Partial Differential Equations} {\bf 37 (12)} (2012)  2165--2208.

\bibitem{Haspot-2011-JDE} B. Haspot,  
Well-posedness in critical spaces for the system of compressible Navier-Stokes in larger spaces,
{\it J. Differential Equations} {\bf 251 (8)} (2011) 2262--2295.


\bibitem{HZ-IUMJ-1995} D. Hoff, K. Zumbrun,
Multi-dimensional diffusion waves for the Navier-Stokes equations of compressible flow,
{\it Indiana Univ. Math. J.} {\bf 44 (2)} (1995)  603--676.


%\bibitem{HW-SIMA-2013}  X. Hu,  G. Wu,  
%Global existence and optimal decay rates for three-dimensional compressible viscoelastic flows,
%{\it SIAM J. Math. Anal.} {\bf 45 (5)} (2013)  2815--2833.

%\bibitem{ig}
%T. Iguchi, The incompressible limit and the initial layer of the compressible Euler equation in $\textrm{R}_{+}^{\textrm{n}}$, {\it Math. Methods Appl. Sci.}  {\bf 20} (1997)  945--958. 
		
\bibitem{isa}
H. Isozaki, 
Singular limits for the compressible Euler equation in an exterior domain, 
{\it J. Reine Angew. Math.} {\bf 381} (1987)  1--36. 
		
%\bibitem{isb}
%H. Isozaki, Singular limits for the compressible Euler equation in an exterior domain, II Bodies in a uniform flow, {\it Osaka J. Math.} {\bf 26} (1989)  399--410. 
		
\bibitem{jiangouJMPA}
S. Jiang, Y. Ou,  
Incompressible limit of the non-isentropic Navier-Stokes equations with well-prepared initial data in three-dimensional bounded domains,
{\it J. Math. Pures Appl. (9)} {\bf 96 (1)} (2011)  1--28.


%\bibitem{juouJMPA}
%Q. C. Ju and Y. B. Ou, Low Mach number limit of Navier-Stokes equations with large temperature variations in bounded domains, {\it J. Math. Pures Appl.}  {\bf 164} (2022)  131--157.

\bibitem{KM-2010}
S. Klainerman, A. Majda, Singular limits of quasilinear hyperbolic systems with large parameters and the incompressible limit of compressible fluids, {\it Comm. Pure Appl. Math.} {\bf 34 (4)} (1981)  481--524.



\bibitem{KM-1982}
S. Klainerman, A. Majda, Compressible and incompressible fluids, {\it Comm. Pure Appl. Math.}  {\bf 35 (5)} (1982)  629--651.

    
			\bibitem{klein}
R. Klein, An applied mathematical view of meteorological modeling, in  \emph{Applied Mathematics Entering the 21st Century}, 227--269, SIAM, Philadelphia, PA,  2004. 
%
% \bibitem{Klein-2004}  R. Klein, {\it An Applied Mathematical View of Meteorological Modelling},   Applied Mathematics Entering the 21st
%Century, SIAM, Philadelphia, PA, 2004.




\bibitem{LNZ-2025-preprint} F. Li, J. Ni, Z. Zhang,
Uniform stability and optimal time decay rate  of the compressible pressureless Navier-Stokes system in  the critical regularity framework, arXiv:2511.02321, 2025.

\bibitem{LS-SIMA-2023} H.-L. Li,  L.-Y. Shou, 
Global existence and optimal time-decay rates of the compressible Navier-Stokes-Euler system,
{\it SIAM J. Math. Anal.} {\bf 55 (3)} (2023)  1810--1846.

\bibitem{LSZ-2025-arXiv} H.-L. Li, L.-Y. Shou, Y. Zhang,  
Large-friction and incompressible limits for pressureless Euler/isentropic Navier-Stokes flows, arXiv:2508.20730, 2025.


\bibitem{LZ-M2AS-2011} H.-L. Li,  T. Zhang,  
Large time behavior of isentropic compressible Navier-Stokes system in $\mathbb R^3$.
{\it Math. Methods Appl. Sci.} {\bf 34 (6)} (2011)  670--682.

\bibitem{lions1998}
 P.-L. Lions, {\it Mathematical Topics in Fluid Mechanics:    Compressible Models}, Oxford Science
 Publication, Oxford, 1998.
 
\bibitem{PM-JMPA-1998} P.-L. Lions,  N. Masmoudi,  
Incompressible limit for a viscous compressible fluid,
{\it J. Math. Pures Appl. (9)} {\bf 77(6)} (1998)  585--627.

\bibitem{LM-CRASPSM-1999} P.-L. Lions, N. Masmoudi,  
Une approche locale de la limite incompressible,
{\it C. R. Acad. Sci. Paris S\'{e}r. I Math.} {\bf 329 (5)} (1999)  387--392.


\bibitem{LS-JMFM-2022} M. Luk\'{a}čov\'{a}-Medvid'ov\'{a}, A. Sch\"{o}mer,  
 Existence of Dissipative Solutions to the Compressible Navier-Stokes System with Potential Temperature Transport,
{\it J. Math. Fluid Mech.} {\bf 24 (3)} (2022) Paper No. 82.

\bibitem{LS-JMFM-2023} M. Luk\'{a}čov\'{a}-Medvid'ov\'{a}, A. Sch\"{o}mer,    
Compressible Navier-Stokes equations with potential temperature transport: stability of the strong solution and numerical error estimates,
{\it J. Math. Fluid Mech.} {\bf 25 (1)} (2023) Paper No. 1.

\bibitem{LS-JDE-2025} M. Luk\'{a}čov\'{a}-Medvid'ov\'{a}, A. Sch\"{o}mer,  
Conditional regularity for the compressible Navier-Stokes equations with potential temperature transport,
{\it J. Differential Equations} {\bf 423} (2025) 1--40.

\bibitem{M-JDE-2016}
 D. Maltese, M. Michálek, P.B. Mucha, A. Novotn\'{y}, M. Pokorn\'{y}, E. Zatorska, Existence of weak solutions
 for compressible Navier-Stokes equations with entropy transport, {\it J. Differential Equations} {\bf 261} (2016)  4448--4485.

\bibitem{MN-PJASAMS-1979} A. Matsumura,  T. Nishida,  
The initial value problem for the equations of motion of compressible viscous and heat-conductive fluids,
{\it  Proc. Japan Acad. Ser. A Math. Sci.} {\bf 55 (9)} (1979)  337--342.

\bibitem{MN-JMKU-1980} A. Matsumura,  T. Nishida,  
The initial value problem for the equations of motion of viscous and heat-conductive gases,
{\it J. Math. Kyoto Univ.} {\bf 20 (1)} (1980)  67--104.

\bibitem{MetivierSchochet2001}
G. M\'{e}tivier, S. Schochet, The incompressible limit of the non-isentropic euler equations, {\it Arch. Ration. Mech. Anal.} {\bf 158 (1)} (2001)  61--90.

\bibitem{MM-JMFM-2015}
M. Mich\'{a}lek,  
Stability result for Navier-Stokes equations with entropy transport,
{\it J. Math. Fluid Mech.} {\bf 17 (2)} (2015) 279--285.

 

\bibitem{NS-2015-JLMS} C. J. Niche,  M. E. Schonbek, 
Decay characterization of solutions to dissipative equations,
{\it J. Lond. Math. Soc. (2)} {\bf 91 (2)} (2015)  573--595.



%\bibitem{Se-00} P.   Secchi, 
%On the singular incompressible limit of inviscid compressible fluids. {\it J. Math. Fluid Mech.} {\bf 2}(2000)  107--125. 

\bibitem{SK-1985-HMJ} Y. Shizuta, S. Kawashima, 
Systems of equations of hyperbolic-parabolic type with applications to the discrete Boltzmann equation,
{\it Hokkaido Math. J.} {\bf 14 (2)} (1985)  249--275.

\bibitem{Ukai1986}
S. Ukai, 
The incompressible limit and the initial layer of the compressible Euler equation,
{\it J. Math. Kyoto Univ.} {\bf 26 (2)} (1986)  323--331.
       
\bibitem{vk}
V. A. Vaigant, A. V. Kazhikhov,
On the existence of global solutions of two-dimensional Navier-Stokes equations of a compressible viscous fluid. Siberian Math. J. \textbf{36 (2)} (1995)   1108--1141.

     
        
\bibitem{XX-JDE-2021} Z. Xin, J. Xu,  
Optimal decay for the compressible Navier-Stokes equations without additional smallness assumptions,
{\it J. Differential Equations} {\bf 274} (2021) 543--575.


\bibitem{Xj-CMP-2019} J. Xu, 
A low-frequency assumption for optimal time-decay estimates to the compressible Navier-Stokes equations,
{\it Comm. Math. Phys.} {\bf 371 (2)} (2019)  525--560.

\bibitem{ZLZ-CMS-2023} X. Zhai,  Y. Li, F. Zhou,  
Global strong solutions to the compressible Navier-Stokes system with potential temperature transport,
{\it Commun. Math. Sci.} {\bf 21 (8)} (2023)  2247--2260.
 


\end{thebibliography}

\end{document}